\documentclass[reqno, 11pt, oneside]{amsart}

\usepackage[margin=2cm, a4paper]{geometry}
\usepackage{amsmath, mathrsfs, url, color}
\usepackage{enumitem}
\usepackage{bbm}
\usepackage{graphicx}
\usepackage{subcaption}
\allowdisplaybreaks
\usepackage{cancel}

\newtheorem{thm}{Theorem}[section]

\newtheorem{cor}{Corollary}[section]
\newtheorem{lem}{Lemma}[section]
\newtheorem{prop}{Proposition}[section]
\newtheorem{rem}{Remark}[section]

\newtheorem{asA}{$A$--\hspace{-1mm}}
\newtheorem{asAt}{$\bar{A}$--\hspace{-1mm}}
\newtheorem{asB}{B--\hspace{-1mm}}
\newtheorem{asAtt}{$\bar{\bar{A}}$--\hspace{-1mm}}

\usepackage[english]{babel}

\usepackage[nodayofweek]{datetime}

\newcommand{\bE}{\mathbb{E}}

\newcommand{\bR}{\mathbb{R}}

\title[]{Well-posedness and tamed Euler schemes for McKean--Vlasov equations driven by L\'evy noise}

\date{
       \currenttime,
       \today, 
			\quad {(File \tt \jobname.tex})
}

\author[]{Neelima,  Sani Biswas,  Chaman Kumar, Gon\c calo dos Reis and Christoph Reisinger}

\thanks{G.~dos Reis acknowledges support from the \emph{Funda{\c c}$\tilde{\text{a}}$o para a Ci$\hat{e}$ncia e a Tecnologia} (Portuguese Foundation for Science and Technology) through the project UIDB/00297/2020 (Centro de Matem\'atica e Aplica\c c$\tilde{\text{o}}$es CMA/FCT/UNL)}

\begin{document}
\selectlanguage{english}

\begin{abstract}
We prove the well-posedness of solutions 
to McKean--Vlasov stochastic differential equations driven by L\'evy noise under mild assumptions where, in particular, the L\'evy measure is not required to be finite. 
The drift, diffusion and jump coefficients are allowed to be random, can grow super-linearly in the state variable, and all may depend on the marginal law of the solution process. We provide a propagation of chaos result under more
relaxed conditions than those existing in the literature, and consistent with our well-posedness result.

We propose a tamed Euler scheme for the associated interacting particle system 
and prove that the rate of its strong convergence is arbitrarily close to $1/2$.
As a by-product, we also obtain the corresponding results on well-posedness, propagation of chaos  and strong convergence of the tamed Euler scheme for McKean--Vlasov stochastic delay differential equations (SDDE) and McKean--Vlasov stochastic differential equations with Markovian switching (SDEwMS), both driven by L\'evy noise. Furthermore, our results on tamed Euler schemes are new even for ordinary SDEs driven by L\'evy noise and with super-linearly growing coefficients. 
\end{abstract}
\maketitle

\section{Introduction}
Consider a filtered probability space $(\Omega, \mathscr{F}, \{\mathscr{F}_t\}_{\{t\geq 0\}}, \mathbb{P})$  satisfying the usual conditions, \textit{i.e.}, $\mathscr{F}_0$ contains all $\mathbb{P}$-null sets and the filtration $\{\mathscr{F}_t\}_{\{t\geq 0\}}$ is right continuous; $\bE$ stands for the usual expectation operator. 
Assume that $w:=\{w_t\}_{\{t\geq 0\}}$ is an $m$-dimensional Wiener process. 
Let $n_p(dt,dz)$ be  a Poisson random measure on a  $\sigma$-finite measure space $(Z, \mathscr{Z}, \nu)$   and $\tilde{n}_p(dt,dz):=n_p(dt,dz)-\nu(dz)dt$ be its compensator.
Let $\mathcal{P}_2(\mathbb{R}^d)$ be the space of probability measures $\mu$ defined  on a measurable space $(\mathbb{R}^d, \mathscr{B}(\mathbb{R}^d))$ such that $\int_{\mathbb{R}^d}|x|^2 \mu(dx)<\infty$ and equip this space with the $\mathcal{L}_2$-Wasserstein  metric 
\begin{align*}
\mathcal{W}_2(\mu_1,\mu_2):= \inf_{\pi\in\Pi(\mu_1, \mu_2)} \Big(\int_{\mathbb{R}^d\times\mathbb{R}^d} |x-y|^2\pi(dx,dy)\Big)^{1/2}, 
\end{align*}
where  $\Pi(\mu_1,\mu_2)$ denotes the set of couplings of $\mu_1,\mu_2\in\mathcal{P}_2(\mathbb{R}^d)$. 
It is well known  that $\mathcal{P}_2(\mathbb{R}^d)$ equipped with above metric is a Polish space. 
Denote by $\mathscr{P}$ the predictable $\sigma-$algebra on $\Omega\times \mathbb{R}_+$.
Let $b$ 
and $\sigma$ 
be $\mathscr{P}\otimes\mathscr{B}(\mathbb{R}^d)\otimes\mathscr{B}(\mathcal{P}_2(\mathbb{R}^d))$-measurable functions taking values in $\mathbb{R}^d$ and $\mathbb{R}^{d\times m}$ respectively.
Also, assume that $\gamma$ 
is a $\mathscr{P}\otimes\mathscr{B}(\mathbb{R}^d)\otimes\mathscr{B}(\mathcal{P}_2(\mathbb{R}^d))\otimes \mathscr{Z}$-measurable function with values in $\mathbb{R}^d$. 

Let $T>0$ be a fixed constant. Consider the following McKean--Vlasov stochastic differential equation (SDE), 
\begin{equation}\label{eq:sde}
    x_t = x_{0} + \int_{0}^t b_s(x_s, L_{x_s}) ds + \int_{0}^t \sigma_s(x_s, L_{x_s}) dw_s + \int_{0}^t \int_{Z} \gamma_s(x_{s}, L_{x_s}, z) \tilde{n}_p(ds,dz)
\end{equation}
almost surely for any $t \in[0,T]$, where  $x_{0}$ is an $\mathbb{R}^d$-valued $\mathscr{F}_{0}$-measurable  random variable and $\{L_{x_t} \}_{0 \leq t \leq T}$ denotes the flow of deterministic marginal distributions of $x$. 
 Equations of this type were first studied in \cite{mckean1966} and 
  differ from ordinary SDEs in the sense that the coefficients additionally depend on the marginal laws $L_{x_t}$ of the solution process $x$, \textit{i.e.} the McKean--Vlasov SDE becomes an ordinary SDE when the law $L_{x_t}$ is known.
 
McKean--Vlasov SDEs have received renewed attention over the past years mainly due to the crucial unravelling of certain technical bottlenecks but also due to their usefulness in modelling across many fields such as statistical physics, neuroscience, and finance among others. At the heart of this is that McKean--Vlasov SDEs describe a limiting behaviour of individual particles interacting with each other in a ``mean-field'' manner -- the so-called Propagation of Chaos result. This latter result also paves the way (there are others) for the numerical solution of such McKean--Vlasov SDEs via approximation of the associated system of interacting particles. 

 The numerical approximation and simulation of ordinary SDEs with coefficients which are Lipschitz functions of the state-variables is classical  (see, e.g., \cite{kloeden1992} for time stepping schemes and
 \cite{giles2018} for a survey of multilevel simulation methods).
 The case of locally Lipschitz-continuous coefficients with superlinear growth is also now
 well-studied in the literature.
 Early works include 
\cite{hutzenthaler2012, sabanis2013}, who proposed and analysed tamed Euler schemes for locally Lipschitz drift coefficients of potentially superlinear growth, which was extended in
 \cite{hutzenthaler2015, sabanis2016, beyn2017, kumar2019} to diffusion coefficients of superlinear growth, and in
 \cite{kumar2014} to include delay.
 The authors of
 \cite{gan2013} studied tamed Milstein schemes for superlinear drift, while the results in
 \cite{kumar2017a} allow L\'evy noise and superlinear drift coefficients,  extended to superlinear diffusion coefficients for Euler and Milstein schemes in \cite{kumar2017b} and \cite{kumar2020a}, respectively, while \cite{dareiotis2016} considered a tamed Euler scheme for L\'evy noise and delay, and
\cite{kumar2020b} a tamed Milstein scheme for Markovian switching.
Alternatives to tamed schemes include balanced and implicit schemes (see, e.g., \cite{tretyakov2013}), truncated schemes (see \cite{mao2015}) and adaptive time stepping schemes (see \cite{kelly2019, fang2020}).
 
 The numerical approximation of McKean--Vlasov equations in the continuous case (\textit{i.e.}, when $\gamma\equiv 0$) was
 initiated  in \cite{bossy1997}
 and has been
 investigated further in a number of more recent works. 
 We discuss in the following those most relevant to the present paper.

The strong convergence of a tamed Euler scheme and an implicit scheme was established in \cite{reis2019a} when the drift coefficient is allowed to grow super-linearly in the state variable and linearly in the measure component, whereas the diffusion coefficient satisfies Lipschitz continuity in both the variables. 
 In \cite{kumar2020c,bao2020a}, the authors independently proposed a tamed Milstein scheme for McKean--Vlasov SDEs where the drift coefficient grows super-linearly and established that its strong rate of convergence is equal to $1$.
An extension to point delays and an efficient antithetic multilevel Monte Carlo sampling method were given in \cite{bao2020b}. 
Variance reduction via importance sampling techniques was addressed in \cite{reis2019importancesampling}. 
Subsequently, the authors of \cite{kumar2020d} proved well-posedness of McKean--Vlasov SDEs with common noise when both drift and diffusion coefficients are allowed to grow super-linearly in the state variable. Furthermore, they investigated the strong convergence of tamed Euler and Milstein schemes, and proved that their rate of convergence is $1/2$ and $1$, respectively.
A possible alternative to tamed schemes are again adaptive schemes, as shown in \cite{reis2020} for superlinear drift and diffusion coefficients.

In this work, we address the numerical approximation of McKean--Vlasov SDEs driven by a general L\'evy noise. 
For preceding work, we refer to \cite{agarwal2020fourier}, where the authors develop a Fourier based numerical approximation for McKean--Vlasov SDEs with L\'evy noise that has no law dependency and a drift  coefficient that is affine in the state variable, plus a linear mean-field interaction term (an expectation of the state variable); their approach avoids altogether the Propagation of Chaos we employ here.

Here, we cover the full spectrum of results for the class \eqref{eq:sde}: well-posedness of the McKean--Vlasov equation and the associated interacting particle system, propagation of chaos, numerical strong approximations and applications. Our results cater to the several fields and applications where McKean--Vlasov SDEs with a general L\'evy noise appear. 
To the best of our knowledge, this is the first paper under these general conditions:
\begin{enumerate}
\item
we  allow all the coefficients to depend on the marginal law $L_{x_t}$ of the solution process $x$, 
\item
they all (including the jump coefficient) grow super-linearly in the state component with
 \item
 a general $\mathcal{W}_2$-Lipschitz (non-linear) measure dependency, and
\item 
the L\'evy measure is not assumed to be finite.
 \end{enumerate}
We review our contributions against existing literature next.
 \begin{itemize}
 \item \textit{Well-posedness}. Our first result is on the existence and uniqueness of the McKean--Vlasov SDE \eqref{eq:sde} under more relaxed assumptions (see Theorem \ref{thm:eu}) than those made in the existing literature, \textit{e.g.} \cite{graham1992mckean, jourdain2008, mehri2020} and references therein. 
More specifically, well-posedness of the equation \eqref{eq:sde} driven by a pure, square integrable L\'evy process (i.e., $b=\sigma=0$) is shown in \cite{jourdain2008} under global Lipschitz conditions on the coefficient $\gamma$ in both state and measure variables. Existence only is shown when the integrability assumptions are dropped, using PDE arguments.
The earlier work \cite{graham1992mckean}, extended by \cite{jourdain2008}, focused on Poisson jumps, Lipschitz coefficients and linear measure dependencies.  More recently, in \cite{mehri2020}, the authors allow super-linearity  of the coefficients in the state variable, but  consider only a specific, linear, measure dependence, with time homogeneous Poisson measures. In this paper, we allow super-linear growth in the state variable and a general measure dependence of all (drift, diffusion and jump) coefficients.  

Our proof uses a different strategy from that given in \cite{mehri2020}. 
We apply a fixed point argument on the space $\mathbb{D}([0,T];  \mathcal{P}_2(\mathbb{R}^d))$ of $\mathcal{P}_2(\mathbb{R}^d)$-valued c\`adl\`ag functions defined on $[0,T]$,  which is a complete metric space under the uniform norm. For this, the flow of marginals $\{L_{x_t}\}_{\{t\in[0,T]\}}$ is shown to be c\`adl\`ag functions.  
This only requires finiteness of $\bE|x_0|^2$  instead of $\bE|x_0|^{2+\epsilon}$ for any $\epsilon>0$, compared to the analogous results in \cite{kumar2020d} for the continuous case. 

\item \textit{Propagation of Chaos} (PoC). For the approach we use in this work, there are typically two stages to the numerical approximation of the McKean--Vlasov SDE \eqref{eq:sde}, as the coefficients depend on all paths (through the measure) of the solution process. 
In the first stage of the discretization, an interacting particle system connected with the McKean--Vlasov SDE \eqref{eq:sde} is shown to converge to the true solution of the McKean--Vlasov SDE \eqref{eq:sde}. 
This is popularly known as the \textit{propagation of chaos}. We present such a result in Section \ref{sec:PoCmain} under more relaxed conditions than those existing in the literature; see above references to \cite{graham1992mckean,jourdain2008} as works studying propagation of chaos for jump processes. 
As reported in \cite{mehri2020}, the only existing result (then) on McKean--Vlasov limits of biological neural networks driven by L\'evy noise was \cite{andreis2018interacting} for local dynamics under global Lipschitz assumptions or local dynamics of gradient type. Our setting allows for a more general measure dependency than \cite{mehri2020}, and further, no numerical approximation procedure is offered there.

The PoC result is of general interest to researchers of mean-field games with jumps (see, e.g., \cite{benazzoli2017mean,andreis2018interacting}), especially in neuroscience, where 
the need for L\'evy noise in McKean--Vlasov SDE models has recently been appreciated (see, e.g., \cite{delarue2015global,ullner2018,mehri2020}).

\item 
\textit{Tamed Euler scheme}. Following the particle approximation, the second stage of the discretization is a time-stepping scheme 
for the resulting coupled SDE system.
In this paper, we specifically propose a tamed Euler scheme 
and prove that its strong rate of convergence (in mean-square sense) is arbitrarily close to $1/2$, 
without assuming linear growth in the (\textit{viz.} drift, diffusion and jump) coefficients or finiteness of the L\'evy measure.
To achieve this, we identify a suitable coercivity assumption on the coefficients of equation \eqref{eq:sde} and tamed scheme \eqref{eq:scheme}  (see Assumptions $A$--\ref{as:coercivity}, $A$--\ref{as:coercivity:p0} and  Assumption $B$--\ref{asb:growth:n}, respectively) so that the one-step error has a desired rate (see Lemma \ref{lem:one-step:mb}), which is required at several points of the analysis. 
These assumptions  appear for the first time in the literature and are possibly  optimal (even when the coefficients do not depend on the measure). 
In this generality,  we had to develop new techniques to bound the moments of the tamed Euler scheme.  
Due to the  presence of jumps, the rate of the one-step error in the $\mathcal{L}^q$-norm drops to $1/q$ and further decreases if one applies taming (\textit{i.e.}, Assumption $B$--\ref{asb:growth:n}) to super-linearly growing jump coefficients; see Lemma \ref{lem:one-step:mb}. In the proof of Lemma \ref{lem:scheme:mb}, the positive powers of $n$ coming from Assumption $B$--\ref{asb:growth:n} must be controlled through the rate observed in Lemma \ref{lem:one-step:mb}. 
But the rate in Lemma \ref{lem:one-step:mb} itself decreases if one takes higher  powers of $n$ in Assumption $B$--\ref{asb:growth:n}.
 Further, if we decrease the powers of $n$ in Assumption $B$--\ref{asb:growth:n}, then the moment  (\textit{i.e.} the value of $p_0$ in Assumption $A$--\ref{as:initial}) needed for Assumption $B$--\ref{asb:diff:rate} also increases, which can lead to exclusion of several examples of practical importance. 
 This exhibits an inter-play between Assumption $B$--\ref{asb:growth:n}, Assumption  $B$--\ref{asb:diff:rate} and the corresponding one-step error that one obtains in Lemma \ref{lem:one-step:mb}, and which is tackled by identifying an appropriate taming, without assuming finiteness of the L\'evy measure.
We give examples of explicit forms of taming in Section 5 which fit into our framework.   
 
In \cite{kumar2017b}, the authors studied a tamed Euler scheme for an ordinary SDE driven by L\'evy noise when both drift and diffusion coefficients are allowed to grow super-linearly, but the jump coefficient grows only linearly, while
in \cite{chen2019meansquare}, the authors also allow the jump coefficients to grow super-linearly, but assume that the L\'evy measure is finite, which restricts the class of L\'evy processes severely. 

To the best of our knowledge, this is the first paper under these general conditions. 
\item 
\textit{Applications}.
Throughout this paper, we assume that the coefficients of McKean--Vlasov SDE \eqref{eq:sde} are random and hence all the results obtained here can be applied to McKean--Vlasov SDEs with Markovian switching (SDEwMS) and McKean--Vlasov SDEs with delays (SDDE), both driven by L\'evy noise.  We motivate and discuss these cases in Section 5. 
Our proof uses a different strategy from that given in [38].
\item Finally, the technique developed in this paper can be extended to higher order schemes, which is ongoing work. 
 \end{itemize}

\subsection*{Notations}
The notation $|\cdot|$ is used for both Euclidean and Hilbert-Schmidt norms, which should not cause any confusion in the reader's mind, while
$xy$ denotes the inner product of $x,y\in\mathbb{R}^d$. We use $\lfloor \cdot\rfloor$ to denote the floor function and $\delta_x$ for the Dirac measure centred at $x\in\mathbb{R}^d$.
Also, $\mathscr{B}(A)$ is the Borel $\sigma$-algebra of a topological space $A$. 
For brevity of space, we use  $x$ for the solution of  the McKean--Vlasov SDE as well as for a point in $\mathbb{R}^d$. This is clear from the context and should not cause confusion.  
Further, $\sigma^*$ stands for transpose of the matrix $\sigma$. 
Also, $\mathbb{N}$ and $\bR$ denote the set of natural and real numbers respectively. 
Lastly, $K$ stands for a generic positive constant and can vary from place to place.

\section{Existence, Uniqueness and Moment Bound}
In this section, we work under the following set of assumptions. Let $p_0\geq 2$ be a fixed constant. 
\begin{asA} \label{as:initial}
$\mathbb{E}|x_{0}|^{p_0} < \infty$. 
\end{asA}
\begin{asA} \label{as:coercivity}
There exist a constant $L>0$ and a sequence $\{M_t\}_{t \in [0,T]}$ of $\mathscr{F}_0$-measurable non-negative random variables satisfying $\sup_{t \in [0,T]}\mathbb{E}M_t<\infty$  such that
\begin{align*}
2  xb_t(x, \mu) +  |\sigma_t(x,\mu)|^2 +   \int_Z |\gamma_t(x,\mu,z)|^2 \nu(dz) \leq L\big\{M_t+|x|^{2}+\mathcal{W}_2^{2}(\mu,\delta_0) \big\}
\end{align*}
almost surely for any $t\in[0,T]$, $x\in\mathbb{R}^d$  and $\mu \in \mathcal{P}_2(\mathbb{R}^d)$. 
\end{asA}
\begin{asA} \label{as:monotonicity}
There exists a constant $L>0$ such that 
\begin{align*}
2(x-\bar{x})(b_t(x, \mu)-b_t(\bar{x}, \bar{\mu})) & + |\sigma_t(x,\mu)-\sigma_t(\bar{x},\bar{\mu})|^2 +\int_Z |\gamma_t(x,\mu,z)-\gamma_t(\bar{x},\bar{\mu},z)|^2 \nu(dz) 
\\
& \leq L\big\{|x-\bar{x}|^2 + \mathcal{W}_2^2(\mu,\bar{\mu}) \big\}
\end{align*}
almost surely for any $t\in[0,T]$, $x, \bar{x}\in\mathbb{R}^d$ and $\mu, \bar{\mu} \in \mathcal{P}_2(\mathbb{R}^d)$. 
\end{asA}
\begin{asA} \label{as:continuity}											
For any $t\in[0,T]$, 
$b_t(x,\mu)$ is almost surely a continuous function of $x\in\mathbb{R}^d$ and $\mu\in\mathcal{P}_2(\mathbb{R}^d)$.

\end{asA}
\begin{asA} \label{as:coercivity:p0}
There exist a constant $L>0$ and a sequence $\{\bar{M}_t\}_{t \in [0, T]}$ of $\mathscr{F}_0$-measurable non-negative random variables satisfying $\sup_{ t \in [0, T]}\mathbb{E}\bar{M}_t<\infty$ such that 
\begin{align*}
2|x&|^{p_0-2}  xb_t(x, \mu) +(p_0-1) |x|^{p_0-2} |\sigma_t(x,\mu)|^2 
\\
& + 2 (p_0-1) \int_Z |\gamma_t(x,\mu,z)|^2 \int_0^1 (1-\theta)| x+\theta \gamma_t(x,\mu,z)|^{p_0-2} d\theta \nu(dz) \leq L\big\{\bar{M}_t+|x|^{p_0}+\mathcal{W}_2^{p_0}(\mu,\delta_0) \big\}
\end{align*}
almost surely for any $t\in[0,T]$, $x\in\mathbb{R}^d$ and $\mu \in \mathcal{P}_2(\mathbb{R}^d)$. 
\end{asA}
\begin{thm} \label{thm:eu}
Let Assumptions $A$--\ref{as:initial} (with $p_0=2$), $A$--\ref{as:coercivity},  $A$--\ref{as:monotonicity} and A-\ref{as:continuity} be satisfied. Then, there exists a unique c\`adl\`ag process $x$ taking values in $\mathbb{R}^d$ satisfying the  McKean--Vlasov SDE \eqref{eq:sde} such that
\begin{align*}
\sup_{t\in[0,T]}\mathbb{E}|x_t|^{2} \leq K,
\end{align*}
where $K:=K(\mathbb{E}|x_0|^{2},m,d,L, T)$ is a positive constant. Moreover, if Assumptions $A$--\ref{as:initial} and $A$--\ref{as:coercivity:p0} hold for  $p_0 > 2$, then 
\begin{align*}
\sup_{t\in[0,T]}\mathbb{E}|x_t|^{p_0} \leq K,
\end{align*}
where $K:=K(\mathbb{E}|x_0|^{p_0},m,d,p_0,L, T)$ is a positive constant. 
\end{thm} 
\begin{proof}
Consider the space $\mathbb{D}([0,T];  \mathcal{P}_2(\mathbb{R}^d))$ of $\mathcal{P}_2(\mathbb{R}^d)$-valued c\`adl\`ag functions defined on $[0,T]$. For any $\mu, \bar{\mu} \in \mathbb{D}([0,T];  \mathcal{P}_2(\mathbb{R}^d)) $, we define $\rho_T(\mu, \bar{\mu}):= \sup_{t \in [0,T]}\mathcal{W}_2(\mu_t, \bar{\mu}_t) $. It can be seen that  the space $\mathbb{D}([0,T];  \mathcal{P}_2(\mathbb{R}^d))$  is a complete metric space under $\rho_T$ (see, e.g., Lemma \ref{lem:cms} in the appendix). 
Consider a map $\Phi:\mathbb{D}([0,T];  \mathcal{P}_2(\mathbb{R}^d)) \mapsto  \mathbb{D}([0,T]; \mathcal{P}_2(\mathbb{R}^d))$ defined by $\Phi(\mu):=\{L_{x^\mu_t}\}_{t\in[0,T]}$, where $\{L_{x^\mu_t}\}_{t\in[0,T]}$ is the flow of marginal laws of the process satisfying the SDE
\begin{equation} \label{eq:sde:eu}
    x_t^\mu = x_0 + \int_0^t b_s(x_s^\mu, \mu_s) ds + \int_0^t \sigma_s(x_s^\mu, \mu_s) dw_s + \int_0^t \int_{Z} \gamma_s(x_{s}^\mu, \mu_s, z) \tilde{n}_p(ds,dz)
\end{equation}
almost surely for any $t \in[0,T]$. 
The above equation is  an ordinary SDE driven by L\'evy noise, with coefficients additionally dependent on the time parameter through $\mu_t$, and its  existence and uniqueness is known from \cite[Theorem 1]{gyongy1980}. 

First we observe that  $\Phi$ is well-defined by showing that $\{L_{x^\mu_t}\}_{t\in[0,T]} \in \mathbb{D}([0,T];  \mathcal{P}_2(\mathbb{R}^d))$. 
Notice that the solution of SDE \eqref{eq:sde:eu} is a c\`adl\`ag process, \textit{i.e.},  $x^\mu \in \mathbb{D}([0,T];\mathbb{R}^d)$. 
Indeed, for every $t\in[0,T]$, $L_{x^\mu_t}\in\mathcal{P}_2(\mathbb{R}^d)$ because 
\begin{align*}
\sup_{t\in[0,T]} \int_{\mathbb{R}^d} |x|^2 \mu_t(dx) <\infty, \quad \mbox{and} \quad \int_{\mathbb{R}^d} |x|^2 L_{x_t^\mu}(dx)=\mathbb{E}|x_t^\mu|^2 \leq \sup_{t\in[0,T]} \mathbb{E}|x_t^\mu|^2 \leq K,  
\end{align*}
where $K:=K(d, m, \mathbb{E}|x_0|^2, L)$ is a positive constant (see \cite[Lemma 1]{kumar2017b}). 
Also, $\{L_{x^\mu_t}\}_{t\in[0,T]} $ is   c\`adl\`ag since  $\mathcal{W}_2^2(L_{x^\mu_t}, L_{x^\mu_s})\leq \mathbb{E}|x_t^\mu-x_s^\mu|^2$,
which converges to $0$ as $s \downarrow t$, and $\mathcal{W}_2^2(L_{x^\mu_{t}}, L_{x^\mu_s}) \leq \mathbb{E}|x_{t-}^\mu-x_s^\mu|^2$, which converges to $0$ as $s \uparrow t-$. 

We proceed to prove that $\Phi^{j}$ is a contraction for sufficiently large $j$. Let $\mu, \bar{\mu} \in \mathbb{D}([0,T]; \mathcal{P}_2(\mathbb{R}^d))$ and use It\^o's formula to obtain 
\begin{align*}
\mathbb{E}  |x_t^\mu &-x_t^{\bar{\mu}}|^2   =  2 \mathbb{E}  \int_0^t (x_s^\mu-x_s^{\bar{\mu}})(b_s(x_s^\mu, \mu_s)-b_s(x_s^{\bar{\mu}}, \bar{\mu}_s)) ds
\\
& +  2 \mathbb{E}  \int_0^t  (x_s^\mu-x_s^{\bar{\mu}})(\sigma_s(x_s^\mu, \mu_s)-\sigma_s(x_s^{\bar{\mu}}, \bar{\mu}_s)) dw_s  +  \mathbb{E}    \int_0^t  \big|\sigma_s(x_s^\mu, \mu_s)-\sigma_s(x_s^{\bar{\mu}}, \bar{\mu}_s)\big|^2 ds 
\\
& + 2 \mathbb{E}    \int_0^t \int_Z \big(x_s^\mu-x_s^{\bar{\mu}}\big)\big(\gamma_s(x_s^\mu, \mu_s, z)-\gamma_s(x_s^{\bar{\mu}}, \bar{\mu}_s, z)\big) \tilde{n}_p(ds,dz) 
\\
&+ \mathbb{E}    \int_0^t \int_Z \big\{|x_s^\mu-x_s^{\bar{\mu}}+\gamma_s(x_s^\mu, \mu_s, z)-\gamma_s(x_s^{\bar{\mu}}, \bar{\mu}_s, z)|^{2} -|x_s^\mu-x_s^{\bar{\mu}}|^{2}
\\
&\quad-2 (x_s^\mu-x_s^{\bar{\mu}}) (\gamma_s(x_s^\mu, \mu_s, z)- \gamma_s(x_s^{\bar{\mu}}, \bar{\mu}_s, z))\big\} n_p(ds,dz) 
\\
 = &  \mathbb{E}  \int_0^t \big\{ 2(x_s^\mu-x_s^{\bar{\mu}})(b_s(x_s^\mu, \mu_s)-b_s(x_s^{\bar{\mu}}, \bar{\mu}_s) ) +  \big|\sigma_s(x_s^\mu, \mu_s)-\sigma_s(x_s^{\bar{\mu}}, \bar{\mu}_s )\big|^2 
\\
& \quad + \int_Z | \gamma_s(x_s^\mu, \mu_s, z)-\gamma_s(x_s^{\bar{\mu}}, \bar{\mu}_s, z)|^{2} \nu(dz) \big\}ds,
\end{align*}
which on using Assumption A-\ref{as:monotonicity} yields
\begin{align*}
\mathbb{E}    |x_t^\mu-x_t^{\bar{\mu}}|^2 \leq L \mathbb{E}  \int_0^t \big\{|x_s^\mu-x_s^{\bar{\mu}}|^2 + \mathcal{W}^2_2(\mu_s, \bar{\mu}_s) \big\} ds <\infty
\end{align*}
and then, by using Gr\"onwall's lemma, one obtains
\begin{align*}
\mathbb{E}    |x_t^\mu-x_t^{\bar{\mu}}|^2 \leq  e^{L t} \int_0^t \mathcal{W}^2_2(\mu_s, \bar{\mu}_s) ds
\end{align*}
for any $t\in[0,T]$. Thus, we have, 
\begin{align*}
\rho^2_t ( \Phi(\mu), \Phi( \bar{\mu} ) )& = \rho^2_t ( L_{x^\mu},  L_{x^{\bar{\mu}}}  ) = \sup_{s\in[0,t]} \mathcal{W}_2^2(L_{x_s^\mu},  L_{x_s^{\bar{\mu}}} ) \leq  \sup_{s\in[0,t]} \mathbb{E}    |x_s^\mu-x_s^{\bar{\mu}}|^2
\\
&\leq K \int_0^t \sup_{r\in[0,s]}\mathcal{W}^2_2(\mu_r, \bar{\mu}_r) ds= e^{L t} \int_0^t \rho_s^2(\mu, \bar{\mu}) ds
\end{align*}
 for any $t\in[0,T]$ where $K=e^{LT}$. Denoting by $\Phi^N$ the $n$-fold composition of $\Phi$ and  iterating the above inequality, one obtains 
\begin{align*}
\rho^2_T ( \Phi^j(\mu), \Phi^j( \bar{\mu} ) ) & \leq K^j \int_0^T \int_0^{t_1}\cdots \int_0^{t_{j-1}} \rho_{t_j}^2(\mu, \bar{\mu}) dt_j \cdots dt_1
\\
& \leq K^j \int_0^T \frac{(T-t_j)^{j-1}}{(j-1)!} \rho_{t_j}^2 (\mu,\bar{\mu}) dt_j \leq \frac{(KT)^j}{j!} \rho_{T}^2(\mu, \bar{\mu}). 
\end{align*}
For $j$ large enough, $\Phi^j$ is a contraction and hence $\Phi$ has a unique fixed point, which is the unique solution of the McKean--Vlasov SDE \eqref{eq:sde}.  

{If additionally Assumptions $A$--\ref{as:initial} and $A$--\ref{as:coercivity:p0} hold for $p_0>2$, then one can bound  the $p_0$-th moment  of the solution of  the McKean--Vlasov SDE \eqref{eq:sde}. By It\^o's formula, one obtains,}
\begin{align*}
\mathbb{E}|x_t|^{p_0} & =\mathbb{E}|x_0|^{p_0} + p_0 \mathbb{E} \int_0^t |x_s|^{p_0-2}  x_s b_{s} (x_{s}, L_{x_s}) ds  + p_0 \mathbb{E} \int_0^t |x_s|^{p_0-2}  x_s \sigma_{s} (x_{s}, L_{x_s}) dw_s 
\\
& + \frac{p_0(p_0-2)}{2} \mathbb{E} \int_0^t |x_s|^{p_0-4} |\sigma_{s} (x_s, L_{x_s})^* x_s|^2 ds +\frac{p_0}{2} \mathbb{E} \int_0^t |x_s|^{p_0-2} |\sigma_s (x_s, L_{x_s})|^2 ds 
\\
& + p_0 \mathbb{E} \int_0^t \int_Z |x_s|^{p_0-2}  x_s \gamma_s (x_s, L_{x_s}, z) \tilde{n}_p(ds, dz) 
\\
&+  \mathbb{E} \int_0^t \int_Z \big\{|x_s+\gamma_s (x_s, L_{x_s}, z)|^{p_0}- |x_s|^{p_0} - p_0 |x_s|^{p_0-2} x_s \gamma_s (x_s, L_{x_s}, z)\big\} n_p(ds,dz)
\end{align*}
for any $t\in[0,T]$. By the remainder formula, 
\begin{align}
|y|^{p_0} & = |a|^{p_0}+p_0 |a|^{p_0-2} a (y-a) +  p_0(p_0-1) \int_{0}^1 (1-\theta) |y-a|^2 |a+\theta(y-a)|^{p_0-2} d\theta \label{eq:remainder}
\end{align}
for any $y\in\mathbb{R}^d$ and $a\in\mathbb{R}^d$. Hence, on using the above estimate, 
\begin{align*}
\mathbb{E}|x_t|^{p_0} &   \leq  \mathbb{E}|x_0|^{p_0} + \frac{p_0}{2} \mathbb{E} \int_0^t \Big\{2|x_s|^{p_0-2}  x_s b_{s} (x_{s}, L_{x_s})    + (p_0-1) |x_s|^{p_0-2} |\sigma_{s} (x_s, L_{x_s})|^2 
\\
&+2(p_0-1) \int_Z |\gamma_s (x_s, L_{x_s}, z)|^{2} \int_0^1 (1-\theta) |x_s +\theta  \gamma_s (x_s, L_{x_s}, z)|^{p_0-2} d\theta \nu(dz) \Big\}  ds,
\end{align*}
and then the application of Assumption $A$--\ref{as:coercivity:p0} yields
\begin{align*}
\mathbb{E}|x_t|^{p_0} &\leq  \mathbb{E}|x_0|^{p_0} + K \mathbb{E} \int_0^t \{ \bar{M}_t+|x_s|^{p_0}+\mathcal{W}_2^{p_0}(L_{x_s}, \delta_0)\} ds
\\
& \leq \mathbb{E}|x_0|^{p_0} + K \sup_{t\in[0,T]}\mathbb{E}  \bar{M}_t + K  \int_0^t \mathbb{E} |x_s|^{p_0} ds + K  \int_0^t (\mathbb{E} |x_s|^2 )^{p_0/2} ds
\end{align*}
for any $t\in[0,T]$. By using H\"older's inequality and Gr\"onwall's inequality, the proof is completed.  
\end{proof}

\section{Propagation of Chaos}
\label{sec:PoCmain}
Let us now introduce the non-interacting particle system associated with the McKean--Vlasov SDE \eqref{eq:sde}. 
{Suppose that $w^{1}, \ldots,w^N$ are $N\in\mathbb{N}$ independent and identically distributed Wiener processes. }

The state $x^i$ of the particle $i$ is given by 
\begin{align} \label{eq:non-interacting}
x_t^{i}=x_0^{i} +\int_{0}^t b_s(x_s^{i}, L_{x_s^i}) ds + \int_{0}^t \sigma_s(x_s^{i}, L_{x_s^i}) dw_s^i + \int_0^t \int_Z \gamma_s(x_{s}^{i}, L_{x_{s}^i}, z) \tilde{n}_p^i(ds,dz)
\end{align}
and the associated interacting particle system is given by 
\begin{align} \label{eq:interacting}
x_t^{i,N}=x_0^{i} +\int_{0}^t b_s(x_s^{i, N}, \mu_s^{x,N}) ds + \int_{0}^t \sigma_s(x_s^{i, N}, \mu_s^{x,N}) dw_s^i + \int_0^t \int_Z \gamma_s(x_{s}^{i, N}, \mu_{s}^{x,N}, z) \tilde{n}_p^i(ds,dz)
\end{align}
almost surely, where $\mu_t^{x,N}$ is an empirical measure of $N$ particles given by, 
\begin{align*}
\mu_t^{x,N}(\cdot)= \frac{1}{N} \sum_{i=1}^N \delta_{x_t^{i,N}}(\cdot)
\end{align*}
 almost surely for any $t\in[0,T]$ and $N\in\mathbb{N}$. Notice that $[0,T] \ni t \mapsto \mu_t^{x,N}\in\mathcal{P}_2(\mathbb{R}^d)$ is a  c\`adl\`ag process almost surely. We next state and prove the PoC result.
\begin{prop} \label{prop:propagation}
Let Assumptions A-\ref{as:initial} to A-\ref{as:coercivity:p0} be satisfied with $p_0>4$. 
Then, the interacting particle system \eqref{eq:interacting} is wellposed and converges to the non-interacting particle system \eqref{eq:non-interacting} with a rate of convergence given by
\[ 
\sup_{i\in\{1,\ldots,N\}}\sup_{t\in[0,T]}\mathbb{E}|x^i_t-x^{i,N}_t|^2\leq K 
\begin{cases}
N^{-1/2}, &  \mbox{ if }  d<4, \\
N^{-1/2} \ln(N), &  \mbox{ if } d=4, \\
N^{-2/d}, &  \mbox{ if } d>4,
\end{cases}
\]
for any $N\in\mathbb{N}$, where the constant $K>0$ does not depend on $d$ and $N$. 
\end{prop}
\begin{proof} 
First, notice that the interacting particle system $\{x^{i,N}\}_{i\in \{1,\ldots, N\}}$ given in equation \eqref{eq:interacting}  can be regarded as an ordinary SDE driven by L\'evy noise with random coefficients  taking values in $\mathbb{R}^{d \times N}$. Thus, due to \cite[Theorem 1]{gyongy1980}, it has a unique c\`adl\`ag solution under Assumptions $A$--\ref{as:initial} to $A$--\ref{as:continuity} such that
\begin{align*}
\sup_{i\in\{1,\ldots,N\}}\sup_{t\in [0,T]}\mathbb{E}|x_t^{i, N}|^{2} \leq K
\end{align*}
for any $N\in\mathbb{N}$ where $K>0$ does not depend on $N$. 
By It\^o's formula, 
\begin{align*}
\mathbb{E}  |x_t^i & -x_t^{i,N}|^2   =  2 \mathbb{E} \int_0^t (x_s^i-x_s^{i,N})(b_s(x_s^{i}, L_{x_s^i}) -b_s(x_s^{i, N}, \mu_s^{x,N})) ds
\\
& +  2 \mathbb{E} \int_0^t  (x_s^i-x_s^{i,N}) (\sigma_s(x_s^{i}, L_{x_s^i})-\sigma_s(x_s^{i, N}, \mu_s^{x,N})) dw_s^i  +  \mathbb{E}    \int_0^t  |\sigma_s(x_s^{i}, L_{x_s^i})-\sigma_s(x_s^{i, N}, \mu_s^{x,N})|^2 ds 
\\
& + 2 \mathbb{E}  \int_0^t \int_Z (x_{s}^i-x_{s}^{i,N}) (\gamma_s(x_{s}^{i}, L_{x_{s}^i}, z)-\gamma_s(x_{s}^{i, N}, \mu_{s}^{x,N}, z)) \tilde{n}_p^i(ds,dz) 
\\
&+ \mathbb{E}  \int_0^t \int_Z \big\{|x_{s}^i-x_{s}^{i,N}+\gamma_s(x_{s}^{i}, L_{x_{s}^i}, z)-\gamma_s(x_{s}^{i, N}, \mu_{s}^{x,N}, z)|^{2} -|x_{s}^i-x_{s}^{i,N}|^{2}
\\
&\quad-2 (x_{s}^i-x_{s}^{i,N}) (\gamma_s(x_{s}^{i}, L_{x_{s}^i}, z)- \gamma_s(x_{s}^{i, N}, \mu_{s}^{x,N}, z))\big\} n_p^i(ds,dz) 
\\
& =  2 \mathbb{E} \int_0^t \big\{ (x_s^i-x_s^{i,N})(b_s(x_s^{i}, L_{x_s^i}) -b_s(x_s^{i, N}, \mu_s^{x,N})) +  \big|\sigma_s(x_s^{i}, L_{x_s^i})-\sigma_s(x_s^{i, N}, \mu_s^{x,N})\big|^2 
\\
&  \qquad + \int_Z |\gamma_s(x_{s}^{i}, L_{x_{s}^i}, z)- \gamma_s(x_{s}^{i, N}, \mu_{s}^{x,N}, z)|^2 \nu(dz) \big\} ds,
\end{align*}
which on the application of Assumption A--\ref{as:monotonicity} gives
\begin{align*}
\mathbb{E}  |x_t^i & -x_t^{i,N}|^2 \leq K \int_0^t  \big\{\mathbb{E}  |x_s^i  -x_s^{i,N}|^2 +\mathbb{E} \mathcal{W}_2^2 (L_{x_{s}^i}, \mu_s^{x,N}) \big\}ds
\\
& \leq K \int_0^t  \big\{ \mathbb{E}  |x_s^i  -x_s^{i,N}|^2 +  \mathbb{E} \mathcal{W}_2^2 (L_{x_{s}^i}, \mu_s^{x}) + \mathbb{E} \mathcal{W}_2^2 (\mu_s^{x}, \mu_s^{x,N})  \big\} ds,
\end{align*}
where $\mu_t^x$ is an empirical measure given by 
\begin{align*}
\mu_t^x(\cdot) =\frac{1}{N} \sum_{i=1}^N \delta_{x_t^i} (\cdot),
\end{align*}
and then Gr\"onwall's inequality yields
\begin{align*}
\sup_{i\in\{1,\ldots,N\}}\sup_{t\in[0,T]}\mathbb{E}  |x_t^i & -x_t^{i,N}|^2  \leq K  \int_0^T \sup_{i\in\{1,\ldots,N\}} \mathbb{E} \mathcal{W}_2^2 (L_{x_s^i},\mu_s^{x}) ds
\end{align*}
for any $t\in[0,T]$. By   \cite[Theorem 5.8]{carmona2018a},  one has
\begin{align*}
\mathbb{E}\mathcal{W}_2^2(L_{x_t^i}, \mu_t^x) \leq K
\begin{cases}
N^{-1/2}, & \mbox{ if } d<4, \\
N^{-1/2} \ln(N), & \mbox{ if } d=4, \\
N^{-1/2}, & \mbox{ if } d>4, 
\end{cases}
\end{align*}
for any $t\in[0,T]$,  $i\in\{1,\ldots,N\}$ and $N\in\mathbb{N}$.  This completes the proof. 
\end{proof}
\section{Tamed Euler Scheme}
In this section, we make the following additional assumptions. Due to the requirement in Proposition~\ref{prop:propagation}, $p_0$ is now and for the rest of this article assumed to be strictly greater than $4$. 
\begin{asA} \label{as:monotonicity:rate}
There exist constants $L>0$ and $\eta>1$ such that
\begin{align*}
2(x-\bar{x})(b_t(x,\mu)& -b_t(\bar{x},\bar{\mu}))+\eta |\sigma_t(x,\mu)-\sigma_t(\bar{x}, \bar{\mu})|^2 
\\
& +\eta \int_Z |\gamma_t(x,\mu,z)-\gamma_t(\bar{x},\bar{\mu},z)|^2 \nu(dz) \leq L\big\{|x-\bar{x}|^2 +\mathcal{W}_2^2 (\mu, \bar{\mu})\big\}
\end{align*}
almost surely for any $t\in[0,T]$, $x,\bar{x} \in\mathbb{R}^d$ and $\mu, \bar{\mu} \in \mathcal{P}_2(\mathbb{R}^d)$. 
\end{asA}
\begin{asA} \label{as:poly:Lips:b}
There exist  constants $L>0$ and $\chi>0$ such that
\begin{align*}
|b_t(x,\mu)-b_t(\bar{x},\bar{\mu})|^2 \leq L\big\{ (1+|x|+|\bar{x}|)^{\chi} |x-\bar{x}|^2+\mathcal{W}_2^2(\mu,\bar{\mu})\big\}
\end{align*}
almost surely for any $t\in[0,T]$, $x,\bar{x}\in\mathbb{R}^d$ and $\mu, \bar{\mu}\in\mathcal{P}_2(\mathbb{R}^d)$. 
\end{asA}
\begin{asA} \label{as:holder:time}
There exists a constant $L>0$ such that
\begin{align*}
|b_t(x,\mu)-b_s(x,\mu)|^2+|\sigma_t(x,\mu)-\sigma_s(x,\mu)|^2+ \int_Z |\gamma_t(x,\mu,z)-\gamma_s(x,\mu, z)|^2  \nu(dz) \leq L |t-s|  
\end{align*}
almost surely for any $t, s \in[0,T]$, $x\in\mathbb{R}^d$ and $\mu\in\mathcal{P}_2(\mathbb{R}^d)$. 
\end{asA}
\begin{asA} \label{as:bounded}
There exists a sequence $\{M_t\}_{t\in[0,T]}$ of $\mathscr{F}_0$-measurable non-negative random variables satisfying $\sup_{t\in[0,T]}\mathbb{E}M_t<\infty$ such that
\begin{align*}
|b_t(0,\delta_0)|^2 \vee  |\sigma_t(0,\delta_0)|^2  \vee  \int_Z |\gamma_t(0,\delta_0, z)|^2\nu(dz) < M_t 
\end{align*}
almost surely for any $t\in[0,T]$. 
\end{asA}
\begin{rem}
One can notice that Assumption A--\ref{as:monotonicity} follows from Assumption A--\ref{as:monotonicity:rate}. Thus,  the McKean--Vlasov SDE  \eqref{eq:sde} has a unique solution under Assumptions A--\ref{as:initial} (with $p_0=2$),  A--\ref{as:coercivity},  A--\ref{as:continuity} and A--\ref{as:monotonicity:rate}.
\end{rem}
\begin{rem} \label{rem:poly:lips:sig:gam}
Due to Assumptions $A$--\ref{as:monotonicity:rate} and A--\ref{as:poly:Lips:b}, there exist constants $L>0$ and $\chi>0$ such that
\begin{align*}
|\sigma_t(x,\mu)-\sigma_t(\bar{x}, \bar{\mu})|^2 +\int_Z |\gamma_t(x,\mu,z)-\gamma_t(\bar{x},\bar{\mu},z)|^2 \nu(dz) \leq L\big\{(1+|x|+|\bar{x}|)^{\chi/2}|x-\bar{x}|^2 +\mathcal{W}_2^2 (\mu, \bar{\mu})\big\}
\end{align*}
almost surely for any $t\in[0,T]$, $x,\bar{x}\in\mathbb{R}^d$ and $\mu, \bar{\mu}\in\mathcal{P}_2(\mathbb{R}^d)$. 
\end{rem} 
\begin{rem} \label{rem:pol:grow:rc}
Due to Assumption A--\ref{as:bounded} and Remark \ref{rem:poly:lips:sig:gam}, there exists a constant $L>0$ and a sequence $\{M_t\}_{t\in[0,T]}$ of $\mathscr{F}_0$-measurable non-negative  random variables satisfying $\sup_{t\in[0,T]}\mathbb{E}M_t<\infty$ such that
\begin{align*}
|b_t(x,\mu)| & \leq L \big\{M_t+ |x|^{\chi/2+1} + \mathcal{W}_2(\mu,\delta_0)\big\},
\\
|\sigma_t(x,\mu)| & \leq L \big\{M_t+|x|^{\chi/4+1} + \mathcal{W}_2(\mu,\delta_0)\big\},
\\
\int_Z |\gamma_t(x,\mu,z)|^2 \nu(dz) & \leq L \big\{M_t+ |x|^{\chi/2+2}+ \mathcal{W}_2^2(\mu,\delta_0)\big\}
\end{align*}
almost surely for any $t\in[0,T]$, $x\in\mathbb{R}^d$ and $\mu \in\mathcal{P}_2(\mathbb{R}^d)$. 
\end{rem}
We now proceed to introduce the tamed Euler scheme for the interacting particle system \eqref{eq:interacting} connected with the McKean--Vlasov SDE \eqref{eq:sde}. 
For every $n \in \mathbb{N}$, let $b^n_t$ and $\sigma^n_t$ be $\mathscr{P}\otimes \mathscr{B} (\mathbb{R}^d) \otimes \mathscr{B}(\mathcal{P}_2(\mathbb{R}^d))$ measurable functions taking values in $\mathbb{R}^d$ and $\mathbb{R}^{d\times m}$ respectively and $\gamma_t^n$ be $\mathscr{P}\otimes \mathscr{B} (\mathbb{R}^d) \otimes \mathscr{B}(\mathcal{P}_2(\mathbb{R}^d)) \otimes \mathscr{Z}$ measurable functions taking values in $\mathbb{R}^d$. 
Also, let us partition the time interval $[0,T]$ into $n\in\mathbb{N}$ sub-intervals of equal length  and define $\kappa_n(t) :=\lfloor nt \rfloor/n$ for any $t\in[0,T]$ and $n\in\mathbb{N}$. 
The tamed Euler scheme for the interacting particle system \eqref{eq:interacting}  associated with the McKean--Vlasov SDE \eqref{eq:sde} is given by 
\begin{align} \label{eq:scheme}
x_t^{i,N,n} & =x_0^{i} +\int_{0}^t b^n_{\kappa_n(s)}(x_{\kappa_n(s)}^{i, N, n}, \mu_{\kappa_n(s)}^{x,N, n}) ds + \int_{0}^t \sigma^n_{\kappa_n(s)}(x_{\kappa_n(s)}^{i, N, n}, \mu_{\kappa_n(s)}^{x,N, n}) dw_s^i \notag 
\\
& \qquad + \int_0^t \int_Z \gamma^n_{\kappa_n(s)}(x_{\kappa_n(s)}^{i, N, n}, \mu_{\kappa_n(s)}^{x,N, n}, z) \tilde{n}_p^i(ds,dz)
\end{align}
almost surely for any $i\in\{1, \ldots, N\}$, where $\mu_t^{x, N, n}$ is an empirical measure given by
\begin{align*}
\mu_t^{x, N, n} (\cdot)= \frac{1}{N} \sum_{i=1}^N \delta_{x_t^{i,N,n}} (\cdot)
\end{align*}
for any $t\in[0,T]$  and $n,N\in\mathbb{N}$.  
\begin{asB} \label{asb:coercivity:scheme}
For every $n\in\mathbb{N}$, there exist a sequence $\{\bar{M}^n_t\}_{t\in[0,T]}$ of $\mathscr{F}_0$-measurable non-negative random variables  satisfying $\sup_{n\in\mathbb{N}}\sup_{t\in[0,T]}\mathbb{E}\bar{M}^n_t<\infty$ and a constant $L>0$ such that
\begin{align*}
2|&x|^{p_0-2}  xb_t^n(x, \mu) +(p_0-1) |x|^{p_0-2} |\sigma_t^n(x,\mu)|^2 
\\
& + 2(p_0-1) \int_Z |\gamma_t^n(x,\mu,z)|^2 \int_0^1 (1-\theta) |x+\theta\gamma_t^n(x,\mu,z)|^{p_0-2} d\theta  \nu(dz) \leq L\big\{\bar{M}^n_t+|x|^{p_0}+\mathcal{W}_2^{p_0}(\mu,\delta_0) \big\}
\end{align*}
almost surely for any $x\in\mathbb{R}^d$, $t\in[0,T]$ and $\mu \in \mathcal{P}_2(\mathbb{R}^d)$. 
\end{asB}
\begin{asB} \label{asb:growth:n}
For every $n\in\mathbb{N}$, there exist a sequence $\{\bar{M}^n_t\}_{t\in[0,T]}$ of $\mathscr{F}_0$-measurable non-negative random variables satisfying $\sup_{n\in\mathbb{N}}\sup_{t\in[0,T]}\mathbb{E} (\bar{M}^n_t)^{p_0}<\infty$ and a constant $L>0$ such that
\begin{align*}
|b_t^n(x,\mu)| &\leq  L \min\big\{ n^{1/4}  \big( \bar{M}^n_t +|x|+\mathcal{W}_2(\mu, \delta_0)\big), M_t+ |x|^{\chi/2+1} + \mathcal{W}_2(\mu,\delta_0)\big\},
\\
|\sigma_t^n(x,\mu)|& \leq L \min\big\{ n^{1/6} \big\{\bar{M}^n_t +|x|+\mathcal{W}_2(\mu, \delta_0)\big\}, M_t+ |x|^{\chi/4+1} + \mathcal{W}_2(\mu,\delta_0) \big\},
\\
\int_Z |\gamma_t^n(x,\mu, z)|^2  \nu(dz)  &  \leq   L \min\big\{ n^{1/3}  (\bar{M}^n_t  +|x|+\mathcal{W}_2(\mu, \delta_0))^2, M_t+ |x|^{\chi/2+2}+ \mathcal{W}_2^2(\mu,\delta_0) \big\},
\\
\int_Z |\gamma_t^n(x,\mu, z)|^{q}  \nu(dz)  & \leq   L \min\big\{n^{1/3}  \big(\bar{M}^n_t +|x|+\mathcal{W}_2(\mu, \delta_0)\big)^{q}, M_t+ |x|^{\chi/2+q}+ \mathcal{W}_2^q(\mu,\delta_0) \big\},  \,\, 2 \leq q\leq p_0
\end{align*}
almost surely for any $t\in[0,T]$ and $\mu\in\mathcal{P}_2(\mathbb{R}^d)$, where the process $M$ is given in Remark \ref{rem:pol:grow:rc}.
\end{asB} 
\begin{asB} \label{asb:diff:rate}
For every $\epsilon>0$, there exists a constant $K>0$ such that 
\begin{align*}
 \mathbb{E} \int_0^T |b_{\kappa_n(s)}(x_{\kappa_n(s)}^{i, N,n}, \mu_{\kappa_n(s)}^{x,N,n})- b^n_{\kappa_n(s)}(x_{\kappa_n(s)}^{i, N, n}, \mu_{\kappa_n(s)}^{x,N, n})|^2 ds & \leq K n^{-2/(2+\epsilon)},
 \\
  \mathbb{E} \int_0^T |\sigma_{\kappa_n(s)}(x_{\kappa_n(s)}^{i, N,n}, \mu_{\kappa_n(s)}^{x,N,n})- \sigma^n_{\kappa_n(s)}(x_{\kappa_n(s)}^{i, N, n}, \mu_{\kappa_n(s)}^{x,N, n})|^2 ds& \leq K n^{-2/(2+\epsilon)}, 
  \\
    \mathbb{E} \int_0^T\int_Z |\gamma_{\kappa_n(s)}(x_{\kappa_n(s)}^{i, N,n}, \mu_{\kappa_n(s)}^{x,N,n}, z)- \gamma^n_{\kappa_n(s)}(x_{\kappa_n(s)}^{i, N, n}, \mu_{\kappa_n(s)}^{x,N, n}, z)|^2 \nu(dz) ds & \leq K n^{-2/(2+\epsilon)}
\end{align*}
for any $t\in[0,T]$, $\mu\in\mathcal{P}_2(\mathbb{R}^d)$ and $n, N\in\mathbb{N}$, where $K$ does not depend on $n$ and $N$. 
\end{asB}
\begin{rem}
Notice that Assumption $B$--\ref{asb:diff:rate}  contains the numerical solution and, on the face of it, seems difficult to verify \textit{a priori}. 
There are two key, competing factors in identifying a suitable explicit form of taming. 
First, the tamed coefficients need to satisfy  Assumption $B$--\ref{asb:growth:n}  and, second, the rate of strong convergence of the numerical scheme (\textit{e.g.} $1/2$ in our case) should be used appropriately (\textit{e.g.} in the denominator of the tamed coefficients). 
Then, Assumption $B$--\ref{asb:diff:rate} holds by construction. We refer to Section 5, where we verify all assumptions for concrete schemes in applications, for details. 
\end{rem}
\begin{lem} \label{lem:one-step:mb}
Let Assumption B--\ref{asb:growth:n} hold. Then, 
\begin{align*}
\mathbb{E}\big(|x_t^{i, N, n}- x_{\kappa_n(t)}^{i, N, n}|^{q} \big| \mathscr{F}_{\kappa_n(t)}\big) \leq  K
\begin{cases}
 n^{-2/3} \big\{\bar{M}^n_t+|x_{\kappa_n(t)}^{i, N, n}|+ \mathcal{W}_2(\mu_{\kappa_n(t)}^{x,N, n}, \delta_0)\big\}^{q}, & \mbox{ if } q= p_0-2,  p_0-1 
\\
 n^{-1/3} \big\{\bar{M}^n_t+|x_{\kappa_n(t)}^{i, N, n}|+ \mathcal{W}_2(\mu_{\kappa_n(t)}^{x,N, n}, \delta_0)\big\}, & \mbox{ if }  q=1
\end{cases}
\end{align*}
almost surely for any $i\in\{1,\ldots,N\}$, $t\in[0,T]$ and $n,N\in\mathbb{N}$, where  constant $K>0$ does not depend on $n$ and $N$.  
\end{lem}
\begin{proof} 
As mentioned above, we have $p_0>4$. From equation \eqref{eq:scheme}, 
\begin{align*}
\mathbb{E} \big(|x_t^{i,N,n}  & - x_{\kappa_n(t)}^{i,N,n}|^{q}|\mathscr{F}_{\kappa_n(t)}\big)   \leq K \mathbb{E}\Big(\Big|\int_{\kappa_n(t)}^t b^n_{\kappa_n(s)}(x_{\kappa_n(s)}^{i, N, n}, \mu_{\kappa_n(s)}^{x,N, n}) ds \Big|^{q}\Big| \mathscr{F}_{\kappa_n(t)}\Big)
\\
& +K  \mathbb{E}\Big(\Big| \int_{\kappa_n(t)}^t \sigma^n_{\kappa_n(s)}(x_{\kappa_n(s)}^{i, N, n}, \mu_{\kappa_n(s)}^{x,N, n}) dw_s^i \Big|^{q} \Big| \mathscr{F}_{\kappa_n(t)}\Big)\notag 
\\
& +K \mathbb{E}\Big(\Big|\int_{\kappa_n(t)}^t \int_Z \gamma^n_{\kappa_n(s)}(x_{\kappa_n(s)}^{i, N, n}, \mu_{\kappa_n(s)}^{x,N, n}, z) \tilde{n}_p^i(ds,dz)\Big|^{q}\Big|\mathscr{F}_{\kappa_n(t)}\Big),
\end{align*}
which on the application of H\"older's and martingale inequalities  yields
\begin{align*}
\mathbb{E} \big(|x_t^{i,N,n}  - & x_{\kappa_n(t)}^{i,N,n}|^{q} |  \mathscr{F}_{\kappa_n(t)}\big)   \leq K n^{-q+1}\mathbb{E}\Big(\int_{\kappa_n(t)}^t | b^n_{\kappa_n(s)}(x_{\kappa_n(s)}^{i, N, n}, \mu_{\kappa_n(s)}^{x,N, n})|^{q} ds\Big| \mathscr{F}_{\kappa_n(t)}\Big)
\\
& +K n^{-q/2+1} \mathbb{E} \Big(\int_{\kappa_n(t)}^t |\sigma^n_{\kappa_n(s)}(x_{\kappa_n(s)}^{i, N, n}, \mu_{\kappa_n(s)}^{x,N, n})|^{q} ds\Big|\mathscr{F}_{\kappa_n(t)}\Big) \notag 
\\
&+ Kn^{-q/2+1}  \mathbb{E} \Big(\int_{\kappa_n(t)}^t \Big( \int_Z |\gamma^n_{\kappa_n(s)}(x_{\kappa_n(s)}^{i, N, n}, \mu_{\kappa_n(s)}^{x,N, n}, z)|^2  \nu(dz)\Big)^{q/2} ds\Big|\mathscr{F}_{\kappa_n(t)}\Big)
\\
&+ K \mathbb{E}\Big(\int_{\kappa_n(t)}^t  \int_Z |\gamma^n_{\kappa_n(s)}(x_{\kappa_n(s)}^{i, N, n}, \mu_{\kappa_n(s)}^{x,N, n}, z)|^{q}  \nu(dz)  ds \Big|\mathscr{F}_{\kappa_n(t)}\Big)
\end{align*}
almost surely for any $t\in[0,T]$, $i\in\{1,\ldots, N\}$ and $n,N\in\mathbb{N}$. When $ q =1$, then the last term on the right-hand side of the above expression does not appear. 
Thus, by Assumption $B$--\ref{asb:growth:n} one obtains
\begin{align*}
\mathbb{E} &\big(|x_t^{i,N,n}  -  x_{\kappa_n(t)}^{i,N,n}|^{q} |  \mathscr{F}_{\kappa_n(t)}\big)  
\\
& \leq K 
\begin{cases}
   (n^{-3q/4}+n^{-q/3}+n^{-2/3})\big\{M^n_t+|x_{\kappa_n(s)}^{i, N, n}|+ \mathcal{W}_2(\mu_{\kappa_n(s)}^{x,N, n}, \delta_0)\big\}^{q}, & \mbox{ if }   q = p_0-2, p_0-1, 
   \\
   (n^{-3/4}+n^{-1/3})\big\{M^n_t+|x_{\kappa_n(s)}^{i, N, n}|+ \mathcal{W}_2(\mu_{\kappa_n(s)}^{x,N, n}, \delta_0)\big\}, & \mbox{ if }  q =1
\end{cases}
\end{align*}
almost surely for any $t\in[0,T]$, $i\in\{1,\ldots, N\}$ and $n,N\in\mathbb{N}$. This completes the proof. 
\end{proof}
The following lemma gives the moment bound of the scheme \eqref{eq:scheme}. 
\begin{lem} \label{lem:scheme:mb}
Let Assumptions $A$--\ref{as:initial}, $B$--\ref{asb:coercivity:scheme}  and B--\ref{asb:growth:n} hold.  Then, 
\begin{align*}
\sup_{i\in\{1,\ldots, N\}}\sup_{t\in[0,T]} \mathbb{E}   |x_t^{i, N, n}|^{p_0} \leq K,
\end{align*}
for any $n, N\in\mathbb{N}$, where $K>0$ does not depend on $n$ and $N$. 
\end{lem}
\begin{proof}
On the application of It\^o's formula, 
\begin{align*}
\mathbb{E}|x^{i,N,n}_t|^{p_0}  =  & \, \mathbb{E}|x^{i}_0|^{p_0} + p_0 \mathbb{E} \int_0^t |x^{i,N,n}_s|^{p_0-2}  x^{i,N,n}_s b_{\kappa_n(s)}^n (x^{i,N,n}_{\kappa_n(s)}, \mu_{\kappa_n(s)}^{x,N,n}) ds 
\\
& + p_0 \mathbb{E} \int_0^t |x^{i,N,n}_s|^{p_0-2}  x^{i,N,n}_s \sigma_{\kappa_n(s)}^n (x^{i,N,n}_{\kappa_n(s)}, \mu_{\kappa_n(s)}^{x,N,n}) dw_s^i 
\\
& + \frac{p_0(p_0-2)}{2} \mathbb{E} \int_0^t |x^{i,N,n}_s|^{p_0-4} |\sigma_{\kappa_n(s)}^n (x^{i,N,n}_{\kappa_n(s)}, \mu_{\kappa_n(s)}^{x,N,n})^* x^{i,N,n}_{\kappa_n(s)}|^2 ds 
\\
& +\frac{p_0}{2} \mathbb{E} \int_0^t |x^{i,N,n}_s|^{p_0-2} |\sigma_{\kappa_n(s)}^n (x^{i,N,n}_{\kappa_n(s)}, \mu_{\kappa_n(s)}^{x,N,n})|^2 ds 
\\
& + p_0 \mathbb{E} \int_0^t \int_Z |x^{i,N,n}_s|^{p_0-2}  x^{i,N,n}_s \gamma_{\kappa_n(s)}^n (x^{i,N,n}_{\kappa_n(s)}, \mu_{\kappa_n(s)}^{x,N,n}, z) \tilde{n}_p(ds, dz) 
\\
&+  \mathbb{E} \int_0^t \int_Z \big\{|x^{i,N,n}_s+\gamma_{\kappa_n(s)}^n (x^{i,N,n}_{\kappa_n(s)}, \mu_{\kappa_n(s)}^{x,N,n}, z)|^{p_0}- |x^{i,N,n}_{s}|^{p_0}
\\
& \qquad \qquad - p_0 |x^{i,N,n}_s|^{p_0-2} x^{i,N,n}_{s} \gamma_{\kappa_n(s)}^n (x^{i,N,n}_{\kappa_n(s)}, \mu_{\kappa_n(s)}^{x,N,n}, z)\big\} n_p(ds,dz),
\end{align*}
which further implies
\begin{align*}
\mathbb{E}|x^{i,N,n}_t|^{p_0}   = & \, \mathbb{E}|x^{i}_0|^{p_0} + p_0 \mathbb{E} \int_0^t |x^{i,N,n}_s|^{p_0-2}  x^{i,N,n}_s b_{\kappa_n(s)}^n (x^{i,N,n}_{\kappa_n(s)}, \mu_{\kappa_n(s)}^{x,N,n}) ds
\\
&  + \frac{p_0(p_0-1)}{2} \mathbb{E} \int_0^t |x^{i,N,n}_s|^{p_0-2} |\sigma_{\kappa_n(s)}^n (x^{i,N,n}_{\kappa_n(s)}, \mu_{\kappa_n(s)}^{x,N,n})|^2 ds 
\\
& +  \mathbb{E} \int_0^t \int_Z \big\{|x^{i,N,n}_s+\gamma_{\kappa_n(s)}^n (x^{i,N,n}_{\kappa_n(s)}, \mu_{\kappa_n(s)}^{x,N,n}, z)|^{p_0}- |x^{i,N,n}_{s}|^{p_0}
\\
& \qquad \qquad - p_0 |x^{i,N,n}_s|^{p_0-2} x^{i,N,n}_{s} \gamma_{\kappa_n(s)}^n (x^{i,N,n}_{\kappa_n(s)}, \mu_{\kappa_n(s)}^{x,N,n}, z)\big\}  \nu(dz) ds
\end{align*}
for any $i\in\{1,\ldots, N\}$, $t\in[0,T]$ and $n,N\in\mathbb{N}$. The last term on the right hand side of the above equation can be estimated by using the remainder formula given in equation \eqref{eq:remainder} and thus one obtains, 
\begin{align*}
\mathbb{E}|  & x^{i,N,n}_t |^{p_0}  \leq  \mathbb{E}|x^{i}_0|^{p_0} + p_0 \mathbb{E} \int_0^t \Big\{|x^{i,N,n}_s|^{p_0-2}  x^{i,N,n}_s-|x^{i,N,n}_{\kappa_n(s)}|^{p_0-2}  x^{i,N,n}_{\kappa_n(s)} \Big\} b_{\kappa_n(s)}^n (x^{i,N,n}_{\kappa_n(s)}, \mu_{\kappa_n(s)}^{x,N,n}) ds \notag
\\
&  + \frac{p_0(p_0-1)}{2} \mathbb{E} \int_0^t \Big\{|x^{i,N,n}_s|^{p_0-2}-|x^{i,N,n}_{\kappa_n(s)}|^{p_0-2}\Big\} |\sigma_{\kappa_n(s)}^n (x^{i,N,n}_{\kappa_n(s)}, \mu_{\kappa_n(s)}^{x,N,n})|^2 ds  \notag
\\
& +p_0 (p_0-1) \mathbb{E} \int_0^t\int_Z  |\gamma_{\kappa_n(s)}^n (x^{i,N,n}_{\kappa_n(s)}, \mu_{\kappa_n(s)}^{x,N,n}, z)|^2 \int_0^1 (1-\theta) \Big\{ |x^{i,N,n}_s+\theta \gamma_{\kappa_n(s)}^n (x^{i,N,n}_{\kappa_n(s)}, \mu_{\kappa_n(s)}^{x,N,n}, z)| ^{p_0-2} \notag
\\
& \qquad \qquad \qquad \qquad  - |x^{i,N,n}_{\kappa_n(s)}+\theta \gamma_{\kappa_n(s)}^n (x^{i,N,n}_{\kappa_n(s)}, \mu_{\kappa_n(s)}^{x,N,n}, z) \big|^{p_0-2} \Big\} d\theta  \nu(dz) ds \notag
\\
& + \frac{p_0}{2} \mathbb{E} \int_0^t \Big\{ 2|x^{i,N,n}_{\kappa_n(s)}|^{p_0-2}  x^{i,N,n}_{\kappa_n(s)}  b_{\kappa_n(s)}^n (x^{i,N,n}_{\kappa_n(s)}, \mu_{\kappa_n(s)}^{x,N,n})  + (p_0-1)  |x^{i,N,n}_{\kappa_n(s)}|^{p_0-2} |\sigma_{\kappa_n(s)}^n (x^{i,N,n}_{\kappa_n(s)}, \mu_{\kappa_n(s)}^{x,N,n})|^2  \notag
\\
& +2 (p_0-1) \int_Z  |\gamma_{\kappa_n(s)}^n (x^{i,N,n}_{\kappa_n(s)}, \mu_{\kappa_n(s)}^{x,N,n}, z)|^2  \int_0^1 (1-\theta)|x^{i,N,n}_{\kappa_n(s)}+\theta \gamma_{\kappa_n(s)}^n (x^{i,N,n}_{\kappa_n(s)}, \mu_{\kappa_n(s)}^{x,N,n}, z)|^{p_0-2} d\theta  \nu(dz) \Big\}ds, \notag
\end{align*}
which on the application of  Assumption $B$--\ref{asb:coercivity:scheme} yields, 
\begin{align}
\mathbb{E}|  x^{i,N,n}_t &|^{p_0}  \leq    \mathbb{E}|x^{i}_0|^{p_0} + T_1+T_2+T_3+ \mathbb{E} \int_0^t \Big\{\bar{M}^n_s+ |x^{i,N,n}_{\kappa_n(s)}|^{p_0} +\mathcal{W}_2^{p_0} (\mu_{\kappa_n(s)}^{x,N,n}, \delta_0)\Big\}ds  \label{eq:T1+T2+T3}
\end{align}
for any $i\in\{1,\ldots, N\}$, $t\in[0,T]$ and $n,N\in\mathbb{N}$. For $T_1$, one uses the following remainder formula, 
\begin{align*}
|y|^{p_0-2} y &=|a|^{p_0-2} a + (p_0-2) \int_0^1 |a+\theta(y-a)|^{p_0-4} (y-a) (a+\theta(y-a) ) (a+\theta(y-a) ) d\theta 
\\
&+ \int_0^1 |a+\theta(y-a) |^{p_0-2} \sum_{j=1}^d (y_j-a_j) \mathbf{e}_j d\theta,
\end{align*}
where $\mathbf{e}_j$ is a vector whose $j$-th element is $1$ and all other elements are $0$ 
and $(y_j-a_j)$ is the $j$-th element of $(y-a)$ for every $j\in\{1,\ldots,d\}$.   Hence, one obtains
\begin{align*}
T_1 & :=p_0 \mathbb{E} \int_0^t \Big\{|x^{i,N,n}_s|^{p_0-2}  x^{i,N,n}_s - |x^{i,N,n}_{\kappa_n(s)}|^{p_0-2}  x^{i,N,n}_{\kappa_n(s)} \Big\} b_{\kappa_n(s)}^n (x^{i,N,n}_{\kappa_n(s)}, \mu_{\kappa_n(s)}^{x,N,n})  ds 
\\
& = (p_0-2) \mathbb{E} \int_0^t  \int_0^1 |x^{i,N,n}_{\kappa_n(s)} + \theta \big(x^{i,N,n}_s-x^{i,N,n}_{\kappa_n(s)}\big)|^{p_0-4}
\\
& \times  (x^{i,N,n}_s-x^{i,N,n}_{\kappa_n(s)}) (x^{i,N,n}_{\kappa_n(s)}  + \theta (x^{i,N,n}_s-x^{i,N,n}_{\kappa_n(s)})) (x^{i,N,n}_{\kappa_n(s)}  + \theta (x^{i,N,n}_s-x^{i,N,n}_{\kappa_n(s)})) b_{\kappa_n(s)}^n (x^{i,N,n}_{\kappa_n(s)}, \mu_{\kappa_n(s)}^{x,N,n})  d\theta ds 
\\
& + \mathbb{E} \int_0^t \int_0^1 |x^{i,N,n}_{\kappa_n(s)}  + \theta (x^{i,N,n}_s-x^{i,N,n}_{\kappa_n(s)})|^{p_0-2} (x^{i,N,n}_s-x^{i,N,n}_{\kappa_n(s)}) b_{\kappa_n(s)}^n (x^{i,N,n}_{\kappa_n(s)}, \mu_{\kappa_n(s)}^{x,N,n}) d\theta ds,
\end{align*}
which on the application of Cauchy-Schwarz inequality and Assumption $B$--\ref{asb:growth:n}  yields
\begin{align*}
T_1 & \leq K \mathbb{E} \int_0^t  \int_0^1 |x^{i,N,n}_{\kappa_n(s)}  + \theta \big(x^{i,N,n}_s-x^{i,N,n}_{\kappa_n(s)}\big)|^{p_0-2} |x^{i,N,n}_s-x^{i,N,n}_{\kappa_n(s)}| 
\\
& \qquad \qquad \times n^{1/4}\Big\{\bar{M}^n_s+ |x^{i,N,n}_{\kappa_n(s)}|+ \mathcal{W}_2 (\mu_{\kappa_n(s)}^{x,N,n}, \delta_0) \Big\} d\theta ds 
\\
& \leq K \mathbb{E} \int_0^t  \Big\{|x^{i,N,n}_{\kappa_n(s)}|^{p_0-2}  +  |x^{i,N,n}_s-x^{i,N,n}_{\kappa_n(s)}|^{p_0-2}\Big\} |x^{i,N,n}_s-x^{i,N,n}_{\kappa_n(s)}| 
\\
& \qquad \qquad \times n^{1/4}\Big\{\bar{M}^n_s+ |x^{i,N,n}_{\kappa_n(s)}|+ \mathcal{W}_2 (\mu_{\kappa_n(s)}^{x,N,n}, \delta_0) \Big\}  ds 
\\
& \leq K \mathbb{E} \int_0^t  |x^{i,N,n}_s-x^{i,N,n}_{\kappa_n(s)}|^{p_0-1}  n^{1/4}\Big\{\bar{M}^n_s+ |x^{i,N,n}_{\kappa_n(s)}|+ \mathcal{W}_2 (\mu_{\kappa_n(s)}^{x,N,n}, \delta_0) \Big\} ds 
\\
& \qquad +  K \mathbb{E}\int_0^t   |x^{i,N,n}_{\kappa_n(s)}|^{p_0-2} |x^{i,N,n}_s-x^{i,N,n}_{\kappa_n(s)}|  n^{1/4}\Big\{\bar{M}^n_s+ |x^{i,N,n}_{\kappa_n(s)}|+ \mathcal{W}_2 (\mu_{\kappa_n(s)}^{x,N,n}, \delta_0) \Big\} ds 
\\
& \leq K  \mathbb{E} \int_0^t  n^{1/4}\Big\{\bar{M}^n_s+ |x^{i,N,n}_{\kappa_n(s)}|+ \mathcal{W}_2 (\mu_{\kappa_n(s)}^{x,N,n}, \delta_0) \Big\}  \mathbb{E}  \big(|x^{i,N,n}_s-x^{i,N,n}_{\kappa_n(s)}|^{p_0-1} |\mathscr{F}_{\kappa_n(s)}\big) ds 
\\
& \qquad +  K \mathbb{E}\int_0^t   |x^{i,N,n}_{\kappa_n(s)}|^{p_0-2}   n^{1/4}\Big\{\bar{M}^n_s+ |x^{i,N,n}_{\kappa_n(s)}|+ \mathcal{W}_2 (\mu_{\kappa_n(s)}^{x,N,n}, \delta_0) \Big\} \mathbb{E}\big( |x^{i,N,n}_s-x^{i,N,n}_{\kappa_n(s)}| |\mathscr{F}_{\kappa_n(s)}\big) ds,
\end{align*}
and then, using Lemma \ref{lem:one-step:mb}, one obtains,
\begin{align*}
T_1 & \leq K  \mathbb{E} \int_0^t  \Big\{\bar{M}^n_s+ |x^{i,N,n}_{\kappa_n(s)}|+ \mathcal{W}_2 (\mu_{\kappa_n(s)}^{x,N,n}, \delta_0) \Big\}^{p_0}  ds 
\\
& \qquad +  K \mathbb{E}\int_0^t   |x^{i,N,n}_{\kappa_n(s)}|^{p_0-2}   \Big\{\bar{M}^n_s+ |x^{i,N,n}_{\kappa_n(s)}|+ \mathcal{W}_2 (\mu_{\kappa_n(s)}^{x,N,n}, \delta_0) \Big\}^2 ds 
\end{align*}
for any $i\in\{1,\ldots, N\}$, $t\in[0,T]$ and $n,N\in\mathbb{N}$. Furthermore, by using Young's inequality, 
\begin{align} \label{eq:T1}
T_1 \leq K+ K \int_0^t  \mathbb{E} |x^{i,N,n}_{\kappa_n(s)}|^{p_0} ds + K  \int_0^t  \mathbb{E}\mathcal{W}_2^{p_0} (\mu_{\kappa_n(s)}^{x,N,n}, \delta_0)  ds 
\end{align}
 for any $i\in\{1,\ldots, N\}$, $t\in[0,T]$ and $n,N\in\mathbb{N}$.

For $T_2$, we use the following remainder formula, 
\begin{align} \label{eq:remain:T2}
|y|^{p_0-2}=|a|^{p_0-2} +(p_0-2) \int_0^1 |a+\theta(y-a)|^{p_0-4} (y-a)(a+\theta (y-a)) d\theta
\end{align}
for any $y, a \in \mathbb{R}^d$ along with the Cauchy-Schwarz inequality and Assumption B--\ref{asb:growth:n} to obtain the following estimate, 
\begin{align*}
T_2& := \frac{p_0(p_0-1)}{2} \mathbb{E} \int_0^t \big\{|x^{i,N,n}_s|^{p_0-2}-|x^{i,N,n}_{\kappa_n(s)}|^{p_0-2}\big\} |\sigma_{\kappa_n(s)}^n (x^{i,N,n}_{\kappa_n(s)}, \mu_{\kappa_n(s)}^{x,N,n})|^2 ds  
\\
& = \frac{p_0(p_0-1)(p_0-2)}{2}  \mathbb{E} \int_0^t \int_0^1 |x^{i,N,n}_{\kappa_n(s)} +\theta(x^{i,N,n}_s-x^{i,N,n}_{\kappa_n(s)})|^{p_0-4} 
\\
& \qquad \qquad \times (x^{i,N,n}_s-x^{i,N,n}_{\kappa_n(s)})(x^{i,N,n}_{\kappa_n(s)}+\theta (x^{i,N,n}_s-x^{i,N,n}_{\kappa_n(s)})) d\theta |\sigma_{\kappa_n(s)}^n (x^{i,N,n}_{\kappa_n(s)}, \mu_{\kappa_n(s)}^{x,N,n})|^2 ds
\\
& \leq K \mathbb{E} \int_0^t \big\{  |x^{i,N,n}_{\kappa_n(s)}| +|x^{i,N,n}_s-x^{i,N,n}_{\kappa_n(s)}|\}^{p_0-3}  |x^{i,N,n}_s-x^{i,N,n}_{\kappa_n(s)}| n^{1/3}\Big\{\bar{M}^n_s+|x^{i,N,n}_{\kappa_n(s)}| + \mathcal{W}_2( \mu_{\kappa_n(s)}^{x,N,n}, \delta_0)\Big\}^2 ds
\\
& \leq K \mathbb{E} \int_0^t |x^{i,N,n}_{\kappa_n(s)}|^{p_0-3}  n^{1/3}\Big\{\bar{M}^n_s+|x^{i,N,n}_{\kappa_n(s)}| + \mathcal{W}_2( \mu_{\kappa_n(s)}^{x,N,n}, \delta_0)\Big\}^2   \mathbb{E}\big(|x^{i,N,n}_s-x^{i,N,n}_{\kappa_n(s)}||\mathscr{F}_{\kappa_n(s)}\big) ds
\\
& \quad + K \mathbb{E} \int_0^t  n^{1/3}\Big\{\bar{M}^n_s+|x^{i,N,n}_{\kappa_n(s)}| + \mathcal{W}_2( \mu_{\kappa_n(s)}^{x,N,n}, \delta_0)\Big\}^2 \mathbb{E}\big(|x^{i,N,n}_s-x^{i,N,n}_{\kappa_n(s)}|^{p_0-2}|\mathscr{F}_{\kappa_n(s)} \big)ds,
\end{align*}
which on the application of Lemma \ref{lem:one-step:mb}  and Young's inequality yields
\begin{align} \label{eq:T2}
T_2& \leq K \mathbb{E} \int_0^t |x^{i,N,n}_{\kappa_n(s)}|^{p_0-3}  \Big\{\bar{M}^n_s+|x^{i,N,n}_{\kappa_n(s)}| + \mathcal{W}_2( \mu_{\kappa_n(s)}^{x,N,n}, \delta_0)\Big\}^3    ds \notag
\\
& \quad + Kn^{-1/3} \mathbb{E} \int_0^t  \Big\{\bar{M}^n_s+|x^{i,N,n}_{\kappa_n(s)}| + \mathcal{W}_2( \mu_{\kappa_n(s)}^{x,N,n}, \delta_0)\Big\}^{p_0} ds \notag
\\
& \leq  K + K \int_0^t \mathbb{E}  |x^{i,N,n}_{\kappa_n(s)}|^{p_0} ds + K \int_0^t  \mathbb{E} \mathcal{W}_2^{p_0}(\mu_{\kappa_n(s)}^{x,N,n}, \delta_0) ds 
\end{align}
 for any $i\in\{1,\ldots, N\}$, $t\in[0,T]$ and $n,N\in\mathbb{N}$.

For $T_3$, one uses equation \eqref{eq:remain:T2} and the Cauchy-Schwarz inequality to obtain the following, 
\begin{align*}
T_3&:= p_0 (p_0-1) \mathbb{E}  \int_0^t\int_Z  |\gamma_{\kappa_n(s)}^n (x^{i,N,n}_{\kappa_n(s)}, \mu_{\kappa_n(s)}^{x,N,n}, z)|^2 \int_0^1 (1-\theta)\Big\{ |x^{i,N,n}_s+\theta \gamma_{\kappa_n(s)}^n (x^{i,N,n}_{\kappa_n(s)}, \mu_{\kappa_n(s)}^{x,N,n}, z)| ^{p_0-2}
\\
& \qquad \qquad \qquad \qquad  - |x^{i,N,n}_{\kappa_n(s)}+\theta \gamma_{\kappa_n(s)}^n (x^{i,N,n}_{\kappa_n(s)}, \mu_{\kappa_n(s)}^{x,N,n}, z) |^{p_0-2} \Big\} d\theta  \nu(dz) ds
\\
& =  p_0 (p_0-1)  (p_0-2) \mathbb{E} \int_0^t\int_Z  |\gamma_{\kappa_n(s)}^n (x^{i,N,n}_{\kappa_n(s)}, \mu_{\kappa_n(s)}^{x,N,n}, z)|^2 
\\
& \qquad \times \int_0^1 (1-\theta) \int_0^1 |x^{i,N,n}_{\kappa_n(s)}+\theta \gamma_{\kappa_n(s)}^n (x^{i,N,n}_{\kappa_n(s)}, \mu_{\kappa_n(s)}^{x,N,n}, z)+\bar{\theta}(x^{i,N,n}_s-x^{i,N,n}_{\kappa_n(s)}) |^{p_0-4} (x^{i,N,n}_s-x^{i,N,n}_{\kappa_n(s)}) 
\\
& \qquad  \times \big(x^{i,N,n}_{\kappa_n(s)}+\theta \gamma_{\kappa_n(s)}^n (x^{i,N,n}_{\kappa_n(s)}, \mu_{\kappa_n(s)}^{x,N,n}, z)+\bar{\theta}(x^{i,N,n}_s-x^{i,N,n}_{\kappa_n(s)})\big) d\bar{\theta} d\theta  \nu(dz) ds
\\
& \leq K \mathbb{E} \int_0^t\int_Z  |\gamma_{\kappa_n(s)}^n (x^{i,N,n}_{\kappa_n(s)}, \mu_{\kappa_n(s)}^{x,N,n}, z)|^2 \big\{ |x^{i,N,n}_{\kappa_n(s)}|+| \gamma_{\kappa_n(s)}^n (x^{i,N,n}_{\kappa_n(s)}, \mu_{\kappa_n(s)}^{x,N,n}, z)|+|x^{i,N,n}_s-x^{i,N,n}_{\kappa_n(s)}|\big\}^{p_0-3}
\\
& \qquad \times  |x^{i,N,n}_s-x^{i,N,n}_{\kappa_n(s)}|   \nu(dz) ds
\\
& \leq K \mathbb{E} \int_0^t\int_Z |\gamma_{\kappa_n(s)}^n (x^{i,N,n}_{\kappa_n(s)}, \mu_{\kappa_n(s)}^{x,N,n}, z)|^2    \nu(dz) |x^{i,N,n}_{\kappa_n(s)}|^{p_0-3} \mathbb{E}\big( |x^{i,N,n}_s-x^{i,N,n}_{\kappa_n(s)}||\mathscr{F}_{\kappa_n(s)}\big)  ds
\\
& \qquad +  K \mathbb{E} \int_0^t\int_Z |\gamma_{\kappa_n(s)}^n (x^{i,N,n}_{\kappa_n(s)}, \mu_{\kappa_n(s)}^{x,N,n}, z)|^{p_0-1}   \nu(dz) \mathbb{E}\big( |x^{i,N,n}_s-x^{i,N,n}_{\kappa_n(s)}||\mathscr{F}_{\kappa_n(s)}\big) ds
\\
&\qquad + K \mathbb{E} \int_0^t\int_Z  |\gamma_{\kappa_n(s)}^n (x^{i,N,n}_{\kappa_n(s)}, \mu_{\kappa_n(s)}^{x,N,n}, z)|^2 \nu(dz) \mathbb{E}\big( |x^{i,N,n}_s-x^{i,N,n}_{\kappa_n(s)}|^{p_0-2}|\mathscr{F}_{\kappa_n(s)}\big) ds,
\end{align*}
and then one uses Lemma \ref{lem:one-step:mb} and Assumption $B$--\ref{asb:growth:n} to obtain
\begin{align*}
T_3  \leq & K \mathbb{E} \int_0^t n^{1/3} \Big\{\bar{M}^n_s+|x^{i,N,n}_{\kappa_n(s)}| + \mathcal{W}_2( \mu_{\kappa_n(s)}^{x,N,n}, \delta_0)\Big\}^2 |x^{i,N,n}_{\kappa_n(s)}|^{p_0-3}    \mathbb{E}\big( |x^{i,N,n}_s-x^{i,N,n}_{\kappa_n(s)}||\mathscr{F}_{\kappa_n(s)}\big)  ds
\\
& + K \mathbb{E} \int_0^t n^{1/3} \Big\{ \bar{M}^n_s+|x^{i,N,n}_{\kappa_n(s)}| + \mathcal{W}_2( \mu_{\kappa_n(s)}^{x,N,n}, \delta_0) \Big\}^{p_0-1}     \mathbb{E}\big( |x^{i,N,n}_s-x^{i,N,n}_{\kappa_n(s)}||\mathscr{F}_{\kappa_n(s)}\big)  ds
\\
& + K \mathbb{E} \int_0^t n^{1/3} \Big\{\bar{M}^n_s+|x^{i,N,n}_{\kappa_n(s)}| + \mathcal{W}_2( \mu_{\kappa_n(s)}^{x,N,n}, \delta_0)\Big\}^{2}   \mathbb{E}\big( |x^{i,N,n}_s-x^{i,N,n}_{\kappa_n(s)}|^{p_0-2}|\mathscr{F}_{\kappa_n(s)}\big)  ds
\\
 \leq & K \mathbb{E} \int_0^t  \Big\{\bar{M}^n_s+|x^{i,N,n}_{\kappa_n(s)}| + \mathcal{W}_2( \mu_{\kappa_n(s)}^{x,N,n}, \delta_0)\Big\}^3 |x^{i,N,n}_{\kappa_n(s)}|^{p_0-3}     ds
\\
& + K \mathbb{E} \int_0^t  \Big\{\bar{M}^n_s+|x^{i,N,n}_{\kappa_n(s)}| + \mathcal{W}_2( \mu_{\kappa_n(s)}^{x,N,n}, \delta_0)\Big\}^{p_0}       ds
\\
& + K \mathbb{E} \int_0^t n^{-1/3} \Big\{\bar{M}^n_s+|x^{i,N,n}_{\kappa_n(s)}| + \mathcal{W}_2( \mu_{\kappa_n(s)}^{x,N,n}, \delta_0)\Big\}^{p_0}    ds
\end{align*}
 for any $i\in\{1,\ldots, N\}$, $t\in[0,T]$ and $n,N\in\mathbb{N}$. Further, due to Young's inequality, one obtains
\begin{align} \label{eq:T3}
T_3 \leq K+ K\int_0^t    \mathbb{E} |x^{i,N,n}_{\kappa_n(s)}|^{p_0} ds + K \int_0^t   \mathbb{E} \mathcal{W}_2^{p_0} (\mu_{\kappa_n(s)}^{x,N,n}, \delta_0)  ds 
\end{align}
 for any $i\in\{1,\ldots, N\}$, $t\in[0,T]$ and $n,N\in\mathbb{N}$.
 
 Further, notice that
\begin{align} \label{eq:w2:mb}
\mathbb{E}\mathcal{W}_2^{p_0} (\mu_{\kappa_n(s)}^{x,N,n}, \delta_0)= \mathbb{E} \Big(\frac{1}{N} \sum_{j=1}^N |x^{i,N,n}_{\kappa_n(s)}|^2\Big)^{p_0/2} \leq \frac{1}{N} \sum_{j=1}^N \mathbb{E}|x^{i,N,n}_{\kappa_n(s)}|^{p_0} 
\end{align}
 for any $s\in[0,T]$ and $n,N\in\mathbb{N}$.
 
Thus, on combining the estimates from equations \eqref{eq:T1}, \eqref{eq:T2} and \eqref{eq:T3} in equation \eqref{eq:T1+T2+T3} and then using equation \eqref{eq:w2:mb}, one obtains
\begin{align*}
\sup_{i\in\{1,\ldots,N\}}\sup_{s\in[0,t]}\mathbb{E}|  x^{i,N,n}_t|^{p_0}   \leq K+ K\int_0^t   \sup_{i\in \{1,\ldots,N\}}\sup_{r\in[0,s]} \mathbb{E} |x^{i,N,n}_{\kappa_n(r)}|^{p_0} ds 
\end{align*}
for any $t\in[0,T]$ and $n,N\in\mathbb{N}$. The application of Gr\"onwall's lemma completes the proof. 
\end{proof}

The following corollary is a consequence of the above lemma.  
\begin{cor} \label{cor:one-step:rate}
Let Assumptions $A$--\ref{as:initial}, $B$--\ref{asb:coercivity:scheme}  and B--\ref{asb:growth:n} hold.   
 Then, 
\begin{align*}
\mathbb{E}|x_t^{i, N, n}- x_{\kappa_n(t)}^{i, N, n}|^{q}  \leq  K
\begin{cases}
 n^{-1},  & \mbox{ if } 2\leq q \leq p_0, 
\\
 n^{-q/2},  & \mbox{ if }  1\leq q \leq 2,
\end{cases}
\end{align*}
for any $i\in\{1,\ldots,N\}$, $t\in[0,T]$ and $n,N\in\mathbb{N}$, where $K>0$ does not depend on $n$ and $N$. 
\end{cor} 

The following theorem gives the rate of convergence of the tamed Euler scheme \eqref{eq:scheme}.  
\begin{thm}
\label{thm:order-12}
Let Assumptions $A$--\ref{as:initial}, $A$--\ref{as:coercivity}, $A$--\ref{as:coercivity:p0}  to A--\ref{as:bounded}, $B$--\ref{asb:coercivity:scheme}, B--\ref{asb:growth:n} and B--\ref{asb:diff:rate} be satisfied. Then, 
the tamed Euler scheme  \eqref{eq:scheme} converges in mean-square sense to the true solution of the interacting particle system  \eqref{eq:interacting} associated to the McKean--Vlasov SDE \eqref{eq:sde}, and for any $\epsilon>0$ such that $p_0 \ge \chi (2+\epsilon)/\epsilon$ we have
\begin{align*}
\sup_{i\in\{1,\ldots,N\}} \sup_{t\in[0,T]}\mathbb{E}|x_t^{i,N}-x_t^{i,N,n}|^2 \leq K n^{-2/(2+\epsilon)}
\end{align*}
for any $n, N\in\mathbb{N}$, where the constant $K>0$ does not depend on $n$ and $N$. 
\end{thm}
\begin{proof}
On using It\^o's formula, 
 \begin{align}
 \mathbb{E}|x_t^{i,N} & -x_t^{i,N,n}|^2=\mathbb{E}|x_0^{i,N}-x_0^{i,N,n}|^2 + 2 \mathbb{E} \int_0^t (x_s^{i,N}-x_s^{i,N,n}) (b_s(x_s^{i, N}, \mu_s^{x,N})- b^n_{\kappa_n(s)}(x_{\kappa_n(s)}^{i, N, n}, \mu_{\kappa_n(s)}^{x,N, n})) ds   \notag
 \\
 & + 2 \mathbb{E} \int_0^t (x_s^{i,N}-x_s^{i,N,n}) (\sigma_s(x_s^{i, N}, \mu_s^{x,N})- \sigma^n_{\kappa_n(s)}(x_{\kappa_n(s)}^{i, N, n}, \mu_{\kappa_n(s)}^{x,N, n})) dw_s^i   \notag
 \\
 & + 2 \mathbb{E} \int_0^t \int_Z (x_s^{i,N}-x_s^{i,N,n}) (\gamma_s(x_s^{i, N}, \mu_s^{x,N}, z)- \gamma^n_{\kappa_n(s)}(x_{\kappa_n(s)}^{i, N, n}, \mu_{\kappa_n(s)}^{x,N, n}, z) ) \tilde{n}_p \notag (ds, dz)  
 \\
  & +  \mathbb{E} \int_0^t  |\sigma_s(x_s^{i, N}, \mu_s^{x,N})- \sigma^n_{\kappa_n(s)}(x_{\kappa_n(s)}^{i, N, n}, \mu_{\kappa_n(s)}^{x,N, n})|^2 ds   \notag
  \\
  & + \mathbb{E} \int_0^t \int_Z ( |x_s^{i,N}  -x_s^{i,N,n}+ \gamma_s(x_s^{i, N}, \mu_s^{x,N}, z)- \gamma^n_{\kappa_n(s)}(x_{\kappa_n(s)}^{i, N, n}, \mu_{\kappa_n(s)}^{x,N, n}, z) |^2 - |x_s^{i,N}  -x_s^{i,N,n}|^2 \notag
  \\
 & \qquad - 2 (x_s^{i,N}  -x_s^{i,N,n})(\gamma_s(x_s^{i, N}, \mu_s^{x,N}, z)- \gamma^n_{\kappa_n(s)}(x_{\kappa_n(s)}^{i, N, n}, \mu_{\kappa_n(s)}^{x,N, n}, z)) n_p(ds,dz) \notag
 \\
  = &~\mathbb{E}|x_0^{i,N}-x_0^{i,N,n}|^2 + 2 \mathbb{E} \int_0^t (x_s^{i,N}-x_s^{i,N,n}) (b_s(x_s^{i, N}, \mu_s^{x,N})- b^n_{\kappa_n(s)}(x_{\kappa_n(s)}^{i, N, n}, \mu_{\kappa_n(s)}^{x,N, n})) ds  \notag
 \\
  & +  \mathbb{E} \int_0^t  |\sigma_s(x_s^{i, N}, \mu_s^{x,N})- \sigma^n_{\kappa_n(s)}(x_{\kappa_n(s)}^{i, N, n}, \mu_{\kappa_n(s)}^{x,N, n})|^2 ds  \notag
  \\
  & + \mathbb{E} \int_0^t \int_Z |\gamma_s(x_s^{i, N}, \mu_s^{x,N}, z)- \gamma^n_{\kappa_n(s)}(x_{\kappa_n(s)}^{i, N, n}, \mu_{\kappa_n(s)}^{x,N, n}, z) |^2  \nu(dz)ds \label{eq:ito:rate}
 \end{align}
 for any $i\in\{1,\ldots,N\}$, $t\in[0,T]$ and $n,N\in\mathbb{N}$. Notice that for any $d_1 \times m_1$ matrices $A$ and $B$, one uses Cauchy-Schwarz and Young's inequalities to obtain
 \begin{align*}
 |A+B|^2 & = |A|^2 + |B|^2 +2 \sum_{j=1}^{d_1}\sum_{k=1}^{m_1}  A_{jk} B_{jk} \leq |A|^2 + |B|^2 +2|A||B| 
 \\
 & \leq |A|^2 + |B|^2 + (\eta-1)|A|^2 + \frac{|B|^2}{4(\eta-1)} = p_1 |A|^2 + \frac{4\eta-3}{4(\eta-1)}|B|^2,
\end{align*}   
which thus yields the following estimates, 
 \begin{align}
 |\sigma_s&(x_s^{i, N},   \mu_s^{x,N})-\sigma^n_{\kappa_n(s)}(x_{\kappa_n(s)}^{i, N, n}, \mu_{\kappa_n(s)}^{x,N, n})|^2   \notag
 \\
 =  & |\sigma_s(x_s^{i, N},  \mu_s^{x,N})-\sigma_s(x_s^{i, N,n}, \mu_s^{x,N,n})+\sigma_s(x_s^{i, N,n}, \mu_s^{x,N,n})- \sigma^n_{\kappa_n(s)}(x_{\kappa_n(s)}^{i, N, n}, \mu_{\kappa_n(s)}^{x,N, n})|^2 \notag
 \\
  \leq & \, \eta |\sigma_s(x_s^{i, N},  \mu_s^{x,N})-\sigma_s(x_s^{i, N,n}, \mu_s^{x,N,n})|^2 + K |\sigma_s(x_s^{i, N,n}, \mu_s^{x,N,n})- \sigma^n_{\kappa_n(s)}(x_{\kappa_n(s)}^{i, N, n}, \mu_{\kappa_n(s)}^{x,N, n})|^2 \label{eq:sig:rate}
 \\
 \int_Z & |\gamma_s( x_s^{i, N},  \mu_s^{x,N}, z)- \gamma^n_{\kappa_n(s)}(x_{\kappa_n(s)}^{i, N, n}, \mu_{\kappa_n(s)}^{x,N, n}, z)|^2 \nu(dz)  \notag
 \\
 = & \int_Z |\gamma_s(x_s^{i, N},  \mu_s^{x,N}, z)-\gamma_s(x_s^{i, N,n}, \mu_s^{x,N,n},z)  +\gamma_s(x_s^{i, N,n}, \mu_s^{x,N,n},z)- \gamma^n_{\kappa_n(s)}(x_{\kappa_n(s)}^{i, N, n}, \mu_{\kappa_n(s)}^{x,N, n}, z)|^2 \nu(dz) \notag
 \\
  \leq & \eta \int_Z|\gamma_s(x_s^{i, N},  \mu_s^{x,N}, z)-\gamma_s(x_s^{i, N,n}, \mu_s^{x,N,n}, z)|^2 \nu(dz)  \notag
 \\
 & + K \int_Z |\gamma_s(x_s^{i, N,n}, \mu_s^{x,N,n}, z)- \gamma^n_{\kappa_n(s)}(x_{\kappa_n(s)}^{i, N, n}, \mu_{\kappa_n(s)}^{x,N, n}, z)|^2 \nu(dz) \label{eq:gam:rate}
 \end{align}
 almost surely for any $i\in\{1,\ldots,N\}$, $t\in[0,T]$ and $n,N\in\mathbb{N}$. Also, 
 \begin{align}
 (b_s(x_s^{i, N}, \mu_s^{x,N}) & - b^n_{\kappa_n(s)}(x_{\kappa_n(s)}^{i, N, n}, \mu_{\kappa_n(s)}^{x,N, n}))= (b_s(x_s^{i, N}, \mu_s^{x,N})- b_s(x_s^{i, N,n}, \mu_s^{x,N,n})) \notag
 \\
 & +(b_s(x_s^{i, N,n}, \mu_s^{x,N,n})- b^n_{\kappa_n(s)}(x_{\kappa_n(s)}^{i, N, n}, \mu_{\kappa_n(s)}^{x,N, n})) \label{eq:b:rate}
 \end{align}
 almost surely for any $i\in\{1,\ldots,N\}$, $t\in[0,T]$ and $n,N\in\mathbb{N}$. On substituting the values from equations \eqref{eq:sig:rate}, \eqref{eq:gam:rate} and \eqref{eq:b:rate} in equation \eqref{eq:ito:rate}, one obtains
 \begin{align*}
  \mathbb{E}| & ~x_t^{i,N}  -x_t^{i,N,n}|^2 \leq \mathbb{E}|x_0^{i,N}-x_0^{i,N,n}|^2 +  \mathbb{E} \int_0^t \Big\{ 2 (x_s^{i, N}-x_s^{i, N,n}) (b_s(x_s^{i, N}, \mu_s^{x,N})- b_s(x_s^{i, N,n}, \mu_s^{x,N,n}))   
  \\
  & + \eta |\sigma_s(x_s^{i, N},  \mu_s^{x,N})-\sigma_s(x_s^{i, N,n}, \mu_s^{x,N,n})|^2 + \eta \int_Z|\gamma_s(x_s^{i, N},  \mu_s^{x,N}, z)-\gamma_s(x_s^{i, N,n}, \mu_s^{x,N,n}, z)|^2 \nu(dz) \Big\} ds 
  \\
  & +2 \mathbb{E} \int_0^t (x_s^{i, N}-x_s^{i, N,n})(b_s(x_s^{i, N,n}, \mu_s^{x,N,n})- b^n_{\kappa_n(s)}(x_{\kappa_n(s)}^{i, N, n}, \mu_{\kappa_n(s)}^{x,N, n})) ds 
  \\
   & +K \mathbb{E} \int_0^t  |\sigma_s(x_s^{i, N,n}, \mu_s^{x,N,n})- \sigma^n_{\kappa_n(s)}(x_{\kappa_n(s)}^{i, N, n}, \mu_{\kappa_n(s)}^{x,N, n})|^2 ds 
   \\
   & +K \mathbb{E} \int_0^t  \int_Z |\gamma_s(x_s^{i, N,n}, \mu_s^{x,N,n}, z)- \gamma^n_{\kappa_n(s)}(x_{\kappa_n(s)}^{i, N, n}, \mu_{\kappa_n(s)}^{x,N, n}, z)|^2 \nu(dz) ds 
 \end{align*}
 which on using Assumption A--\ref{as:monotonicity:rate}, Cauchy--Schwarz and Young's inequalities, yields
 \begin{align}
   \mathbb{E}| x_t^{i,N} & -x_t^{i,N,n}|^2 \leq   \mathbb{E}|x_0^{i,N}-x_0^{i,N,n}|^2 +\int_0^t \mathbb{E}| x_s^{i,N}  -x_s^{i,N,n}|^2 ds \notag
   \\
 & + K \mathbb{E} \int_0^t |b_s(x_s^{i, N,n}, \mu_s^{x,N,n})- b^n_{\kappa_n(s)}(x_{\kappa_n(s)}^{i, N, n}, \mu_{\kappa_n(s)}^{x,N, n})|^2 ds \notag
  \\
   & +K \mathbb{E} \int_0^t  |\sigma_s(x_s^{i, N,n}, \mu_s^{x,N,n})- \sigma^n_{\kappa_n(s)}(x_{\kappa_n(s)}^{i, N, n}, \mu_{\kappa_n(s)}^{x,N, n})|^2 ds \notag
   \\
   & +K \mathbb{E} \int_0^t  \int_Z |\gamma_s(x_s^{i, N,n}, \mu_s^{x,N,n}, z)- \gamma^n_{\kappa_n(s)}(x_{\kappa_n(s)}^{i, N, n}, \mu_{\kappa_n(s)}^{x,N, n}, z)|^2 \nu(dz) ds \notag
   \\
   =:& ~ \mathbb{E}|x_0^{i,N}-x_0^{i,N,n}|^2+\int_0^t \mathbb{E}| x_s^{i,N}  -x_s^{i,N,n}|^2 ds +E_1+E_2+E_3 \label{eq:E1+E2+E3}
 \end{align}
  almost surely for any $i\in\{1,\ldots,N\}$, $t\in[0,T]$ and $n,N\in\mathbb{N}$. 
  
  For $E_1$, one writes
  \begin{align*}
  E_1 := & ~K \mathbb{E} \int_0^t |b_s(x_s^{i, N,n}, \mu_s^{x,N,n})- b^n_{\kappa_n(s)}(x_{\kappa_n(s)}^{i, N, n}, \mu_{\kappa_n(s)}^{x,N, n})|^2 ds 
  \\
   \leq & ~ K \mathbb{E} \int_0^t |b_s(x_s^{i, N,n}, \mu_s^{x,N,n})- b_s(x_{\kappa_n(s)}^{i, N,n}, \mu_{\kappa_n(s)}^{x,N,n})|^2 ds 
  \\
  & + K \mathbb{E} \int_0^t |b_s(x_{\kappa_n(s)}^{i, N,n}, \mu_{\kappa_n(s)}^{x,N,n})- b_{\kappa_n(s)}(x_{\kappa_n(s)}^{i, N, n}, \mu_{\kappa_n(s)}^{x,N, n})|^2 ds
  \\
  & + K \mathbb{E} \int_0^t |b_{\kappa_n(s)}(x_{\kappa_n(s)}^{i, N,n}, \mu_{\kappa_n(s)}^{x,N,n})- b^n_{\kappa_n(s)}(x_{\kappa_n(s)}^{i, N, n}, \mu_{\kappa_n(s)}^{x,N, n})|^2 ds,
  \end{align*}
  and then the application of Assumptions A--\ref{as:poly:Lips:b} and A--\ref{as:holder:time} yields the following estimates, 
  \begin{align*}
  E_1  \leq & ~ K \mathbb{E} \int_0^T \big\{(1+|x_s^{i, N,n}|+ |x_{\kappa_n(s)}^{i, N,n}|)^{\chi} |x_s^{i, N,n}-x_{\kappa_n(s)}^{i, N,n}|^2 + \mathcal{W}_2^2 (\mu_s^{x,N,n}, \mu_{\kappa_n(s)}^{x,N,n}) \big\} ds 
  \\
  & + K \mathbb{E} \int_0^T |s- \kappa_n(s)|^2 ds + K \mathbb{E} \int_0^T |b_{\kappa_n(s)}(x_{\kappa_n(s)}^{i, N,n}, \mu_{\kappa_n(s)}^{x,N,n})- b^n_{\kappa_n(s)}(x_{\kappa_n(s)}^{i, N, n}, \mu_{\kappa_n(s)}^{x,N, n})|^2 ds,
  \end{align*}
  which on using H\"older's inequality and Assumption B--\ref{asb:diff:rate} gives
  \begin{align*}
  E_1  \leq & ~ K \int_0^t \big\{\mathbb{E} (1+|x_s^{i, N,n}|+ |x_{\kappa_n(s)}^{i, N,n}|)^{\chi(2+\epsilon)/\epsilon}\big\}^{\epsilon/(2+\epsilon)} \big\{\mathbb{E}|x_s^{i, N,n}-x_{\kappa_n(s)}^{i, N,n}|^{2+\epsilon}\}^{2/(2+\epsilon)} ds 
  \\
  &+ K n^{-1} + K \int_0^t\mathbb{E}  \mathcal{W}_2^2 (\mu_s^{x,N,n}, \mu_{\kappa_n(s)}^{x,N,n})  ds
  \end{align*}
  for any $i\in\{1,\ldots,N\}$, $t\in[0,T]$ and $n,N\in\mathbb{N}$. Also, notice that
  \begin{align} \label{eq:w2:rate:one-step}
  \mathbb{E}  \mathcal{W}_2^2 (\mu_s^{x,N,n}, \mu_{\kappa_n(s)}^{x,N,n})= \frac{1}{N} \sum_{j=1}^N  \mathbb{E}|x_s^{j, N,n}-x_{\kappa_n(s)}^{j, N,n}|^2 
  \end{align}
  and thus, due to Lemma \ref{lem:scheme:mb}, Corollary \ref{cor:one-step:rate} and equation \eqref{eq:w2:rate:one-step}, one has
  \begin{align} \label{eq:E1}
   E_1  \leq K n^{-2/(2+\epsilon)}
  \end{align}
    for any $i\in\{1,\ldots,N\}$, $t\in[0,T]$ and $n,N\in\mathbb{N}$.
   
   Furthermore, one can write 
\begin{align*}
E_2 & := K \mathbb{E} \int_0^t  |\sigma_s(x_s^{i, N,n}, \mu_s^{x,N,n})- \sigma^n_{\kappa_n(s)}(x_{\kappa_n(s)}^{i, N, n}, \mu_{\kappa_n(s)}^{x,N, n})|^2 ds \notag
\\
& \leq  K \mathbb{E} \int_0^t  |\sigma_s(x_s^{i, N,n}, \mu_s^{x,N,n})- \sigma_s(x_{\kappa_n(s)}^{i, N, n}, \mu_{\kappa_n(s)}^{x,N, n})|^2 ds
\\
& \quad +  K \mathbb{E} \int_0^t |\sigma_s(x_{\kappa_n(s)}^{i, N, n}, \mu_{\kappa_n(s)}^{x,N, n}) - \sigma_{\kappa_n(s)}(x_{\kappa_n(s)}^{i, N, n}, \mu_{\kappa_n(s)}^{x,N, n})|^2 ds  
\\
 & \quad + K \mathbb{E} \int_0^t  |\sigma_s(x_s^{i, N,n}, \mu_s^{x,N,n})- \sigma^n_{\kappa_n(s)}(x_{\kappa_n(s)}^{i, N, n}, \mu_{\kappa_n(s)}^{x,N, n})|^2 ds,
\end{align*}    
 which on the application of Remark \ref{rem:poly:lips:sig:gam}, Assumptions A--\ref{as:holder:time} and B--\ref{asb:diff:rate} yields
 \begin{align*}
 E_2  \leq & ~ K \int_0^t \big\{\mathbb{E} (1+|x_s^{i, N,n}|+ |x_{\kappa_n(s)}^{i, N,n}|)^{\chi(2+\epsilon)/\epsilon}\big\}^{\epsilon/(2+\epsilon)} \big\{\mathbb{E}|x_s^{i, N,n}-x_{\kappa_n(s)}^{i, N,n}|^{2+\epsilon}\}^{2/(2+\epsilon)} ds 
  \\
  &+ K n^{-1} + K \int_0^t\mathbb{E}  \mathcal{W}_2^2 (\mu_s^{x,N,n}, \mu_{\kappa_n(s)}^{x,N,n})  ds,
 \end{align*}
 and hence due to Lemma \ref{lem:scheme:mb} and Lemma \ref{lem:one-step:mb}, one obtains
 \begin{align} \label{eq:E2}
  E_2 \leq K n^{-2/(2+\epsilon)}
\end{align}     
  for any $i\in\{1,\ldots,N\}$, $t\in[0,T]$ and $n,N\in\mathbb{N}$.
  
   Finally, one estimates $E_3$ by
    \begin{align*}
    E_3 & := K \mathbb{E} \int_0^t  \int_Z |\gamma_s(x_s^{i, N,n}, \mu_s^{x,N,n}, z)- \gamma^n_{\kappa_n(s)}(x_{\kappa_n(s)}^{i, N, n}, \mu_{\kappa_n(s)}^{x,N, n}, z)|^2 \nu(dz) ds \notag
    \\
   & = K \mathbb{E} \int_0^t  \int_Z |\gamma_s(x_s^{i, N,n}, \mu_s^{x,N,n}, z)- \gamma_{s}(x_{\kappa_n(s)}^{i, N, n}, \mu_{\kappa_n(s)}^{x,N, n}, z)|^2 \nu(dz) ds
   \\
   & \quad + K \mathbb{E} \int_0^t  \int_Z |\gamma_s(x_{\kappa_n(s)}^{i, N, n}, \mu_{\kappa_n(s)}^{x,N, n},  z)- \gamma_{\kappa_n(s)}(x_{\kappa_n(s)}^{i, N, n}, \mu_{\kappa_n(s)}^{x,N, n}, z)|^2 \nu(dz) ds
   \\
   & \quad + K \mathbb{E} \int_0^t  \int_Z |\gamma_{\kappa_n(s)}(x_{\kappa_n(s)}^{i, N, n}, \mu_{\kappa_n(s)}^{x,N, n},  z)- \gamma_{\kappa_n(s)}^n(x_{\kappa_n(s)}^{i, N, n}, \mu_{\kappa_n(s)}^{x,N, n}, z)|^2 \nu(dz) ds
    \end{align*}
and then uses Remark \ref{rem:poly:lips:sig:gam}, Assumptions A--\ref{as:holder:time} and B--\ref{asb:diff:rate} along with Lemma \ref{lem:scheme:mb} and Lemma \ref{lem:one-step:mb} to obtain 
      \begin{align} 
    E_3  \leq K n^{-2/(2+\epsilon)} \label{eq:E3}
  \end{align} 
  for any $i\in\{1,\ldots,N\}$, $t\in[0,T]$ and $n,N\in\mathbb{N}$. On substituting the estimates from equations \eqref{eq:E1}, \eqref{eq:E2} and \eqref{eq:E3} in equation \eqref{eq:E1+E2+E3}, 
  \begin{align*}
  \mathbb{E}| x_t^{i,N}  -x_t^{i,N,n}|^2 \leq  \mathbb{E}|x_0^{i,N}-x_0^{i,N,n}|^2+\int_0^t \mathbb{E}| x_s^{i,N}  -x_s^{i,N,n}|^2 ds +K n^{-2/(2+\epsilon)} 
\end{align*}   
 for any $i\in\{1,\ldots,N\}$, $t\in[0,T]$ and $n,N\in\mathbb{N}$. The proof is completed on using Gr\"onwall's lemma. 
\end{proof}

\section{Applications}

In recent years, researchers have shown growing interests in the study of McKean--Vlasov stochastic differential equations with Markovian switching (SDEwMS), see for example \cite{bensoussan2020, nguyen2020} and references therein. 
The McKean--Vlasov stochastic delay differential equations (SDDEs) have recently received considerable attention from researchers, see for example, \cite{chen2019maximumprinciple} for optimal control, \cite{fouque2020} for deep learning and references therein.
Motivated by this, we give immediate applications of our results to McKean--Vlasov SDEwMS and McKean--Vlasov SDDEs, both driven by L\'evy noise. 

\subsection{McKean--Vlasov SDEs with Markovian switching driven by L\'evy noise}
Let $m_0$ be a fixed positive integer and  $\alpha:=(\alpha_t;t\geq 0)$ be a continuous-time Markov chain with state space $\mathcal{S}=\{1,2,\ldots,m_0 \}$ and generator $Q=(q_{i_0j_0}; i_0, j_0 \in \mathcal{S})$, such that its transition probabilities are 
\begin{align}  \label{eq:MProb}
\mathbb{P}(\alpha_{t+\delta}=j_0|\alpha_t=i_0)=
\begin{cases}
q_{i_0j_0}\delta+o(\delta), &\text{if } i_0\neq j_0,\\
1+q_{i_0j_0}\delta+o(\delta), &\text{if } i_0= j_0,
\end{cases} 
\end{align}
for any $\delta>0$ where $o(\delta)$ stands for the Bachmann--Landau little--o notation,  $q_{i_0j_0}\geq 0$,  for any $i_0\neq j_0\in \mathcal{S}$ and $q_{i_0i_0}=-\sum_{j_0\neq i_0} q_{i_0j_0}$ for any $i_0\in\mathcal{S}$. Also, assume that $f:[0, T]\times \mathbb{R}^d\times \mathcal{S} \times \mathcal{P}_2(\mathbb{R}^d) \to \mathbb{R}^d$, $g:[0, T]\times \mathbb{R}^d\times \mathcal{S} \times \mathcal{P}_2(\mathbb{R}^d) \to \mathbb{R}^{d \times m}$ and $h: [0, T]\times \mathbb{R}^d\times \mathcal{S} \times \mathcal{P}_2(\mathbb{R}^d)\times Z \to \mathbb{R}^d$ are measurable functions. Consider the following $d$--dimensional McKean--Vlasov SDE with Markovian switching driven by L\'evy noise,  
\begin{align} \label{eq:sdems}
y_t=y_0+\int^t_0 f_s(y_s,\alpha_s, L_{y_s})ds+\int^t_0 g_s(y_s, \alpha_s, L_{y_s})dw_s +\int^t_0 \int_Z h_s(y_s, \alpha_s, L_{y_s}, z) \tilde{n}_p(ds,dz)
\end{align} 
almost surely for any $t\in[0,T]$, where the initial value $y_0$ is an $\mathscr{F}_0$-measurable random variable taking values in $\mathbb{R}^d$. 
\begin{rem}
If $m_0=1$, \textit{i.e.}, the Markov chain $\alpha$ remains at the same state throughout the interval $[0,T]$, then equation \eqref{eq:sdems} becomes a McKean--Vlasov SDE driven by L\'evy noise. Hence,  all  results obtained in this section  also hold true for  McKean--Vlasov SDEs driven by L\'evy noise.
\end{rem}
Let $\bar{p}_0\geq 2$ be a fixed constant. We make the following assumptions. 
\begin{asAt} \label{ast:initial:ms}
$\mathbb{E}|y_0|^{\bar{p}_0} < \infty$. 
\end{asAt}
\begin{asAt} \label{ast:coercivity:ms}
There exists a constant $L>0$   such that for every $i_0 \in \mathcal{S}$, 
\begin{align*}
2  xf_t(x, i_0,\mu) + |g_t(x, i_0,\mu)|^2 +  \int_Z |h_t(x,i_0,\mu, z)|^2 \nu(dz) \leq L\big\{1+|x|^{2}+\mathcal{W}_2^{2}(\mu,\delta_0) \big\}
\end{align*}
 for any $t\in[0,T]$, $x\in\mathbb{R}^d$  and $\mu \in \mathcal{P}_2(\mathbb{R}^d)$. 
\end{asAt}
\begin{asAt} \label{ast:monotonicity:ms}
There exists a constant $L>0$ such that for every $i_0\in\mathcal{S}$, 
\begin{align*}
2(x-\bar{x})(f_t(x, i_0,\mu)-f_t(\bar{x}, i_0,\bar{\mu})) & + |g_t(x,i_0,\mu)-g_t(\bar{x}, i_0,\bar{\mu})|^2 +\int_Z |h_t(x,i_0,\mu, z)-h_t(\bar{x},i_0,\bar{\mu}, z)|^2 \nu(dz) 
\\
& \leq L\big\{|x-\bar{x}|^2 + \mathcal{W}_2^2(\mu,\bar{\mu}) \big\}
\end{align*}
 for any $t\in[0,T]$, $x, \bar{x}\in\mathbb{R}^d$ and $\mu, \bar{\mu} \in \mathcal{P}_2(\mathbb{R}^d)$. 
\end{asAt}
\begin{asAt} \label{ast:continuity:ms}
For any  $t\in[0,T]$ and $i_0\in\mathcal{S}$,  $f_t(x,i_0,\mu)$ is a continuous function of $x\in\mathbb{R}^d$ and $\mu\in\mathcal{P}_2(\mathbb{R}^d)$. 
\end{asAt}
\begin{asAt} \label{ast:coercivity:ms:p0}
There exists a constant $L>0$   such that for every $i_0 \in \mathcal{S}$, 
\begin{align*}
&2|x|^{\bar{p}_0-2}  xf_t(x, i_0,\mu) +(\bar{p}_0-1) |x|^{\bar{p}_0-2} |g_t(x, i_0,\mu)|^2 
\\
& + 2 (\bar{p}_0-1) \int_Z |h_t(x,i_0,\mu, z)|^2 \int_0^1 (1-\theta)| x+\theta h_t(x,i_0,\mu, z)|^{\bar{p}_0-2} d\theta \nu(dz) \leq L\big\{1+|x|^{\bar{p}_0}+\mathcal{W}_2^{\bar{p}_0}(\mu,\delta_0) \big\}
\end{align*}
 for any $x\in\mathbb{R}^d$, $t\in[0,T]$ and $\mu \in \mathcal{P}_2(\mathbb{R}^d)$. 
\end{asAt}
\begin{thm}[\textbf{Existence, Uniqueness and Moment Bound}] \label{thm:eu:ms}
Let Assumptions $\bar{A}$--\ref{ast:initial:ms} (with $\bar{p}_0=2$), $\bar{A}$--\ref{ast:coercivity:ms}, $\bar{A}$--\ref{ast:monotonicity:ms} and $\bar{A}$--\ref{ast:continuity:ms} be satisfied. Then, there exists a unique $\mathbb{R}^d$-valued c\'ad\'ag process $y$ satisfying the McKean--Vlasov SDEwMS \eqref{eq:sdems} such that 
\begin{align*}
\sup_{t\in[0,T]}\mathbb{E}|y_t|^{2} \leq K,
\end{align*}
where $K:=K(\mathbb{E}|x_0|^{2},m,d,L)$ is a positive constant. 
In addition, if Assumptions $\bar{A}$--\ref{as:initial} and $\bar{A}$--\ref{as:coercivity:p0} hold { for $\bar{p}_0> 2$}, then  
\begin{align*}
\sup_{t\in[0,T]}\mathbb{E}|y_t|^{\bar{p}_0} \leq K,
\end{align*}
where $K:=K(\mathbb{E}|x_0|^{\bar{p}_0},m,d,\bar{p}_0,L)$ is a positive constant. 
\end{thm} 
\begin{proof}
We can directly apply Theorem \ref{thm:eu} to the McKean--Vlasov SDE \eqref{eq:sde} by taking  
\begin{align}\label{eq:coefficients:ms}
b_t(x,\mu)=f_t(x,\alpha_t, \mu),  \, \, \sigma_t(x,\mu)=g_t(x,\alpha_t, \mu), \, \, \gamma_t(x,\mu, z)= h_t(x,\alpha_t,\mu, z)
\end{align} 
almost surely for any $t\in[0,T]$. One can easily verify that Assumptions A--\ref{as:initial} to A--\ref{as:coercivity:p0} follow from  Assumptions $\bar{A}$--\ref{ast:initial:ms} to $\bar{A}$--\ref{ast:coercivity:ms:p0} with $M_t=\bar{M}_t=1$ for any $t\in[0,T]$. 
\end{proof}

Let $w^1,\ldots,w^N$, $N\in\mathbb{N}$, be independent and identically distributed Wiener processes.
For the propagation of chaos, we consider the following non-interacting particle system,
\begin{align} \label{eq:non-interacting:ms}
y_t^{i}=y_0^{i} +\int_{0}^t f_s(y_s^{i}, \alpha_s, L_{y_s^i}) ds + \int_{0}^t g_s(y_s^{i},\alpha_s, L_{y_s^i}) dw_s^i + \int_0^t \int_Z h_s(y_{s}^{i},\alpha_s, L_{y_{s}^i}, z) \tilde{n}_p^i(ds,dz)
\end{align}
and the associated interacting particle system, 
\begin{align} \label{eq:interacting:ms}
y_t^{i,N}=y_0^{i} +\int_{0}^t f_s(y_s^{i, N},\alpha_s, \mu_s^{y,N}) ds + \int_{0}^t g_s(y_s^{i, N}, \alpha_s, \mu_s^{y,N}) dw_s^i + \int_0^t \int_Z h_s(y_{s}^{i, N},\alpha_s, \mu_{s}^{y,N}, z) \tilde{n}_p^i(ds,dz)
\end{align}
almost surely, where  $\mu_t^{y,N}$ is an empirical measure given by
\begin{align*}
\mu_t^{y,N}(\cdot)= \frac{1}{N} \sum_{i=1}^N \delta_{y_t^{i,N}}(\cdot)
\end{align*}
for any $t\in [0,T]$ and $n, N\in \mathbb{N}$. The proof of the result given below follows from Proposition \ref{prop:propagation}.
\begin{prop}[\textbf{Propagation of Chaos}]
If Assumptions $\bar{A}$--\ref{ast:initial:ms} to $\bar{A}$--\ref{ast:coercivity:ms:p0} are satisfied with $\bar{p}_0>4$,  then the interacting particle system \eqref{eq:interacting:ms} is wellposed and converges to the non-interacting particle system \eqref{eq:non-interacting:ms} and 
\[ 
\sup_{i\in\{1,\ldots,N\}}\sup_{t\in[0,T]}\mathbb{E}|y^i_t-y^{i,N}_t|^2\leq K 
\begin{cases}
N^{-1/2}, &  \mbox{ if }  d<4, \\
N^{-1/2} \ln(N), &  \mbox{ if } d=4, \\
N^{-2/d}, &  \mbox{ if } d>4. 
\end{cases}
\]
for any $N\in\mathbb{N}$, where the constant $K>0$ does not depend on $d$ and $N$. 
\end{prop}

Before introducing the tamed Euler scheme, we make the following additional assumptions.
\begin{asAt} \label{ast:coercivity:ms:rate}
There exists a constant $L>0$  such that
\begin{align*}
2|x&|^{\bar{p}_0-2} x f_t(x, i_0, \mu) +(\bar{p}_0-1) |x|^{\bar{p}_0-2} |g_t(x,i_0, \mu)|^2 
\\
& + (\bar{p}_0-1)2^{\bar{p}_0-3} \int_Z \big\{  |h_t(x,i_0,\mu, z)|^2|x|^{\bar{p}_0-2} + |h_t(x,i_0,\mu, z)|^{\bar{p}_0} \big\}  \nu(dz) \leq  L\big\{1+|x|^{\bar{p}_0}+\mathcal{W}_2^{\bar{p}_0}(\mu,\delta_0) \big\}
\end{align*}
for any $t\in[0,T]$, $i_0 \in \mathcal{S}$, $x \in\mathbb{R}^d$ and $\mu  \in \mathcal{P}_2(\mathbb{R}^d)$. 
\end{asAt}
\begin{asAt} \label{ast:monotonicity:rate:ms}
There exist constants $L>0$ and $\eta>1$ such that
\begin{align*}
2(x-\bar{x})&(f_t(x,i_0,\mu) -f_t(\bar{x},i_0,\bar{\mu}))+\eta |g_t(x,i_0,\mu)-g_t(\bar{x}, i_0,\bar{\mu})|^2 
\\
& +\eta \int_Z |h_t(x,i_0,\mu, z)-h_t(\bar{x},i_0,\bar{\mu}, z)|^2 \nu(dz) \leq L\big\{|x-\bar{x}|^2 +\mathcal{W}_2^2 (\mu, \bar{\mu})\big\}
\end{align*}
for any $t\in[0,T]$, $i_0 \in \mathcal{S}$, $x,\bar{x} \in\mathbb{R}^d$ and $\mu, \bar{\mu} \in \mathcal{P}_2(\mathbb{R}^d)$. 
\end{asAt}
\begin{asAt} \label{ast:poly:Lips:b:ms}
There exist  constants $L>0$ and $\chi>0$ such that
\begin{align*}
|f_t(x, i_0, \mu)-f_t(\bar{x}, i_0, \bar{\mu})|^2 \leq L\big\{ (1+|x|+|\bar{x}|)^{\chi} |x-\bar{x}|^2+\mathcal{W}_2^2(\mu,\bar{\mu})\big\}
\end{align*}
for any $t\in[0,T]$, $i_0 \in \mathcal{S}$, $x,\bar{x}\in\mathbb{R}^d$ and $\mu, \bar{\mu}\in\mathcal{P}_2(\mathbb{R}^d)$. 
\end{asAt}
\begin{asAt} \label{ast:coercivity:scheme:ms}
There exists a constant  $L>0$ such that
\begin{align*}
\int_Z |h_t(x,i_0,\mu,z)|^{q}  \nu(dz) \leq L\big\{ 1+|x|^{\chi/2+ q}+\mathcal{W}_2^{q}(\mu,\delta_0) \big\}, \quad 2\leq q \leq \bar{p}_0,
\end{align*}
for any $t\in[0,T]$, $x\in\mathbb{R}^d$,  $i_0 \in \mathcal{S}$ and $\mu \in \mathcal{P}_2(\mathbb{R}^d)$. 
\end{asAt}
\begin{asAt} \label{ast:holder:time:ms}
There exists a constant $L>0$ such that
\begin{align*}
|f_t(x,i_0,\mu) - & f_s(x,i_0,\mu)|^2  +|g_t(x,i_0,\mu)-g_s(x,i_0,\mu)|^2\\
& + \int_Z |h_t(x,i_0,\mu,z)-h_s(x,i_0,\mu, z)|^2  \nu(dz)  \leq L |t-s|  
\end{align*}
for any $t, s \in[0,T]$, $i_0 \in \mathcal{S}$, $x\in\mathbb{R}^d$ and $\mu\in\mathcal{P}_2(\mathbb{R}^d)$. 
\end{asAt}
\begin{asAt} \label{ast:bounded:ms}
There exists a constant $L>0$  such that
\begin{align*}
\sup_{t\in[0,T]}|f_t(0,i_0,\delta_0)|^2 \vee \sup_{t\in[0,T]} |g_t(0,i_0,\delta_0)|^2  \vee \sup_{t\in[0,T]} \int_Z |h_t(0,i_0,\delta_0, z)|^2\nu(dz) < L
\end{align*}
for any $t \in[0,T]$ and $i_0 \in \mathcal{S}$. 
\end{asAt}
\begin{rem} \label{rem:poly:lips:sig:gam:ms}
Due to Assumptions $\bar{A}$--\ref{ast:monotonicity:rate:ms} and $\bar{A}$--\ref{ast:poly:Lips:b:ms}, there exists a constant $L>0$ such that
\begin{align*}
|g_t(x,i_0,\mu)-g_t(\bar{x}, i_0,\bar{\mu})|^2 +&\int_Z |h_t(x,i_0,\mu,z)-h_t(\bar{x},i_0,\bar{\mu},z)|^2 \nu(dz) \\
&\leq L\big\{(1+|x|+|\bar{x}|)^{\chi/2}|x-\bar{x}|^2 +\mathcal{W}_2^2 (\mu, \bar{\mu})\big\}
\end{align*}
for any $t\in[0,T]$, $i_0 \in \mathcal{S}$, $x,\bar{x}\in\mathbb{R}^d$ and $\mu, \bar{\mu}\in\mathcal{P}_2(\mathbb{R}^d)$. 
\end{rem} 
\begin{rem} \label{rem:growth:ms}
Due to Assumption $\bar{A}$--\ref{ast:bounded:ms} and Remark \ref{rem:poly:lips:sig:gam:ms}, there exists a constant $L>0$ such that
\begin{align*}
|f_t(x,i_0,\mu)| & \leq L \big\{1+|x|^{\chi/2+1} + \mathcal{W}_2(\mu,\delta_0)\big\},
\\
|g_t(x,i_0,\mu)| & \leq L \big\{1+|x|^{\chi/4+1} + \mathcal{W}_2(\mu,\delta_0)\big\},
\\
\int_Z |h_t(x,i_0,\mu,z)|^2 \nu(dz) & \leq L \big\{1+|x|^{\chi/2+2}+ \mathcal{W}_2^2(\mu,\delta_0)\big\}
\end{align*}
 for any $t\in[0,T]$, $i_0 \in \mathcal{S}$, $x\in\mathbb{R}^d$ and $\mu \in\mathcal{P}_2(\mathbb{R}^d)$. 
\end{rem}
We now introduce a tamed Euler scheme for the McKean--Vlasov SDEwMS \eqref{eq:sdems} as follows,
\begin{align} \label{eq:scheme:ms}
y_t^{i,N,n}=y_0^{i}& +\int_{0}^t \frac{f_{\kappa_n(s)}(y_{\kappa_n(s)}^{i, N,n},\alpha_{\kappa_n(s)}, \mu_{\kappa_n(s)}^{y,N,n})}{1+n^{-1/2}|y_{\kappa_n(s)}^{i, N,n}|^\chi} ds + \int_{0}^t \frac{g_{\kappa_n(s)}(y_{\kappa_n(s)}^{i, N,n}, \alpha_{\kappa_n(s)}, \mu_{\kappa_n(s)}^{y,N,n})}{1+n^{-1/2}|y_{\kappa_n(s)}^{i, N,n}|^\chi} dw_s^i  \notag \\
&+ \int_0^t \int_Z \frac{h_{\kappa_n(s)}(y_{\kappa_n(s)}^{i, N,n},\alpha_{\kappa_n(s)}, \mu_{\kappa_n(s)}^{y,N,n}, z)}{1+n^{-1/2}|y_{\kappa_n(s)}^{i, N,n}|^\chi} \tilde{n}_p^i(ds,dz),
\end{align}
almost surely for any $t\in [0,T]$, $n,N\in\mathbb{N}$, where  
\begin{align*}
\mu_t^{y,N,n}(\cdot)= \frac{1}{N} \sum_{i=1}^N \delta_{y_t^{i,N,n}}(\cdot)
\end{align*}
  is an empirical measure.

\begin{thm}[\textbf{Tamed Euler Scheme}]
Let Assumptions $\bar{A}$--\ref{ast:initial:ms}, $\bar{A}$--\ref{ast:coercivity:ms} and  $\bar{A}$--\ref{ast:coercivity:ms:rate} to $\bar{A}$--\ref{ast:bounded:ms}  be satisfied. Then, the tamed Euler scheme  \eqref{eq:scheme:ms} converges in mean-square sense to the true solution of the interacting particle system  \eqref{eq:interacting:ms} connected with the McKean--Vlasov SDEwMS \eqref{eq:sdems}, 
and for any $\epsilon>0$ such that $\bar{p}_0 \ge \chi (2+\epsilon)/\epsilon$, we have
\begin{align*}
\sup_{i\in\{1,\ldots,N\}} \sup_{t\in[0,T]}\mathbb{E}|y_t^{i,N}-y_t^{i,N,n}|^2 \leq K n^{-2/(2+\epsilon)}
\end{align*}
for any $n, N\in\mathbb{N}$, where the constant $K>0$ does not depend on $n$ and $N$. 
\end{thm}
\begin{proof}
We prove this result by applying Theorem \ref{thm:order-12} on the interacting particle system \eqref{eq:interacting} associated to the McKean--Vlasov SDE \eqref{eq:sde} with coefficients given in equation \eqref{eq:coefficients:ms} and on its tamed Euler scheme \eqref{eq:scheme} with coefficients given by
\begin{align} \label{eq:coefficients:tamed:ms}
b_t^n(x,\mu)=\frac{f_t(x,\alpha_t, \mu)}{1+n^{-1/2}|x|^{\chi}},  \, \, \sigma_t^n(x,\mu)=\frac{g_t(x,\alpha_t, \mu)}{1+n^{-1/2}|x|^{\chi}}, \, \, \gamma_t^n(x,\mu, z)= \frac{h_t(x,\alpha_t,\mu, z)}{1+n^{-1/2}|x|^{\chi}}
\end{align}
almost surely for any $t\in [0,T]$, $x\in \mathbb{R}^d$, $\mu \in \mathcal{P}_2(\mathbb{R}^d)$, $z \in Z$ and $n\in \mathbb{N}$.
For this, we show that  assumptions of Theorem \ref{thm:order-12} are satisfied under the assumptions of this theorem.  

Notice that Assumptions $A$--\ref{ast:initial:ms}, $A$--\ref{as:coercivity} and $A$--\ref{as:coercivity:p0} have already been verified in the proof of Theorem \ref{thm:eu:ms}. 
Also, Assumptions $A$--\ref{as:monotonicity:rate}, $A$--\ref{as:poly:Lips:b}, $A$--\ref{as:holder:time} and $A$--\ref{as:bounded} follow from Assumptions $\bar{A}$--\ref{ast:monotonicity:rate:ms}, $\bar{A}$--\ref{ast:poly:Lips:b:ms}, $\bar{A}$--\ref{ast:holder:time:ms} and $\bar{A}$--\ref{ast:bounded:ms} respectively. 
Moreover, Assumption $B$--\ref{asb:coercivity:scheme} follows with $p_0=\bar{p}_0$ from Assumptions $\bar{A}$--\ref{ast:coercivity:ms:p0}.  Indeed, 
\begin{align}
F_1 := & 2|x|^{\bar{p}_0-2}  x b_t^n(x, \mu) +(\bar{p}_0-1) |x|^{\bar{p}_0-2} |\sigma_t^n(x,\mu)|^2  \notag
\\
& + 2(\bar{p}_0-1) \int_Z |\gamma_t^n(x,\mu,z)|^2 \int_0^1 (1-\theta)|x+\theta\gamma_t^n(x,\mu,z)|^{\bar{p}_0-2} d\theta  \nu(dz)  \notag
\\
 = & 2|x|^{\bar{p}_0-2} x\frac{f_t(x,\alpha_t, \mu)}{1+n^{-1/2}|x|^{\chi}}+(\bar{p}_0-1) |x|^{\bar{p}_0-2} \Big| \frac{g_t(x,\alpha_t, \mu)}{1+n^{-1/2}|x|^{\chi}} \Big|^2  \notag
\\
& + 2(\bar{p}_0-1) \int_Z \Big|\frac{h_t(x,\alpha_t,\mu, z)}{1+n^{-1/2}|x|^{\chi}}\Big|^2 \int_0^1 (1-\theta)\Big|x+\theta \frac{h_t(x,\alpha_t,\mu, z)}{1+n^{-1/2}|x|^{\chi}}\Big|^{\bar{p}_0-2} d\theta  \nu(dz)  \notag
\\
\leq & \frac{1}{1+n^{-1/2}|x|^{\chi}} \big\{ 2|x|^{\bar{p}_0-2} x f_t(x,\alpha_t, \mu) +(\bar{p}_0-1) |x|^{\bar{p}_0-2} |g_t(x,\alpha_t, \mu)|^2 \big\}  \notag
\\
& + (\bar{p}_0-1)2^{\bar{p}_0-3} \int_Z \frac{|h_t(x,\alpha_t,\mu, z)|^2}{(1+n^{-1/2}|x|^{\chi})^2} \Big\{ |x|^{\bar{p}_0-2}+ \frac{|h_t(x,\alpha_t,\mu, z)|^{\bar{p}_0-2}}{(1+n^{-1/2}|x|^{\chi})^{\bar{p}_0-2}} \Big\}  \nu(dz)  \notag
\\
\leq & \frac{1}{1+n^{-1/2}|x|^{\chi}} \big\{ 2|x|^{\bar{p}_0-2} x f_t(x,\alpha_t, \mu) +(\bar{p}_0-1) |x|^{\bar{p}_0-2} |g_t(x,\alpha_t, \mu)|^2 \big\}  \notag
\\
& + (\bar{p}_0-1)2^{\bar{p}_0-3} \int_Z \Big\{  \frac{|h_t(x,\alpha_t,\mu, z)|^2|x|^{\bar{p}_0-2}}{(1+n^{-1/2}|x|^{\chi})^2} + \frac{|h_t(x,\alpha_t,\mu, z)|^{\bar{p}_0}}{(1+n^{-1/2}|x|^{\chi})^{\bar{p}_0}} \Big\}  \nu(dz) \notag
\\
\leq & \frac{1}{1+n^{-1/2}|x|^{\chi}} \Big\{ 2|x|^{\bar{p}_0-2} x f_t(x,\alpha_t, \mu) +(\bar{p}_0-1) |x|^{\bar{p}_0-2} |g_t(x,\alpha_t, \mu)|^2   \notag
\\
& + (\bar{p}_0-1)2^{\bar{p}_0-3} \int_Z \big\{  |h_t(x,\alpha_t,\mu, z)|^2|x|^{\bar{p}_0-2} + |h_t(x,\alpha_t,\mu, z)|^{\bar{p}_0} \big\}  \nu(dz) \Big\}, \notag
\end{align}
which on the application of  Assumptions $\bar{A}$--\ref{ast:coercivity:ms:p0}  yields
\begin{align*}
F_1 \leq  &L(1+|x|^{\bar{p}_0}+\mathcal{W}_2^{\bar{p}_0}(\mu,\delta_0))
\end{align*}
for any $t\in[0,T]$, $x\in\mathbb{R}^d$, $n\in \mathbb{N}$ and $\mu\in\mathcal{P}_2(\mathbb{R}^d)$.

We now proceed to verify Assumption $B$--\ref{asb:growth:n}. 
By using Remark \ref{rem:growth:ms}, 
\begin{align*}
|b_t^n&(x,\mu)|^2 = \Big|\frac{f_t(x,\alpha_t, \mu)}{1+n^{-1/2}|x|^\chi}\Big|^2\leq \frac{K \{1+|x|^{\chi+2}+ \mathcal{W}_2^2(\mu, \delta_0)  \}}{1+n^{-1/2}|x|^\chi} \\
& \leq Kn^{1/2} 
\Big\{ \frac{1+|x|^{\chi+2}}{n^{1/2}+|x|^\chi}   +  \mathcal{W}_2^2(\mu, \delta_0) \Big\} \leq Kn^{1/2}\{ 1+|x|^2+\mathcal{W}_2^2(\mu, \delta_0) \},
\end{align*}
which implies
\[
|b_t^n(x,\mu)| \leq K n^{1/4} \{ 1+|x|+\mathcal{W}_2(\mu, \delta_0)\}
\]
for any $t\in [0,T]$, $x\in \mathbb{R}^d$, $\mu \in \mathcal{P}_2(\mathbb{R}^d)$ and $n\in \mathbb{N}$. Similarly,
\begin{align*}
|\sigma_t^n&(x,\mu)|^{4} = \Big|\frac{g_t(x,\alpha_t, \mu)}{1+n^{-1/2}|x|^{\chi}}\Big|^{4}\leq \frac{K \{1+|x|^{\chi+4}+ \mathcal{W}_2^{4}(\mu, \delta_0)  \}}{1+n^{-1/2}|x|^{\chi}} \\
& \leq Kn^{1/2} 
\Big\{ \frac{1+|x|^{\chi+4}}{n^{1/2}+|x|^{\chi}}   +  \mathcal{W}_2^{4}(\mu, \delta_0) \Big\} \leq Kn^{1/2}\{ 1+|x|^{4}+\mathcal{W}_2^{4}(\mu, \delta_0) \},
\end{align*}
which implies
\[
|\sigma_t^n(x,\mu)| \leq K n^{1/8} \{ 1+|x|+\mathcal{W}_2(\mu, \delta_0)\}
\]
for any $t\in [0,T]$, $x\in \mathbb{R}^d$, $\mu \in \mathcal{P}_2(\mathbb{R}^d)$ and $n\in \mathbb{N}$.  Further,
\begin{align*}
\Big(\int_Z|\gamma_t^n(x, &\mu,z)|^2\nu(dz)\Big)^2 = K \Big(\int_Z\frac{|h_t(x,\alpha_t, \mu,z)|^2}{1+n^{-1/2}|x|^{\chi}}\nu(dz)\Big)^2 \\
&\leq K\frac{ \{1+|x|^{\chi+4}+ \mathcal{W}_2^4(\mu, \delta_0)  \}}{1+n^{-1/2}|x|^{\chi}}  \leq Kn^{1/2} 
\Big\{ \frac{1+|x|^{\chi+4}}{n^{1/2}+|x|^{\chi}}   +  \mathcal{W}_2^{4}(\mu, \delta_0) \Big\} \\
&\leq Kn^{1/2}\{ 1+|x|^{4}+\mathcal{W}_2^{4}(\mu, \delta_0) \},
\end{align*}
which implies
\[
\int_Z|\gamma_t^n(x,\mu,z)|^2\nu(dz)\leq K n^{1/4} \{ 1+|x|^2+\mathcal{W}_2^2(\mu, \delta_0)\}
\]
for any $t\in [0,T]$, $x\in \mathbb{R}^d$, $\mu \in \mathcal{P}_2(\mathbb{R}^d)$ and $n\in \mathbb{N}$. 
Moreover,  using Assumption~$\bar{A}$--\ref{ast:coercivity:scheme:ms} for $2\leq q \leq \bar{p}_0$, 
\begin{align*}
\Big(\int_Z|\gamma_t^n&(x,\mu,z)|^{q}\nu(dz)\Big)^2 = K \Big(\int_Z\Big|\frac{h_t(x,\alpha_t, \mu,z)}{1+n^{-1/2}|x|^{\chi}}\Big|^{q} \nu (dz)\Big)^2\leq \frac{K \{1+|x|^{ \chi+2q}+ \mathcal{W}_2^{2q}(\mu, \delta_0)  \}}{1+n^{-1/2}|x|^{\chi}}
 \\
& \leq Kn^{1/2} \Big\{ \frac{1+|x|^{ \chi+2q}}{n^{1/2}+|x|^{\chi}}   +  \mathcal{W}_2^{2q}(\mu, \delta_0) \Big\} \leq Kn^{1/2}\{ 1+|x|^{2q}+\mathcal{W}_2^{2q}(\mu, \delta_0) \},
\end{align*}
which implies
\[
\int_Z|\gamma_t^n(x,\mu,z)|^{q}\nu(dz)\leq K n^{1/4} \{ 1+|x|^{q}+\mathcal{W}_2^{q}(\mu, \delta_0)\}
\]
for any $t\in [0,T]$, $x\in \mathbb{R}^d$, $\mu \in \mathcal{P}_2(\mathbb{R}^d)$ and $n\in \mathbb{N}$. 

Thus, on the application of Lemma \ref{lem:scheme:mb}, one obtains 
\begin{align} \label{eq:mb:ms}
\sup_{i\in\{1,\ldots, N\}}\sup_{t\in[0,T]} \mathbb{E}   |y_t^{i, N, n}|^{\bar{p}_0} \leq K,
\end{align}
for any $n, N\in\mathbb{N}$, where $K>0$ does not depend on $n$ and $N$. 

We now verify Assumption B--\ref{asb:diff:rate}. 
From equations \eqref{eq:coefficients:ms} and \eqref{eq:coefficients:tamed:ms}, we have
\begin{align*}
 \mathbb{E} &\int_0^T |b_{\kappa_n(s)}(y_{\kappa_n(s)}^{i, N,n}, \mu_{\kappa_n(s)}^{y,N,n})- b^n_{\kappa_n(s)}(y_{\kappa_n(s)}^{i, N, n}, \mu_{\kappa_n(s)}^{y,N, n})|^2 ds
  \\
 &\leq K \mathbb{E} \int_0^T\Big| f_{\kappa_n(s)}(y_{\kappa_n(s)}^{i, N,n},\alpha_s, \mu_{\kappa_n(s)}^{y,N,n})- f_{\kappa_n(s)}(y_{\kappa_n(s)}^{i, N,n},\alpha_{\kappa_n(s)}, \mu_{\kappa_n(s)}^{y,N,n}) \Big|^2 ds
 \\
 &\qquad+K \mathbb{E} \int_0^T\Big| f_{\kappa_n(s)}(y_{\kappa_n(s)}^{i, N,n},\alpha_{\kappa_n(s)}, \mu_{\kappa_n(s)}^{y,N,n}) -\frac{f_{\kappa_n(s)}(y_{\kappa_n(s)}^{i, N,n},\alpha_{\kappa_n(s)}, \mu_{\kappa_n(s)}^{y,N,n})}{1+n^{-1/2}|y_{\kappa_n(s)}^{i, N,n}|^{\chi}} \Big|^2 ds
 \\
 &\leq  K \mathbb{E} \int_0^T\Big\{ | f_{\kappa_n(s)}(y_{\kappa_n(s)}^{i, N,n},\alpha_s, \mu_{\kappa_n(s)}^{y,N,n})|^2+ |f_{\kappa_n(s)}(y_{\kappa_n(s)}^{i, N,n},\alpha_{\kappa_n(s)}, \mu_{\kappa_n(s)}^{y,N,n})|^2 \Big\} \mathbb{I}_{\{\alpha_{\kappa_n(s)}\neq\alpha_s\}} ds 
 \\
 &\qquad+K \mathbb{E} \int_0^T |f_{\kappa_n(s)}(y_{\kappa_n(s)}^{i, N,n},\alpha_{\kappa_n(s)}, \mu_{\kappa_n(s)}^{y,N,n})|^2 n^{-1} |y_{\kappa_n(s)}^{i, N,n}|^{2 \chi}  ds,
 \end{align*}
 which due to Remark \ref{rem:growth:ms}, equation \eqref{eq:mb:ms} and equation \eqref{eq:MProb} yields, 
  \begin{align*}
   \mathbb{E} \int_0^T |b_{\kappa_n(s)}&(y_{\kappa_n(s)}^{i, N,n}, \mu_{\kappa_n(s)}^{y,N,n})- b^n_{\kappa_n(s)}(y_{\kappa_n(s)}^{i, N, n}, \mu_{\kappa_n(s)}^{y,N, n})|^2 ds
  \\
 & \leq K n^{-1} + Kn^{-1} \mathbb{E} \int_0^T |y_{\kappa_n(s)}^{i, N,n}|^{2\chi} \big\{|y_{\kappa_n(s)}^{i, N,n}|^{\chi+2}+ \mathcal{W}^2_2(\mu_{\kappa_n(s)}^{y,N,n},\delta_0)\big\}ds
 \\
 &\leq  K n^{-1} + Kn^{-1} \mathbb{E} \int_0^T |y_{\kappa_n(s)}^{i, N,n}|^{2 \chi} \Big\{|y_{\kappa_n(s)}^{i, N,n}|^{\chi+2}+ \frac{1}{N}\sum_{i=1}^N |y_{\kappa_n(s)}^{i,N,n}|^2 \Big\}ds 
 \\
 &\leq  K n^{-1} +  Kn^{-1} \sup_{i\in \{1,\ldots, N\}}\sup_{s\in[0,T]}\sup_{n\in \mathbb{N}}  \mathbb{E}\big\{ |y_{\kappa_n(s)}^{i, N,n}|^{3\chi+2}+|y_{\kappa_n(s)}^{i, N,n}|^{2\chi+2}\big\}   \leq K n^{-1}
 \end{align*}
 for any $i\in\{1,\ldots,N\}$ and $n, N\in\mathbb{N}$.
 Similarly, 
\begin{align*}
 \mathbb{E} &\int_0^T |\sigma_{\kappa_n(s)}(y_{\kappa_n(s)}^{i, N,n}, \mu_{\kappa_n(s)}^{y,N,n})- \sigma^n_{\kappa_n(s)}(y_{\kappa_n(s)}^{i, N, n}, \mu_{\kappa_n(s)}^{y,N, n})|^2 ds 
 \\
 &\leq  K n^{-1} + Kn^{-1} \mathbb{E} \int_0^T |y_{\kappa_n(s)}^{i, N,n}|^{2 \chi} \Big\{|y_{\kappa_n(s)}^{i, N,n}|^{\chi/2+2}+ \frac{1}{N}\sum_{i=1}^N |y_{\kappa_n(s)}^{i,N,n}|^2 \Big\}ds 
 \\
 &\leq   K n^{-1} + Kn^{-1} \sup_{i\in \{1,\ldots, N\}}\sup_{s\in[0,T]}\sup_{n\in \mathbb{N}}  \mathbb{E}\big\{ |y_{\kappa_n(s)}^{i, N,n}|^{5\chi+2}+|y_{\kappa_n(s)}^{i, N,n}|^{2\chi+2}\big\} \leq K n^{-1}
 \end{align*}
 for any $i\in\{1,\ldots,N\}$ and $n, N\in\mathbb{N}$.
 Finally,
 \begin{align*}
   \mathbb{E} \int_0^T& \int_Z |\gamma_{\kappa_n(s)}(y_{\kappa_n(s)}^{i, N,n}, \mu_{\kappa_n(s)}^{y,N,n}, z)- \gamma^n_{\kappa_n(s)}(y_{\kappa_n(s)}^{i, N, n}, \mu_{\kappa_n(s)}^{y,N, n}, z)|^2 \nu(dz) ds \\
   & \leq Kn^{-1} + Kn^{-1} \mathbb{E} \int_0^T |y_{\kappa_n(s)}^{i, N,n}|^{2 \chi} \Big\{|y_{\kappa_n(s)}^{i, N,n}|^{\chi/2+2}+ \frac{1}{N}\sum_{i=1}^N |y_{\kappa_n(s)}^{i,N,n}|^2 \Big\}ds 
   \\
 &\leq  Kn^{-1} + Kn^{-1} \sup_{i\in \{1,\ldots, N\}}\sup_{s\in[0,T]}\sup_{n\in \mathbb{N}}  \mathbb{E}\big\{ |y_{\kappa_n(s)}^{i, N,n}|^{5\chi+2}+|y_{\kappa_n(s)}^{i, N,n}|^{2\chi+2}\big\} \leq K n^{-1} 
\end{align*}
 for any $i\in\{1,\ldots,N\}$ and $n, N\in\mathbb{N}$.  This completes the proof. 
\end{proof}

\subsection{McKean--Vlasov SDDE  driven by L\'evy noise}
For simplicity, consider the case of a single delay, but the result can be extended to multiple delays without any difficulty. 
Let $f:[0,T]\times \mathbb{R}^d \times \mathbb{R}^d \times \mathcal{P}_2(\mathbb{R}^d) \times \mathcal{P}_2(\mathbb{R}^d)$,  $g:[0,T]\times \mathbb{R}^d \times \mathbb{R}^d \times \mathcal{P}_2(\mathbb{R}^d) \times \mathcal{P}_2(\mathbb{R}^d)$ and  $h:[0,T]\times \mathbb{R}^d \times \mathbb{R}^d \times \mathcal{P}_2(\mathbb{R}^d) \times \mathcal{P}_2(\mathbb{R}^d) \times Z$ be measurable functions taking values in $\mathbb{R}^d$, $\mathbb{R}^{d \times m}$ and $\mathbb{R}^d$ respectively. 
Let $\tau>0$ be a fixed constant denoting the fixed delay parameter.
Consider the following $d$--dimensional McKean--Vlasov SDDE driven by L\'evy noise,  
\begin{align} \label{eq:sdedelay}
y_t & =y_0+\int^t_0 f_s(y_s, y_{s-\tau}, L_{y_s}, L_{y_{s-\tau}})ds +\int^t_0 g_s(y_s, y_{s-\tau}, L_{y_s}, L_{y_{s-\tau}})dw_s  \notag
\\
& \qquad +\int^t_0 \int_Z h_s(y_s, y_{s-\tau}, L_{y_s}, L_{y_{s-\tau}}, z) \tilde{n}_p(ds,dz)
\end{align} 
almost surely for any $t\in[0,T]$ with  the initial data $\xi_t:=y_{t-\tau}$ for any $t\in [0, \tau]$. 
Without loss of generality, one can assume that $T=M\tau$ for some $M\in\mathbb{N}$ such that the interval $[0,T]$ can be divided into $M$ sub-intervals   $[(j-1)\tau, j\tau]$ for $j=1,\ldots,M$.

In order to apply the results obtained for McKean--Vlasov SDEs with random coefficients in Sections $2$, $3$ and $4$,  one can first re-write the results for interval $[t_0, t_1]$ instead of $[0,T]$, which is a trivial computation, and then use them for McKean--Vlasov SDDEs driven by L\'evy noise by taking $t_0=(j-1)\tau$ and $t_1=j\tau$ for $j=1,\ldots, M$. In this section, we adapt the approach of \cite{kumar2014} from SDDEs to McKean--Vlasov SDDEs, utilising heavily our results from earlier sections.
\begin{rem}
When $\tau\equiv 0$, then the McKean--Vlasov SDDE \eqref{eq:sdedelay} reduces to a McKean--Vlasov SDE. Hence, all the results derived in this section also hold for McKean--Vlasov SDEs.
\end{rem}

 Let $\bar{\bar{p}}_0\geq 2$ be a fixed constant. 
We make the following assumptions. 
\begin{asAtt} \label{astt:initial:delay}
$\sup_{t\in[0, \tau]}\mathbb{E}|\xi_t|^{\bar{\bar{p}}_0} < \infty$.
\end{asAtt}
\begin{asAtt} \label{astt:coercivity:delay}
There exist constants $L>0$ and $\rho \geq 2 $  such that
\begin{align*}
2  xf_t(x,y,\mu, \mu') +  & |g_t(x,y,\mu, \mu')|^2 +   \int_Z |h_t(x,y,\mu, \mu', z)|^2  \nu(dz)  \leq L\big\{1+|x|^{2}+|y|^{\rho}
\\
& \qquad +\mathcal{W}_2^{2}(\mu,\delta_0)+\mathcal{W}_{2}^{\rho}(\mu',\delta_0) \big\}
\end{align*}
 for any $t\in[0,T]$, $x, y\in\mathbb{R}^d$  and $\mu, \bar{\mu} \in \mathcal{P}_2(\mathbb{R}^d)$. 
\end{asAtt}
\begin{asAtt} \label{astt:monotonicity:delay}
There exists a  constant $L>0$  such that
\begin{align*}
2(x-\bar{x}) &(f_t(x,y,\mu, \mu')  -f_t(\bar{x},y,\bar{\mu}, \mu'))  + |g_t(x,y,\mu, \mu')-g_t(\bar{x},y,\bar{\mu}, \mu')|^2 
\\
& +\int_Z |h_t(x,y,\mu, \mu', z)-h_t(\bar{x},y,\bar{\mu}, \mu', z)|^2 \nu(dz)  \leq L\big\{|x-\bar{x}|^2 + \mathcal{W}_2^2(\mu,\bar{\mu}) \big\}
\end{align*}
 for any $t\in[0,T]$, $x, \bar{x}, y,  \in\mathbb{R}^d$ and $\mu, \bar{\mu}, \mu' \in \mathcal{P}_2(\mathbb{R}^d)$. 
\end{asAtt}
\begin{asAtt} \label{astt:continuity:delay}
For any   $t\in[0,T]$, $y\in\mathbb{R}^d$ and $\mu'\in\mathcal{P}_2(\mathbb{R}^d)$,    $f_t(x,y,\mu, \mu')$ is a continuous function of $x\in\mathbb{R}^d$ and $\mu\in\mathcal{P}_2(\mathbb{R}^d)$. 
\end{asAtt}
\begin{asAtt} \label{astt:coercivity:delay:pj}
There exists a constant $L>0$ such that 
\begin{align*}
&2|x|^{p_j-2}  xf_t(x,y,\mu, \mu') +(p_j-1) |x|^{p_j-2} |g_t(x,y,\mu, \mu')|^2 
\\
&\qquad  + 2 (p_j-1) \int_Z |h_t(x,y,\mu, \mu', z)|^2 \int_0^1 (1-\theta)| x+\theta h_t(x,y,\mu, \mu', z)|^{p_j-2} d\theta \nu(dz) 
\\
& \leq L\big\{1+|x|^{p_j}+|y|^{p_{j-1}}+\mathcal{W}_2^{p_j}(\mu,\delta_0)+\mathcal{W}^{p_{j-1}}_{2}(\mu',\delta_0) \big\}
\end{align*}
 for any $t\in[0,T]$, $x, y\in\mathbb{R}^d$  and $\mu, \mu' \in \mathcal{P}_2(\mathbb{R}^d)$ where $p_j=(2/\rho)^j \bar{\bar{p}}_0$ for $j=1,\ldots,M$. 
\end{asAtt}
\begin{thm}[\textbf{Existence, Uniqueness and Moment Bound}] \label{thm:eu:delay}
Let Assumptions $\bar{\bar{A}}$--\ref{astt:initial:delay} (with $\bar{\bar{p}}_0=2$), $\bar{\bar{A}}$--\ref{astt:coercivity:delay}, $\bar{\bar{A}}$--\ref{astt:monotonicity:delay} and $\bar{\bar{A}}$--\ref{astt:continuity:delay} be satisfied. Then, there exists a unique c\`adl\`ag process taking values in $\mathbb{R}^d$ satisfying the  McKean--Vlasov SDE \eqref{eq:sdedelay} such that 
\begin{align*}
\sup_{t\in[0,T]}\mathbb{E}|y_t|^{2} \leq K, 
\end{align*}
where $K:=K( \sup_{t\in[0,\tau]} \mathbb{E}|\xi_t|^{2},m,d,L)$ is a positive constant. 
Moreover, if Assumptions $\bar{\bar{A}}$--\ref{astt:initial:delay} and $\bar{\bar{A}}$--\ref{astt:coercivity:delay:pj}  hold for  $\bar{\bar{p}}_0>2$, then 
\begin{align*}
\sup_{t\in[0,T]}\mathbb{E}|y_t|^{p^*} \leq K,
\end{align*}
where $p^*:=p_M=(2/ \rho)^M \bar{\bar{p}}_0$ and $K:=K( \sup_{t\in[0,\tau]} \mathbb{E}|\xi_t|^{\bar{\bar{p}}_0},m,d,\bar{\bar{p}}_0,L)$ is a positive constant. 
\end{thm}

\begin{proof}
We prove the result using an inductive argument. This is done by verifying Assumptions $A$--\ref{as:initial} to $A$--\ref{as:coercivity:p0}  on the interval $[(j-1)\tau, j\tau]$ and applying Theorem \ref{thm:eu} for $j=1,\ldots,M$. 
 For this, we define
\begin{align} \label{eq:coefficients:delay}
b_t(x,\mu):=f_t(x,y_{t-\tau}, &\, \mu, L_{y_{t-\tau}}),  \,\,   \sigma_t(x,\mu):=g_t(x,y_{t-\tau}, \mu, L_{y_{t-\tau}}), \notag
\\
  \gamma_t(x,\mu, z) & := h_t(x,y_{t-\tau}, \mu, L_{y_{t-\tau}}, z) 
\end{align} 
 almost surely for any $t\in[0,T]$, $x\in\mathbb{R}^d$, $\mu\in\mathcal{P}_2(\mathbb{R}^d)$ and $z\in Z$.
\newline

\textbf{Case} $\mathbf{j=1.}$ 
Assumption $A$--\ref{as:initial} holds trivially from Assumption $\bar{\bar{A}}$--\ref{astt:initial:delay}.   Further,  Assumption $A$--\ref{as:coercivity} follows from Assumption $\bar{\bar{A}}$--\ref{astt:coercivity:delay}. Indeed, for any $t\in[0,\tau]$, 
\begin{align}
2 x  & b_t(x, \mu)   + |\sigma_t(x,\mu)|^2 +   \int_Z |\gamma_t(x,\mu,z)|^2 \nu(dz)   =2 xf_t(x,y_{t-\tau},  \mu, L_{y_{t-\tau}}) +|g_t(x,y_{t-\tau},  \mu, L_{y_{t-\tau}})|^2 \notag
\\
&  \quad + \int_Z |h_t(x,y_{t-\tau}, \mu, L_{y_{t-\tau}}, z) |^2  \nu(dz)  \leq L  \{1+|x|^{2}+|y_{t-\tau}|^{\rho}+ \mathcal{W}_2^{2}(\mu,\delta_0)+ \mathcal{W}_{2}^{\rho}(L_{y_{t-\tau}},\delta_0)\} \notag
\\
& \leq L  \{1+|x|^{2}+|\xi_{t}|^{\rho}+ \mathcal{W}_2^{2}(\mu,\delta_0)+ \bE|\xi_{t}|^{\rho}\} = L  \{M_t+ \mathcal{W}_2^{2}(\mu,\delta_0)+|x|^{2}\}, \label{eq:coercivity:j=1}
\end{align}
where $M_t:=1+|\xi_{t}|^{\rho}+ (\mathbb{E}|\xi_{t}|^2)^{\rho/2}$ and $\sup_{t\in[0,\tau]}\mathbb{E}M_t <\infty$ because $ \rho\leq \bar{\bar{p}}_0$. 
Also, Assumptions $A$--\ref{as:monotonicity} and $A$--\ref{as:continuity} hold trivially from Assumptions $\bar{\bar{A}}$--\ref{astt:monotonicity:delay} and $\bar{\bar{A}}$--\ref{astt:continuity:delay} respectively. 
Thus, by applying Theorem  \ref{thm:eu}, there exists a unique c\`adl\`ag process satisfying  equation \eqref{eq:sdedelay} in the interval $[0,\tau]$ such that
\begin{align*}
\sup_{t\in [0,\tau]}\mathbb{E}|y_t|^{2} \leq K,
\end{align*}
where $K:=K( \sup_{t\in[0,\tau]} \mathbb{E}|\xi_t|^{2},m,d,L)$ is a positive constant.   
Moreover,  Assumption $\bar{\bar{A}}$--\ref{astt:coercivity:delay:pj} with $j=1$ implies Assumption $A$--\ref{as:coercivity:p0} for $p_0=p_1$. Indeed,  for any $t\in[0,\tau]$, 
\begin{align} 
2|x&|^{p_1-2}  xb_t(x, \mu) +(p_1-1) |x|^{p_1-2} |\sigma_t(x,\mu)|^2 \notag
\\
& \quad
+ 2 (p_1-1) \int_Z |\gamma_t(x,\mu,z)|^2 \int_0^1 (1-\theta)| x+\theta \gamma_t(x,\mu,z)|^{p_1-2} d\theta \nu(dz) \notag
\\
& =2|x|^{p_1-2}  xf_t(x,y_{t-\tau},  \mu, L_{y_{t-\tau}}) +(p_1-1) |x|^{p_1-2} |g_t(x,y_{t-\tau},  \mu, L_{y_{t-\tau}})|^2 \notag
\\
& \quad + 2 (p_1-1) \int_Z |h_t(x,y_{t-\tau}, \mu, L_{y_{t-\tau}}, z) |^2 \int_0^1 (1-\theta)| x+\theta h_t(x,y_{t-\tau}, \mu, L_{y_{t-\tau}}, z) |^{p_1-2} d\theta \nu(dz)  \notag
\\
& \leq L  \{1+|x|^{p_1}+|y_{t-\tau}|^{\bar{\bar{p}}_0}+ \mathcal{W}_2^{p_1}(\mu,\delta_0)+ \mathcal{W}^{\bar{\bar{p}}_0}_2(L_{y_{t-\tau}},\delta_0)\} \notag
\\
& \leq L  \big\{1+|x|^{p_1}+|\xi_{t}|^{\bar{\bar{p}}_0}+ \mathcal{W}_2^{p_1}(\mu,\delta_0)+ (\mathbb{E}|\xi_{t}|^{2})^{\bar{\bar{p}}_0/2}\} = L  \{\bar{M}_t+ \mathcal{W}_2^{p_1}(\mu,\delta_0)+|x|^{p_1}\big\} \label{eq:coercivity:pj:j=1}
\end{align}
almost surely where $\bar{M}_t=1+|\xi_{t}|^{\bar{\bar{p}}_0}+ (\mathbb{E}|\xi_{t}|^{2})^{\bar{\bar{p}}_0/2}$   satisfies $\sup_{t\in[0,\tau]}\mathbb{E}\bar{M}_t<\infty$. Thus,  by Theorem~\ref{thm:eu}, 
$$
\sup_{t\in[0,\tau]}\bE|y_t|^{p_1}\leq K,
$$ 
where $K:=K( \sup_{t\in[0,\tau]} \mathbb{E}|\xi_t|^{\bar{\bar{p}}_0},m,d,\bar{\bar{p}}_0,L)$ is a positive constant.

For the inductive argument, we assume that result hold for the case $j=\ell$, \textit{i.e.}, the McKean--Vlasov SDDE \eqref{eq:sdedelay} has a unique solution in the interval $[(\ell-1)\tau, \ell \tau]$,  which satisfies
\begin{align}
\sup_{t\in [(\ell-1)\tau, \ell\tau]}\mathbb{E}|y_t|^{p_{\ell}}\leq K, \label{eq:**}
\end{align}
where $K:=K( \sup_{ t \in [0, \tau]} \mathbb{E}|\xi_t|^{\bar{\bar{p}}_0},m,d,\bar{\bar{p}}_0,L)$ is a positive constant.
\newline
\textbf{Case $\mathbf{j=\ell+1.}$}
Assumption $A$--\ref{as:initial} holds from the inductive hypothesis. 
By using an argument similar to the one used in the estimate \eqref{eq:coercivity:j=1},  Assumption $A$--\ref{as:coercivity} holds from Assumption $\bar{\bar{A}}$--\ref{astt:coercivity:delay} with $M_t:=1+|y_{t-\tau}|^{\rho}+ (\mathbb{E}|y_{t-\tau}|^2)^{\rho/2}$, which clearly satisfies $\sup_{t\in[\ell \tau, (\ell+1) \tau]}\mathbb{E}M_t <\infty$ due to inductive argument as $ \rho\leq p_{\ell}$.
Further, Assumptions $A$--\ref{as:monotonicity} and $A$--\ref{as:continuity} hold trivially from Assumptions $\bar{\bar{A}}$--\ref{astt:monotonicity:delay} and $\bar{\bar{A}}$--\ref{astt:continuity:delay} respectively. 
Moreover, by using a similar argument as given for the estimate \eqref{eq:coercivity:pj:j=1},  Assumption $A$--\ref{as:coercivity:p0} holds from Assumption $\bar{\bar{A}}$--\ref{astt:coercivity:delay:pj} with $p_0=p_{\ell+1}$ and $\bar{M}_t=1+|y_{t-\tau}|^{p_{\ell}}+ (\mathbb{E}|y_{t-\tau}|^2)^{p_{\ell/2}}$ satisfies $\sup_{t\in[\ell\tau, (\ell+1)\tau]}\mathbb{E}\bar{M}_t<\infty$ due to estimate \eqref{eq:**}. Hence, by using Theorem \ref{thm:eu}, the McKean--Vlasov SDDE \eqref{eq:sdedelay} has a unique c\`adl\`ag solution, which satisfies
\begin{align*}
\sup_{t\in [\ell\tau, (\ell+1)\tau]}\mathbb{E}|y_t|^{p_{\ell+1}}\leq K, 
\end{align*}
where $K:=K(  \sup_{ t \in [(\ell-1)\tau, \ell \tau]} \mathbb{E}|y_t|^{p_{\ell}},m,d, p_{\ell},L)= K(\sup_{ t \in [0, \tau]} \mathbb{E}|\xi_t|^{\bar{\bar{p}}_0},m,d, \bar{\bar{p}}_0,L)$ is a positive constant. From the inductive argument, one can also conclude that
\begin{align}
\sup_{t\in [0, T]}\mathbb{E}|y_t|^{p^*}\leq K, \label{eq:mb:delay}
\end{align}
where $K:= K(\sup_{ t \in [0, \tau]} \mathbb{E}|\xi_t|^{\bar{\bar{p}}_0},m,d, \bar{\bar{p}}_0,L)$. 
 This completes the proof. 
\end{proof}
The non-interacting particle system corresponding to McKean--Vlasov SDDE \eqref{eq:sdedelay} is given by 
\begin{align} \label{eq:sdedelay:non-interacting}
y_t^i & =y_0+\int^t_0 f_s(y_s^i, y_{s-\tau}^i, L_{y_s^i}, L_{y_{s-\tau}^i})ds +\int^t_0 g_s(y_s^i, y_{s-\tau}^i, L_{y_s^i}, L_{y_{s-\tau}}^i)dw_s^i  \notag
\\
& \qquad +\int^t_0 \int_Z h_s(y_s^i, y_{s-\tau}^i, L_{y_s^i}, L_{y_{s-\tau}^i}, z) \tilde{n}^i_p(ds,dz)
\end{align} 
almost surely for any $t\in[0,T]$.  Also, the interacting particle system associated with the above non-interacting particle system is given by
\begin{align} \label{eq:sdedelay:interacting}
y_t^{i,N} & =y_0+\int^t_0 f_s(y_s^{i,N}, y_{s-\tau}^{i,N}, \mu_s^{y,N}, \mu_{s-\tau}^{y,N})ds +\int^t_0 g_s(y_s^{i,N}, y_{s-\tau}^{i,N},  \mu_s^{y,N}, \mu_{s-\tau}^{y,N})dw_s^i  \notag
\\
& \qquad +\int^t_0 \int_Z h_s(y_s^{i,N}, y_{s-\tau}^{i,N},  \mu_s^{y,N}, \mu_{s-\tau}^{y,N}, z) \tilde{n}_p^i(ds,dz)
\end{align} 
almost surely for any $t\in[0,T]$ and $N\in\mathbb{N}$, where $\mu_t^{y,N}$ is an empirical measures given by
\begin{align*}
\mu_t^{y,N}(\cdot)= \frac{1}{N} \sum_{i=1}^N \delta_{y_t^{i,N}}(\cdot).
\end{align*}
\begin{prop}[\textbf{Propagation of Chaos}] \label{prop:propagation:delay}
Let Assumptions $\bar{\bar{A}}$--\ref{astt:initial:delay} to $\bar{\bar{A}}$--\ref{astt:coercivity:delay:pj} be satisfied with $\bar{\bar{p}}_0>4$. 
Then, the interacting particle system \eqref{eq:sdedelay:interacting} is wellposed and converges to the non-interacting particle system \eqref{eq:sdedelay:non-interacting} with a rate of convergence given by
\[ 
\sup_{i\in\{1,\ldots,N\}}\sup_{t\in[0,T]}\mathbb{E}|y^i_t-y^{i,N}_t|^2\leq K 
\begin{cases}
N^{-1/2}, &  \mbox{ if }  d<4, \\
N^{-1/2} \ln(N), &  \mbox{ if } d=4, \\
N^{-2/d}, &  \mbox{ if } d>4,
\end{cases}
\]
for any $N\in\mathbb{N}$, where  constant $K>0$ does not depend on $d$ and $N$. 
\end{prop}
\begin{proof}
We have already seen in the proof of Theorem \ref{thm:eu:delay} that Assumptions $A$--\ref{as:initial} to $A$--\ref{as:coercivity:p0} follow from Assumptions $\bar{\bar{A}}$--\ref{astt:initial:delay} to $\bar{\bar{A}}$--\ref{astt:coercivity:delay:pj}. Hence, the result follows directly from Proposition \ref{prop:propagation}. 
\end{proof}
We now proceed to the time discretization, \textit{i.e.}  a tamed Euler scheme for the interacting particle system \eqref{eq:sdedelay:interacting} associated with the McKean--Vlasov SDDE \eqref{eq:sdedelay}. For this, we make the following additional assumptions. 
\begin{asAtt} \label{astt:coercivity:delay:rate}
There exists a constant $L>0$  such that
\begin{align*}
2|x|^{p_j-2} &x f_t(x, y, \mu, \mu') +(p_j-1) |x|^{p_j-2} |g_t(x, y, \mu, \mu')|^2 
\\
& + (p_j-1)2^{p_j-3} \int_Z \big\{  |h_t(x, y, \mu, \mu', z)|^2|x|^{p_j-2} + |h_t(x, y, \mu, \mu', z)|^{p_j} \big\}  \nu(dz)
\\
  \leq &  L\big\{1+|x|^{p_j}+|y|^{p_{j-1}}+\mathcal{W}_2^{p_j}(\mu,\delta_0)+\mathcal{W}_2^{p_{j-1}}(\mu',\delta_0) \big\}
\end{align*}
 for any $t\in[0,T]$, $x, y\in\mathbb{R}^d$  and $\mu, \mu' \in \mathcal{P}_2(\mathbb{R}^d)$ where $p_j=(2/\rho)^j \bar{\bar{p}}_0$ for $j=1,\ldots,M$. 
\end{asAtt}
\begin{asAtt} \label{astt:monotonicity:delay:rate}
There exist constants $L>0$ and $\eta>1$  such that
\begin{align*}
2(x-\bar{x}) &(f_t(x,y,\mu, \mu')  -f_t(\bar{x},y,\bar{\mu}, \mu'))  + \eta |g_t(x,y,\mu, \mu')-g_t(\bar{x},y,\bar{\mu}, \mu')|^2 
\\
& + \eta \int_Z |h_t(x,y,\mu, \mu', z)-h_t(\bar{x},y,\bar{\mu}, \mu', z)|^2 \nu(dz)  \leq L\big\{|x-\bar{x}|^2 + \mathcal{W}_2^2(\mu,\bar{\mu}) \big\}
\end{align*}
 for any $x, \bar{x}, y  \in\mathbb{R}^d$, $\mu, \bar{\mu}, \mu' \in \mathcal{P}_2(\mathbb{R}^d)$ and $t\in[0,T]$. 
\end{asAtt}
\begin{asAtt} \label{astt:poly:Lips:f:delay}
There exist  constants $L>0$ and $\chi>0$  such that
\begin{align*}
|f_t(x,y,\mu, \mu')  -f_t(\bar{x},y,\bar{\mu}, \mu') |^2 & \leq L\big\{ (1+|x|+|\bar{x}|)^{\chi} |x-\bar{x}|^2 +\mathcal{W}_2^2(\mu,\bar{\mu})\big\}
\end{align*}
for any $t\in[0,T]$, $x,\bar{x}, y\in\mathbb{R}^d$ and $\mu, \bar{\mu}, \mu'\in\mathcal{P}_2(\mathbb{R}^d)$. 
\end{asAtt}
\begin{asAtt} \label{astt:holder:time:delay}
There exists a constant $L>0$ such that
\begin{align*}
|f_t(x,y, & \mu, \mu')-f_s(x,y,\mu, \mu')|^2+|g_t(x,y,\mu, \mu')-g_s(x,y,\mu, \mu')|^2
\\
&+ \int_Z |h_t(x,y,\mu, \mu', z)-h_s(x,y,\mu, \mu', z)|^2  \nu(dz) \leq L |t-s|  
\end{align*}
for any $t, s \in[0,T]$, $x, y\in\mathbb{R}^d$ and $\mu, \mu'\in\mathcal{P}_2(\mathbb{R}^d)$. 
\end{asAtt}
\begin{asAtt} 
There exists a constant $L>0$ such that
\begin{align*}
|f_t(0,0, \delta_0, \delta_0)|^2 \vee  |g_t(0,0, \delta_0,\delta_0)|^2  \vee  \int_Z |h_t(0,0, \delta_0, \delta_0, z)|^2\nu(dz) < L 
\end{align*}
for any $t\in[0,T]$. 
\end{asAtt}
\begin{asAtt}\label{astt:poly:lips:delpar}
There exist constants $L>0$ and $\chi>0$ such that
\begin{align*}
|f_t & (x, y,  \mu, \mu')-f_t(x,\bar{y},\mu, \bar{\mu}')|^2 + |g_t(x, y,  \mu, \mu')-g_t(x,\bar{y}, \mu, \bar{\mu}')|^2
\\
& +\int_Z |h_t(x,y,\mu, \mu', z)-h_t(x,\bar{y},\mu, \bar{\mu}', z)|^2 \nu(dz)  \leq L\big\{ (1+|y|+|\bar{y}|)^{\chi} |y-\bar{y}|^2 +\mathcal{W}_2^{2}(\mu',\bar{\mu}')\big\}
\end{align*}
for any $t\in[0,T]$, $x, y, \bar{y} \in\mathbb{R}^d$ and $\mu, \mu', \bar{\mu}'\in\mathcal{P}_2(\mathbb{R}^d)$. 
\end{asAtt}
\begin{asAtt} \label{astt:poly:gam:pj}
There exist  constants $L>0$ and $\chi>0$ such that for every $j\in\{1,\ldots,M\}$, 
\begin{align*}
\int_Z |h_t(x,y,\mu,\mu')|^{q}\nu(dz) \leq L\big\{1+|x|^{\chi/2+q}+|y|^{\chi + q} + \mathcal{W}_2^{q} (\mu,\delta_0)+ & \mathcal{W}^{q}_{2} (\mu',\delta_0)\big\}, 
\\
& 2\leq q\leq \bar{\bar{p}}_0 
\end{align*}
for any $t\in[0,T]$, $x, y \in\mathbb{R}^d$ and $\mu, \mu' \in\mathcal{P}_2(\mathbb{R}^d)$. 
\end{asAtt}
\begin{asAtt} \label{astt:initial:rate}
There exists a constant $L>0$ such that $\sup_{t\in[0,\tau]}\mathbb{E}|\xi_t-\xi_{\kappa_n(t)}|^{2+\epsilon}\leq L n^{-2}$. 
\end{asAtt}
\begin{rem} \label{rem:poly:lip:nondelay}
Due to Assumptions $\bar{\bar{A}}$--\ref{astt:monotonicity:delay:rate} and $\bar{\bar{A}}$--\ref{astt:poly:Lips:f:delay}, there exists a constant $L>0$ such that
\begin{align*}
|g_t(x,&\, y,  \mu, \mu')-g_t(\bar{x},y,\bar{\mu}, \mu')|^2+\int_Z |h_t(x,y,\mu, \mu', z)-h_t(\bar{x},y,\bar{\mu}, \mu', z)|^2 \nu(dz) 
\\
& \leq L\big\{ (1+|x|+|\bar{x}|)^{\chi/2} |x-\bar{x}|^2 +\mathcal{W}_2^2(\mu,\bar{\mu})\big\}
\end{align*}
for any $t\in[0,T]$, $x,\bar{x}, y\in\mathbb{R}^d$ and $\mu, \bar{\mu}, \mu'\in\mathcal{P}_2(\mathbb{R}^d)$. 
\end{rem}
\begin{rem} \label{rem:growth:delay}
Due to Assumptions  $\bar{\bar{A}}$--\ref{astt:poly:Lips:f:delay} and $\bar{\bar{A}}$--\ref{astt:poly:lips:delpar} and Remark \ref{rem:poly:lip:nondelay}, there exists a constant $L>0$ such that
\begin{align*}
|f_t(x, y,  \mu, \mu')| & \leq L\big\{1+|x|^{\chi/2+1}+|y|^{\chi/2+1} + \mathcal{W}_2 (\mu,\delta_0)+  \mathcal{W}_2 (\mu',\delta_0)\big\}
\\
|g_t(x, y,  \mu, \mu')| & \leq L\big\{1+|x|^{\chi/4+1}+|y|^{\chi/2+1} + \mathcal{W}_2 (\mu,\delta_0)+  \mathcal{W}_2 (\mu',\delta_0)\big\}
\\
\int_Z |h_t(x, y,  \mu, \mu', z)|^2 \nu(dz) & \leq L\big\{1+|x|^{\chi/2+2}+|y|^{\chi+2} + \mathcal{W}_2^2 (\mu,\delta_0)+  \mathcal{W}_2^{2} (\mu',\delta_0)\big\} 
\end{align*}
for any $t\in[0,T]$, $x, y \in\mathbb{R}^d$ and $\mu, \mu' \in\mathcal{P}_2(\mathbb{R}^d)$. 
\end{rem}
\begin{rem}
Notice that Assumption $\bar{\bar{A}}$--\ref{astt:monotonicity:delay} follows from Assumption $\bar{\bar{A}}$--\ref{astt:monotonicity:delay:rate}. Thus, Theorem \ref{thm:eu:delay} and Proposition \ref{prop:propagation:delay} hold if  Assumptions $\bar{\bar{A}}$--\ref{astt:initial:delay}, $\bar{\bar{A}}$--\ref{astt:coercivity:delay}, $\bar{\bar{A}}$--\ref{astt:coercivity:delay:pj} and $\bar{\bar{A}}$--\ref{astt:monotonicity:delay:rate} are satisfied. 
\end{rem}

We propose the following tamed Euler scheme for the interacting particle system connected with the McKean--Vlasov SDDE \eqref{eq:sdedelay},  
\begin{align} \label{eq:scheme:delay}
 y_t^{i,N, n}  & =y_0+\int^t_0 \frac{f_{\kappa_n(s)}(y_{\kappa_n(s)}^{i,N,n}, y_{\kappa_n(s-\tau)}^{i,N,n}, \mu_{\kappa_n(s)}^{y,N,n}, \mu_{\kappa_n(s-\tau)}^{y,N,n})}{1+n^{-1/2}|y_{\kappa_n(s)}^{i,N,n}|^{\chi}+ n^{-1/2}|y_{\kappa_n(s-\tau)}^{i,N,n}|^{2\chi}}ds \notag
\\
& \quad +\int^t_0 \frac{g_{\kappa_n(s)}(y_{\kappa_n(s)}^{i,N,n}, y_{\kappa_n(s-\tau)}^{i,N,n},  \mu_{\kappa_n(s)}^{y,N, n}, \mu_{\kappa_n(s-\tau)}^{y,N, n})}{1+n^{-1/2}|y_{\kappa_n(s)}^{i,N,n}|^{\chi}+n^{-1/2}|y_{\kappa_n(s-\tau)}^{i,N,n}|^{2\chi}}dw_s^i   \notag
\\
& \quad +\int^t_0 \int_Z \frac{h_{\kappa_n(s)}(y_{\kappa_n(s)}^{i,N,n}, y_{\kappa_n(s-\tau)}^{i,N,n},  \mu_{\kappa_n(s)}^{y,N,n}, \mu_{\kappa_n(s-\tau)}^{y,N,n}, z)}{1+n^{-1/2}|y_{\kappa_n(s)}^{i,N,n}|^{\chi}+n^{-1/2}|y_{\kappa_n(s-\tau)}^{i,N,n}|^{2\chi}} \tilde{n}_p^i(ds,dz),
\end{align} 
where $\mu_{\kappa_n(t)}^{y,N,n}$ is an empirical measure given by
\begin{align*}
\mu_{\kappa_n(t)}^{y,N,n}(\cdot):= \frac{1}{N} \sum_{i=1}^N \delta_{y_{\kappa_n(t)}^{i,N,n}}(\cdot)
\end{align*}
almost surely for any $t\in[0,T]$.
\begin{thm}[\textbf{Tamed Euler Scheme}]
Let Assumptions $\bar{\bar{A}}$--\ref{astt:initial:delay}, $\bar{\bar{A}}$--\ref{astt:coercivity:delay}, $\bar{\bar{A}}$--\ref{astt:coercivity:delay:rate} to
$\bar{\bar{A}}$--\ref{astt:initial:rate} be satisfied. 
Then, the tamed Euler scheme  \eqref{eq:scheme:delay} converges in mean-square sense to the true solution of the interacting particle system  \eqref{eq:sdedelay:interacting} connected with the McKean--Vlasov SDE \eqref{eq:sdedelay}, 
and for any $\epsilon>0$ such that $p_0 \ge \chi (2+\epsilon)/\epsilon$,
we have
\begin{align*}
\sup_{i\in\{1,\ldots,N\}} \sup_{t\in[0,T]}\mathbb{E}|y_t^{i,N}-y_t^{i,N,n}|^2 \leq K n^{-2/(2+\epsilon)}
\end{align*}
for any $n, N\in\mathbb{N}$, where $K>$ is a  constant independent of $n$ and $N$.  
\end{thm}
\begin{proof}
The result is shown by applying Theorem \ref{thm:order-12} on the interacting particle system \eqref{eq:interacting} connected to the 
McKean--Vlasov SDE \eqref{eq:sde} with coefficients given in equation \eqref{eq:coefficients:delay} and on its tamed Euler scheme \eqref{eq:scheme:delay} with coefficients given by
\begin{align}
b^n_t(x,\mu)& :=\frac{f_t(x, y_{\kappa_n(t-\tau)}^{i,N,n}, \mu, \mu_{\kappa_n(t-\tau)}^{y,N,n})}{1+n^{-1/2}|x|^{\chi}+n^{-1/2}|y_{\kappa_n(t-\tau)}^{i,N,n}|^{2\chi}}, \,\, \sigma^n_t(x,\mu):=\frac{g_t(x, y_{\kappa_n(t-\tau)}^{i,N,n}, \mu, \mu_{\kappa_n(t-\tau)}^{y,N,n})}{1+n^{-1/2}|x|^{\chi}+n^{-1/2}|y_{\kappa_n(t-\tau)}^{i,N,n}|^{2\chi}} \notag
\\
&\quad  \gamma^n_t(x,\mu, z):=\frac{h_t(x, y_{\kappa_n(t-\tau)}^{i,N,n}, \mu, \mu_{\kappa_n(t-\tau)}^{y,N,n},z)}{1+n^{-1/2}|x|^{\chi}+n^{-1/2}|y_{\kappa_n(t-\tau)}^{i,N,n}|^{2\chi}} \label{eq:coeff:tame:delay}
\end{align}
for any $t\in[0,T]$, $i\in\{1,\ldots,N\}$, $x\in\mathbb{R}^d$ and $\mu\in\mathcal{P}_2(\mathbb{R}^d)$ and  $n, N\in\mathbb{N}$. 

As before, we apply an inductive argument, \textit{i.e.}, we verify Assumptions $A$--\ref{as:initial}, $A$--\ref{as:coercivity}, $A$--\ref{as:coercivity:p0} to
 $A$--\ref{as:bounded}, $B$--\ref{asb:coercivity:scheme}, $B$--\ref{asb:growth:n} and $B$--\ref{asb:diff:rate}  
 on the interval $[(j-1)\tau, j\tau]$ and apply Theorem \ref{thm:order-12} for $j=1,\ldots,M$. 

First, notice that Assumptions $A$--\ref{as:initial}, $A$--\ref{as:coercivity}, $A$--\ref{as:continuity} and $A$--\ref{as:coercivity:p0} have already been verified in the proof of Theorem \ref{thm:eu:delay}. 
Also, Assumptions $A$--\ref{as:monotonicity:rate},  $A$--\ref{as:poly:Lips:b} and $A$--\ref{as:holder:time}  hold from  Assumptions $\bar{\bar{A}}$--\ref{astt:monotonicity:delay:rate}, $\bar{\bar{A}}$--\ref{astt:poly:Lips:f:delay}  and $\bar{\bar{A}}$--\ref{astt:holder:time:delay} respectively. 
Thus, it remains to verify Assumptions  $A$--\ref{as:bounded}, $B$--\ref{asb:coercivity:scheme}, $B$--\ref{asb:growth:n}   and $B$--\ref{asb:diff:rate}. 
\newline
\textbf{Case} $\mathbf{j=1.}$ 
To verify Assumption $A$--\ref{as:bounded}, one uses Remark \ref{rem:growth:delay} to obtain
\begin{align*}
|b_t(0,\delta_0)|^2= |f_t(0, y_{t-\tau}, \delta_0, L_{y_{t-\tau}})|^2 \leq L \{1+|y_{t-\tau}|^{\chi+2} + \mathcal{W}_2^{2}(L_{y_{t-\tau}}, \delta_0)\} =:L M_t
\end{align*}
almost surely for any $t\in[0,\tau]$, where $M_t=1+|\xi_{t}|^{\chi+2} + \mathbb{E}|\xi_{t}|^{2}$ satisfies $\sup_{t\in[0,\tau]}\mathbb{E}M_t< \infty$ due to Assumption $\bar{\bar{A}}$--\ref{astt:initial:delay}. Similarly,  
$$
|\sigma_t(0,\delta_0)|^2 \vee \int_Z |\gamma_t(0,\delta_0,z)|^2\nu(dz) \leq  L M_t 
$$
almost surely for any $t\in[0,\tau]$.  

For Assumption $B$--\ref{asb:coercivity:scheme}, using Assumption $\bar{\bar{A}}$--\ref{astt:coercivity:delay:rate} with $j=1$ (\textit{i.e.} $p_1$), one has
\begin{align*}
F_2:= &2|x|^{p_1-2}   xb_t^n(x, \mu) +(p_1-1) |x|^{p_1-2} |\sigma_t^n(x,\mu)|^2 
\\
& \quad + 2(p_1-1) \int_Z |\gamma_t^n(x,\mu,z)|^2 \int_0^1 |x+\theta\gamma_t^n(x,\mu,z)|^{p_1-2} d\theta  \nu(dz)
\\
=   &  \frac{2|x|^{p_1-2} x f_t(x, y_{\kappa_n(t-\tau)}^{i,N,n}, \mu, \mu_{\kappa_n(t-\tau)}^{y,N,n})}{1+n^{-1/2}|x|^{\chi}+n^{-1/2}|y_{\kappa_n(t-\tau)}^{i,N,n}|^{2\chi}} +  \frac{(p_1-1) |x|^{p_1-2} |g_t(x, y_{\kappa_n(t-\tau)}^{i,N,n}, \mu, \mu_{\kappa_n(t-\tau)}^{y,N,n})|^2}{(1+n^{-1/2}|x|^{\chi}+n^{-1/2}|y_{\kappa_n(t-\tau)}^{i,N,n}|^{2\chi})^2}
\\
& + 2(p_1-1) \int_Z \frac{|h_t(x, y_{\kappa_n(t-\tau)}^{i,N,n}, \mu, \mu_{\kappa_n(t-\tau)}^{y,N,n},z)|^2}{(1+n^{-1/2}|x|^{\chi}+n^{-1/2}|y_{\kappa_n(t-\tau)}^{i,N,n}|^{2\chi})^2} 
\\
& \qquad \times \int_0^1 \Big|x+ \frac{\theta h_t(x, y_{\kappa_n(t-\tau)}^{i,N,n}, \mu, \mu_{\kappa_n(t-\tau)}^{y,N,n},z)}{1+n^{-1/2}|x|^{\chi}+n^{-1/2}|y_{\kappa_n(t-\tau)}^{i,N,n}|^{2\chi}}\Big|^{p_1-2} d\theta  \nu(dz)
\\
\leq & \frac{2|x|^{p_1-2} x f_t(x, y_{\kappa_n(t-\tau)}^{i,N,n}, \mu, \mu_{\kappa_n(t-\tau)}^{y,N,n})}{1+n^{-1/2}|x|^{\chi}+n^{-1/2}|y_{\kappa_n(t-\tau)}^{i,N,n}|^{2\chi}} +  \frac{(p_1-1) |x|^{p_1-2} |g_t(x, y_{\kappa_n(t-\tau)}^{i,N,n}, \mu, \mu_{\kappa_n(t-\tau)}^{y,N,n})|^2}{(1+n^{-1/2}|x|^{\chi}+n^{-1/2}|y_{\kappa_n(t-\tau)}^{i,N,n}|^{2\chi})^2}
\\
& + (p_1-1)2^{p_1-3} \int_Z \frac{|h_t(x, y_{\kappa_n(t-\tau)}^{i,N,n}, \mu, \mu_{\kappa_n(t-\tau)}^{y,N,n},z)|^2}{(1+n^{-1/2}|x|^{\chi}+n^{-1/2}|y_{\kappa_n(t-\tau)}^{i,N,n}|^{2\chi})^2} 
\\
& \qquad \times \Big\{|x|^{p_1-2}+ \frac{|h_t(x, y_{\kappa_n(t-\tau)}^{i,N,n}, \mu, \mu_{\kappa_n(t-\tau)}^{y,N,n},z)|^{p_1-2}}{(1+n^{-1/2}|x|^{\chi}+n^{-1/2}|y_{\kappa_n(t-\tau)}^{i,N,n}|^{2\chi})^{p_1-2}} \Big\}\nu(dz)
\\
\leq & \frac{1}{1+n^{-1/2}|x|^{\chi}+n^{-1/2}|y_{\kappa_n(t-\tau)}^{i,N,n}|^{2\chi}}\Big\{ 2|x|^{p_1-2} x f_t(x, y_{\kappa_n(t-\tau)}^{i,N,n}, \mu, \mu_{\kappa_n(t-\tau)}^{y,N,n})
\\
&+(p_1-1) |x|^{p_1-2} |g_t(x, y_{\kappa_n(t-\tau)}^{i,N,n}, \mu, \mu_{\kappa_n(t-\tau)}^{y,N,n})|^2
\\
&+ (p_1-1)2^{p_1-3} \int_Z \big\{|h_t(x, y_{\kappa_n(t-\tau)}^{i,N,n}, \mu, \mu_{\kappa_n(t-\tau)}^{y,N,n},z)|^2 |x|^{p_1-2} + |h_t(x, y_{\kappa_n(t-\tau)}^{i,N,n}, \mu, \mu_{\kappa_n(t-\tau)}^{y,N,n},z)|^{p_1} \big\} \nu(dz)\Big\},
\end{align*}
which due to Assumption $\bar{\bar{A}}$--\ref{astt:coercivity:delay:rate} yields
\begin{align*}
F_2 &\leq L\big\{1+|x|^{p_1} +|y_{\kappa_n(t-\tau)}^{i,N,n}|^{\bar{\bar{p}}_0} + \mathcal{W}_2^{p_1}(\mu,\delta_0)+ \mathcal{W}_2^{\bar{\bar{p}}_0}(\mu_{\kappa_n(t-\tau)}^{y,N,n},z),\delta_0)\big\}
\\
& =:L \big\{\bar{M}_t^n +|x|^{p_1} + \mathcal{W}_2^{p_1}(\mu,\delta_0)\big\}
\end{align*}
almost surely for any $t\in[0,\tau]$, $i\in\{1,\ldots,N\}$ and $n,N\in\mathbb{N}$, where $\bar{M}_t^n=1+|\xi_{\kappa_n(t)}|^{ \bar{\bar{p}}_0}+\mathbb{E}|\xi_{\kappa_n(t)}|^{p_1}$ satisfies $\sup_{n\in\mathbb{N}}\sup_{t\in [0,\tau]}\mathbb{E}\bar{M}^n_t<\infty$ due to Assumption $\bar{\bar{A}}$--\ref{astt:initial:delay}.

We now proceed to verify Assumption $B$--\ref{asb:growth:n}. For this, one uses Remark \ref{rem:growth:delay} to obtain
\begin{align*}
|b^n_t(x,\mu)|^2& \leq \frac{|f_t(x, y_{\kappa_n(t-\tau)}^{i,N,n}, \mu, \mu_{\kappa_n(t-\tau)}^{y,N,n})|^2}{1+n^{-1/2}|x|^{\chi}+n^{-1/2}|y_{\kappa_n(t-\tau)}^{i,N,n}|^{2\chi}} 
\\
& \leq L \frac{1+|x|^{\chi+2}+|y_{\kappa_n(t-\tau)}^{i,N,n}|^{\chi+2}+\mathcal{W}_2^2 (\mu,\delta_0)+ \mathcal{W}_2^{2} (\mu_{\kappa_n(t-\tau)}^{y,N,n},\delta_0)}{1+n^{-1/2}|x|^{\chi}+n^{-1/2}|y_{\kappa_n(t-\tau)}^{i,N,n}|^{2\chi}},
\end{align*}
which implies, for any $t\in[0,\tau]$, $x\in\mathbb{R}^d$ and $\mu\in\mathcal{P}_2(\mathbb{R}^d)$, 
\begin{align*}
|b^n_t(x,\mu)| \leq K n^{1/4} \{ M_t^n+|x|+ \mathcal{W}_2 (\mu,\delta_0)\},
\end{align*}
 where $M_t^n= 1+|\xi_{\kappa_n(t)}|+ \big(\mathbb{E}|\xi_{\kappa_n(t)}|^{2}\big)^{1/2}$ satisfies $\sup_{n\in\mathbb{N}}\sup_{t\in[0,\tau]}\mathbb{E}~(M^n_t)^{p_1}<\infty$ due to  Assumption $\bar{\bar{A}}$--\ref{astt:initial:delay}. 
Furthermore, the application of Remark \ref{rem:growth:delay} yields 
\begin{align*}
|\sigma^n_t(x,\mu)|^4& \leq \frac{|g_t(x, y_{\kappa_n(t-\tau)}^{i,N,n}, \mu, \mu_{\kappa_n(t-\tau)}^{y,N,n})|^4}{1+n^{-1/2}|x|^{\chi}+n^{-1/2}|y_{\kappa_n(t-\tau)}^{i,N,n}|^{2\chi}}
\\
& \leq L \frac{1+|x|^{\chi+4}+|y_{\kappa_n(t-\tau)}^{i,N,n}|^{2\chi+4}+ \mathcal{W}_2^4(\mu,\delta_0)+ \mathcal{W}_2^{4}(\mu_{\kappa_n(t-\tau)}^{y,N,n},\delta_0)}{1+n^{-1/2}|x|^{\chi}+n^{-1/2}|y_{\kappa_n(t-\tau)}^{i,N,n}|^{2\chi}},
\end{align*}
which implies, for any $t\in[0,\tau]$, $x\in\mathbb{R}^d$ and $\mu\in\mathcal{P}_2(\mathbb{R}^d)$, 
\begin{align*}
|\sigma^n_t(x,\mu)| \leq K n^{1/8} \{M_t^n+|x| + \mathcal{W}_2(\mu,\delta_0)\},
\end{align*}
where $M_t^n=1+ |\xi_{\kappa_n(t)}|+( \mathbb{E}|\xi_{\kappa_n(t)}|^2)^{1/2}$ satisfies $\sup_{n\in\mathbb{N}}\sup_{t\in[0,\tau]}\mathbb{E}(M_t^n)^{p_1}<\infty$ due to Assumption  $\bar{\bar{A}}$--\ref{astt:initial:delay}. 
Again, due to Remark \ref{rem:growth:delay}, 
\begin{align*}
\Big(\int_Z &|\gamma_t^n(x,\mu,z)|^2 \nu(dz) \Big)^2 \leq  \Big(\int_Z \frac{|h_t(x, y_{\kappa_n(t-\tau)}^{i,N,n}, \mu, \mu_{\kappa_n(t-\tau)}^{y,N,n},z)|^2}{1+n^{-1/2}|x|^{\chi}+n^{-1/2}|y_{\kappa_n(t-\tau)}^{i,N,n}|^{2\chi}} \nu(dz) \Big)^2 
\\
& \leq L \frac{1+|x|^{\chi+4}+|y_{\kappa_n(t-\tau)}^{i,N,n}|^{2\chi+4}+\mathcal{W}_2^4(\mu,\delta_0)+\mathcal{W}_2^{4}(\mu_{\kappa_n(t-\tau)}^{y,N,n},\delta_0)}{1+n^{-1/2}|x|^{\chi}+n^{-1/2}|y_{\kappa_n(t-\tau)}^{i,N,n}|^{2\chi}},
\end{align*}
which implies, for any $t\in[0,\tau]$, $x\in\mathbb{R}^d$ and $\mu\in\mathcal{P}_2(\mathbb{R}^d)$,
\begin{align*}
\int_Z |\gamma_t^n(x,\mu,z)|^2 \nu(dz)  \leq K n^{1/4} \{(M_t^n)^2+|x|^2+\mathcal{W}_2^2(\mu,\delta_0)\}
\end{align*}
where $M^n_t=1+|y_{\kappa_n(t-\tau)}^{i,N,n}| + (\mathbb{E}|y_{\kappa_n(t-\tau)}^{i,N,n}|^2)^{1/2}$ satisfies $\sup_{n\in\mathbb{N}}\sup_{t\in[0,\tau]}\mathbb{E}(M_t^n)^{p_1}<\infty$. 
Finally, for $2\leq q \leq \bar{\bar{p}}_0$,  Assumption $\bar{\bar{A}}$--\ref{astt:poly:gam:pj} yields
\begin{align*}
\Big(\int_Z & |\gamma_t^n(x,\mu,z)|^{q} \nu(dz) \Big)^2 \leq \Big(\int_Z \frac{|h_t(x, y_{\kappa_n(t-\tau)}^{i,N,n}, \mu, \mu_{\kappa_n(t-\tau)}^{y,N,n},z)|^{q}}{1+n^{-1/2}|x|^{\chi}+n^{-1/2}|y_{\kappa_n(t-\tau)}^{i,N,n}|^{2\chi}} \nu(dz) \Big)^2 
\\
& \leq L \frac{1+|x|^{\chi+2q}+ |y_{\kappa_n(t-\tau)}^{i,N,n}|^{2\chi + 2q}+\mathcal{W}_2^{2 q}(\mu,\delta_0)+\mathcal{W}^{2 q}_2(\mu_{\kappa_n(t-\tau)}^{y,N,n},\delta_0)}{1+n^{-1/2}|x|^{\chi}+n^{-1/2}|y_{\kappa_n(t-\tau)}^{i,N,n}|^{2\chi}}
\\
& \leq Ln^{1/2}  \big\{1+|x|^{2q}+ |y_{\kappa_n(t-\tau)}^{i,N,n}|^{2q}+\mathcal{W}_2^{2 q}(\mu,\delta_0)+\mathcal{W}^{2 q}_2(\mu_{\kappa_n(t-\tau)}^{y,N,n},\delta_0) \big\},
\end{align*}
which further implies, for any $t\in[0,\tau]$, $x\in\mathbb{R}^d$ and $\mu\in\mathcal{P}_2(\mathbb{R}^d)$,
\begin{align*}
\int_Z & |\gamma_t^n(x,\mu,z)|^{q} \nu(dz)  \leq Ln^{1/4} \{\bar{M}^n_t + |x| + \mathcal{W}_2(\mu,\delta_0)\}^q,
\end{align*}
where $\bar{M}^n_t=1+|\xi_{\kappa_n(t)}|+(\mathbb{E}|\xi_{\kappa_n(t)}|^{2})^{1/2}$, satisfying $\bE(M_t^n)^{q}<\infty$ due to Assumption $\bar{\bar{A}}$--\ref{astt:initial:delay}. Thus, Assumption $B$--\ref{asb:growth:n} is satisfied in  the interval $[0,\tau]$. 

As Assumptions $A$--\ref{as:initial}, $B$--\ref{asb:coercivity:scheme}  and $B$--\ref{asb:growth:n} hold in the interval $[0,\tau]$, thus by Lemma \ref{lem:scheme:mb}, 
\begin{align*}
\sup_{i \in\{1,\ldots, N\}}\sup_{t\in[0,\tau]}\mathbb{E}|y_t^{i,N,n}|^{p_1} \leq K
\end{align*}   
for any $n, N\in\mathbb{N}$, where the constant $K>0$ does not depend on $n$ and $N$.

We now proceed to verify Assumption $B$--\ref{asb:diff:rate}. From equations \eqref{eq:coefficients:delay} and \eqref{eq:coeff:tame:delay}, one can write
\begin{align*}
 \mathbb{E}& \int_0^\tau |b_{\kappa_n(s)}(y_{\kappa_n(s)}^{i, N,n}, \mu_{\kappa_n(s)}^{y,N,n})- b^n_{\kappa_n(s)}(y_{\kappa_n(s)}^{i, N, n}, \mu_{\kappa_n(s)}^{y,N, n})|^2 ds 
 \\
 & =   \mathbb{E} \int_0^\tau \Big|f_{\kappa_n(s)}(y_{\kappa_n(s)}^{i, N,n}, y_{s-\tau}^{i,N}, \mu_{\kappa_n(s)}^{y,N,n}, \mu_{s-\tau}^{y,N})- \frac{f_{\kappa_n(s)}(y_{\kappa_n(s)}^{i, N,n}, y_{\kappa_n(s-\tau)}^{i,N,n}, \mu_{\kappa_n(s)}^{y,N,n}, \mu_{\kappa_n(s-\tau)}^{y,N,n})}{1+n^{-1/2}|y_{\kappa_n(s)}^{i, N,n}|^{\chi}+n^{-1/2}|y_{\kappa_n(s-\tau)}^{i,N,n}|^{2\chi}}\Big|^2 ds 
 \\
 & \leq  K \mathbb{E} \int_0^\tau \Big|f_{\kappa_n(s)}(y_{\kappa_n(s)}^{i, N,n}, y_{s-\tau}^{i,N}, \mu_{\kappa_n(s)}^{y,N,n}, \mu_{s-\tau}^{y,N})- f_{\kappa_n(s)}(y_{\kappa_n(s)}^{i, N,n}, y_{s-\tau}^{i,N,n}, \mu_{\kappa_n(s)}^{y,N,n}, \mu_{s-\tau}^{y,N,n})\Big|^2 ds 
  \\
 & + K \mathbb{E} \int_0^\tau \Big|f_{\kappa_n(s)}(y_{\kappa_n(s)}^{i, N,n}, y_{s-\tau}^{i,N,n}, \mu_{\kappa_n(s)}^{y,N,n}, \mu_{s-\tau}^{y,N,n})- f_{\kappa_n(s)}(y_{\kappa_n(s)}^{i, N,n}, y_{\kappa_n(s-\tau)}^{i,N,n}, \mu_{\kappa_n(s)}^{y,N,n}, \mu_{\kappa_n(s-\tau)}^{y,N,n})\Big|^2 ds 
 \\
 & +  K \mathbb{E} \int_0^\tau \Big|f_{\kappa_n(s)}(y_{\kappa_n(s)}^{i, N,n}, y_{\kappa_n(s-\tau)}^{i,N,n}, \mu_{\kappa_n(s)}^{y,N,n}, \mu_{\kappa_n(s-\tau)}^{y,N,n})- \frac{f_{\kappa_n(s)}(y_{\kappa_n(s)}^{i, N,n}, y_{\kappa_n(s-\tau)}^{i,N,n}, \mu_{\kappa_n(s)}^{y,N,n}, \mu_{\kappa_n(s-\tau)}^{y,N,n})}{1+n^{-1/2}|y_{\kappa_n(s)}^{i, N,n}|^{\chi}+n^{-1/2}|y_{\kappa_n(s-\tau)}^{i,N,n}|^{2\chi}}\Big|^2 ds 
\end{align*}
and then uses Assumption $\bar{\bar{A}}$--\ref{astt:poly:lips:delpar} and Remark \ref{rem:growth:delay} to obtain
\begin{align}
 \mathbb{E}& \int_0^\tau |b_{\kappa_n(s)}(y_{\kappa_n(s)}^{i, N,n}, \mu_{\kappa_n(s)}^{y,N,n})- b^n_{\kappa_n(s)}(y_{\kappa_n(s)}^{i, N, n}, \mu_{\kappa_n(s)}^{y,N, n})|^2 ds  \notag
 \\
 & \leq  K \mathbb{E} \int_0^\tau \big\{(1+| y_{s-\tau}^{i,N}|+|y_{s-\tau}^{i,N,n}|)^{\chi} | y_{s-\tau}^{i,N} - y_{s-\tau}^{i,N,n}|^2 + \mathcal{W}_2^2(\mu_{s-\tau}^{y,N}, \mu_{s-\tau}^{y,N,n})\big\} ds \notag
  \\
 & + K \mathbb{E} \int_0^\tau \big\{ (1+|y_{s-\tau}^{i,N,n}|+|y_{\kappa_n(s-\tau)}^{i,N,n}|)^{\chi}|y_{s-\tau}^{i,N,n}-y_{\kappa_n(s-\tau)}^{i,N,n}|^2+\mathcal{W}_2^2(\mu_{s-\tau}^{y,N,n}, \mu_{\kappa_n(s-\tau)}^{y,N,n}) \big\} ds \notag
 \\
& +  K n^{-1}\mathbb{E} \int_0^\tau \{|y_{\kappa_n(s)}^{i, N,n}|^{2\chi}+|y_{\kappa_n(s-\tau)}^{i,N,n}|^{4\chi}\} \notag
\\
&\qquad \times\big \{ 1+ |y_{\kappa_n(s)}^{i, N,n}|^{\chi+2} + |y_{\kappa_n(s-\tau)}^{i,N,n}|^{\chi+2} + \mathcal{W}_2^2 (\mu_{\kappa_n(s)}^{y,N,n}, \delta_0)+ \mathcal{W}_2^2(\mu_{\kappa_n(s-\tau)}^{y,N,n}, \delta_0) \big\} ds \label{eq:re:}
\end{align}
for any $n, N\in\mathbb{N}$ and $i\in\{1,\ldots,N\}$. 
Notice that the first term on the right side of the above inequality is zero as $y_{s-\tau}^{i,N}=y_{s-\tau}^{i,N,n}=\xi_{s}$ for any $s\in[0,\tau]$. In the second term, one uses Assumption $\bar{\bar{A}}$--\ref{astt:initial:rate} and H\"older's inequality and on the third term, one uses equation \eqref{eq:mb:delay} to obtain the following, 
\begin{align*}
 \mathbb{E}& \int_0^\tau |b_{\kappa_n(s)}(y_{\kappa_n(s)}^{i, N,n}, \mu_{\kappa_n(s)}^{y,N,n})- b^n_{\kappa_n(s)}(y_{\kappa_n(s)}^{i, N, n}, \mu_{\kappa_n(s)}^{y,N, n})|^2 ds \leq K n^{-2/(2+\epsilon)}
\end{align*}
for any $n, N\in\mathbb{N}$ and $i\in\{1,\ldots,N\}$, where the constant $K>0$ does not depend on $n$ and $N$. Further, from equations \eqref{eq:coefficients:delay} and \eqref{eq:coeff:tame:delay}, 
\begin{align*}
 \mathbb{E}& \int_0^\tau |\sigma_{\kappa_n(s)}(y_{\kappa_n(s)}^{i, N,n}, \mu_{\kappa_n(s)}^{y,N,n})- \sigma^n_{\kappa_n(s)}(y_{\kappa_n(s)}^{i, N, n}, \mu_{\kappa_n(s)}^{y,N, n})|^2 ds 
 \\
 & \leq  K \mathbb{E} \int_0^\tau \Big|g_{\kappa_n(s)}(y_{\kappa_n(s)}^{i, N,n}, y_{s-\tau}^{i,N}, \mu_{\kappa_n(s)}^{y,N,n}, \mu_{s-\tau}^{y,N})- g_{\kappa_n(s)}(y_{\kappa_n(s)}^{i, N,n}, y_{s-\tau}^{i,N,n}, \mu_{\kappa_n(s)}^{y,N,n}, \mu_{s-\tau}^{y,N,n})\Big|^2 ds 
  \\
 & + K \mathbb{E} \int_0^\tau \Big|g_{\kappa_n(s)}(y_{\kappa_n(s)}^{i, N,n}, y_{s-\tau}^{i,N,n}, \mu_{\kappa_n(s)}^{y,N,n}, \mu_{s-\tau}^{y,N,n})- g_{\kappa_n(s)}(y_{\kappa_n(s)}^{i, N,n}, y_{\kappa_n(s-\tau)}^{i,N,n}, \mu_{\kappa_n(s)}^{y,N,n}, \mu_{\kappa_n(s-\tau)}^{y,N,n})\Big|^2 ds 
 \\
 & +  K \mathbb{E} \int_0^\tau \Big|g_{\kappa_n(s)}(y_{\kappa_n(s)}^{i, N,n}, y_{\kappa_n(s-\tau)}^{i,N,n}, \mu_{\kappa_n(s)}^{y,N,n}, \mu_{\kappa_n(s-\tau)}^{y,N,n})- \frac{g_{\kappa_n(s)}(y_{\kappa_n(s)}^{i, N,n}, y_{\kappa_n(s-\tau)}^{i,N,n}, \mu_{\kappa_n(s)}^{y,N,n}, \mu_{\kappa_n(s-\tau)}^{y,N,n})}{1+n^{-1/2}|y_{\kappa_n(s)}^{i, N,n}|^{\chi}+n^{-1/2}|y_{\kappa_n(s-\tau)}^{i,N,n}|^{2\chi}}\Big|^2 ds,
\end{align*}
which on using  Assumption $\bar{\bar{A}}$--\ref{astt:poly:lips:delpar} and Remark \ref{rem:growth:delay} yields
\begin{align*}
 \mathbb{E}& \int_0^\tau |\sigma_{\kappa_n(s)}(y_{\kappa_n(s)}^{i, N,n}, \mu_{\kappa_n(s)}^{y,N,n})- \sigma^n_{\kappa_n(s)}(y_{\kappa_n(s)}^{i, N, n}, \mu_{\kappa_n(s)}^{y,N, n})|^2 ds 
 \\
 & \leq  K \mathbb{E} \int_0^\tau \big\{(1+| y_{s-\tau}^{i,N}|+|y_{s-\tau}^{i,N,n}|)^{\chi} | y_{s-\tau}^{i,N} - y_{s-\tau}^{i,N,n}|^2 + \mathcal{W}_2^2(\mu_{s-\tau}^{y,N}, \mu_{s-\tau}^{y,N,n})\big\} ds 
  \\
 & + K \mathbb{E} \int_0^\tau \big\{ (1+|y_{s-\tau}^{i,N,n}|+|y_{\kappa_n(s-\tau)}^{i,N,n}|)^{\chi}|y_{s-\tau}^{i,N,n}-y_{\kappa_n(s-\tau)}^{i,N,n}|^2+\mathcal{W}_2^2(\mu_{s-\tau}^{y,N,n}, \mu_{\kappa_n(s-\tau)}^{y,N,n}) \big\} ds
 \\
& +  K n^{-1}\mathbb{E} \int_0^\tau \{|y_{\kappa_n(s)}^{i, N,n}|^{2\chi}+|y_{\kappa_n(s-\tau)}^{i,N,n}|^{4\chi}\}
\\
&\qquad \times\big \{ 1+ |y_{\kappa_n(s)}^{i, N,n}|^{\chi/2+2} + |y_{\kappa_n(s-\tau)}^{i,N,n}|^{\chi+2} + \mathcal{W}_2^2 (\mu_{\kappa_n(s)}^{y,N,n}, \delta_0)+ \mathcal{W}_2^2(\mu_{\kappa_n(s-\tau)}^{y,N,n}, \delta_0) \big\} ds
\end{align*}
and then one uses  H\"older's inequality, Assumption $\bar{\bar{A}}$--\ref{astt:initial:rate} and equation \eqref{eq:mb:delay} to obtain
\begin{align*}
\mathbb{E}& \int_0^\tau |\sigma_{\kappa_n(s)}(y_{\kappa_n(s)}^{i, N,n}, \mu_{\kappa_n(s)}^{y,N,n})- \sigma^n_{\kappa_n(s)}(y_{\kappa_n(s)}^{i, N, n}, \mu_{\kappa_n(s)}^{y,N, n})|^2 ds \leq K n^{-2/(2+\epsilon)}
\end{align*}
for any $n, N\in\mathbb{N}$ and $i\in\{1,\ldots,N\}$, where the constant $K>0$ does not depend on $n$ and $N$. Again, from equations \eqref{eq:coefficients:delay} and \eqref{eq:coeff:tame:delay} along with Assumption $\bar{\bar{A}}$--\ref{astt:poly:lips:delpar}, Remark \ref{rem:growth:delay},  H\"older's inequality, Assumption~$\bar{\bar{A}}$--\ref{astt:initial:rate} and equation \eqref{eq:mb:delay}, 
\begin{align*}
  \mathbb{E} \int_0^\tau \int_Z |\gamma_{\kappa_n(s)}(y_{\kappa_n(s)}^{i, N,n}, \mu_{\kappa_n(s)}^{y,N,n}, z)- \gamma^n_{\kappa_n(s)}(y_{\kappa_n(s)}^{i, N, n}, \mu_{\kappa_n(s)}^{y,N, n}, z)|^2 \nu(dz) ds \leq K n^{-2/(2+\epsilon)}
\end{align*}
for any $n, N\in\mathbb{N}$ and $i\in\{1,\ldots,N\}$, where the constant $K>0$ does not depend on $n$ and $N$. This completes the verification of Assumption  $B$--\ref{asb:diff:rate}. 

Thus, the application of Corollary \ref{cor:one-step:rate} and Theorem \ref{thm:order-12}  on the interval $[0,\tau]$ yields
\begin{align*}
\sup_{i\in \{1,\ldots,N\}}\sup_{t\in [0,\tau]}\mathbb{E}|y_t^{i, N,n}-y_{\kappa_n(t)}^{i, N, n}|^2 \leq K n^{-2/(2+\epsilon)},
\\
\sup_{i\in \{1,\ldots,N\}}\sup_{t\in [0,\tau]}\mathbb{E}|y_t^{i, N}-y_t^{i, N, n}|^2 \leq K n^{-2/(2+\epsilon)}
\end{align*}
for any $n, N\in\mathbb{N}$ where the positive constant $K$ is independent of $n$ and $N$.

For the inductive argument, let us assume that the result  holds for $j=\ell$ and thus 
\begin{align}
\sup_{i \in\{1,\ldots, N\}}\sup_{t\in [(\ell-1)\tau,\ell \tau]}\mathbb{E}|y_t^{i,N,n}|^{p_\ell} & \leq K, \label{eq:mb:ind}
\\
\sup_{i\in \{1,\ldots,N\}}\sup_{t\in [(\ell-1)\tau,\ell \tau]}\mathbb{E}|y_t^{i, N, n}-y_{\kappa_n(t)}^{i, N, n}|^2 & \leq K n^{-2/(2+\epsilon)}, \label{eq:rate:indmm}
\\
\sup_{i\in \{1,\ldots,N\}}\sup_{t\in [(\ell-1)\tau,\ell \tau]}\mathbb{E}|y_t^{i, N}-y_t^{i, N, n}|^2 & \leq K n^{-2/(2+\epsilon)} \label{eq:rate:ind}
\end{align}
for any $n, N\in\mathbb{N}$, where the positive constant $K$ is independent of $n$ and $N$. 
\newline
\textbf{Case $\mathbf{j=\ell+1.}$} 
By using Remark \ref{rem:growth:delay},
\begin{align*}
|b_t(0,\delta_0)|^2= |f_t(0, y_{t-\tau}, \delta_0, L_{y_{t-\tau}})|^2 \leq L \{1+|y_{t-\tau}|^{\chi+2} + \mathcal{W}_2^{2}(L_{y_{t-\tau}}, \delta_0)\} =:L M_t
\end{align*}
almost surely for any $t\in[\ell\tau,(\ell+1)\tau]$ where $M_t=1+| y_{t-\tau}|^{\chi+2} + \mathbb{E}|y_{t-\tau}|^{2}$ satisfies $\sup_{t\in[\ell\tau,(\ell+1)\tau]}\mathbb{E}M_t= 1+\sup_{t\in[(\ell-1)\tau,\ell\tau]} \big\{\mathbb{E} |y_{t}|^{\chi+2} + \mathbb{E}|y_{t}|^{2}\big\}<\infty$ due to equation \eqref{eq:**}. Similarly,  
$$
|\sigma_t(0,\delta_0)|^2 \vee \int_Z |\gamma_t(0,\delta_0,z)|^2\nu(dz) \leq  L M_t 
$$
almost surely which implies Assumption $A$--\ref{as:bounded} holds on the interval $[\ell\tau,(\ell+1)\tau]$. 

For Assumption $B$--\ref{asb:coercivity:scheme}, we  adopt the arguments similar to the one used in the estimation of $F_2$ to obtain
\begin{align*}
2|x|^{p_{\ell+1}-2}  & xb_t^n(x, \mu) +(p_{\ell+1}-1) |x|^{p_{\ell+1}-2} |\sigma_t^n(x,\mu)|^2 
\\
& \quad + 2(p_{\ell+1}-1) \int_Z |\gamma_t^n(x,\mu,z)|^2 \int_0^1 |x+\theta\gamma_t^n(x,\mu,z)|^{p_{\ell+1}-2} d\theta  \nu(dz)
\\
& \leq L\big\{1+|x|^{p_{\ell+1}} +|y_{\kappa_n(t-\tau)}^{i,N,n}|^{p_\ell} + \mathcal{W}_2^{p_\ell}(\mu,\delta_0)+ \mathcal{W}_2^{p_{l}}(\mu_{\kappa_n(t-\tau)}^{y,N,n},z),\delta_0)\big\}
\\
& =:L \big\{\bar{M}_t^n +|x|^{p_{\ell+1}} + \mathcal{W}_2^{p_{\ell+1}}(\mu,\delta_0)\big\},
\end{align*}
 where $\bar{M}_t^n=1+|y_{\kappa_n(t-\tau)}^{i,N,n}|^{p_\ell} + \mathbb{E}|y_{\kappa_n(t-\tau)}^{i,N,n}|^{p_\ell}$ which due to equation \eqref{eq:mb:ind} satisfies 
 $$
 \sup_{n\in\mathbb{N}}\sup_{t\in[\ell\tau, (\ell+1)\tau]}\mathbb{E}M_t^n=1+ 2\sup_{n\in\mathbb{N}}\sup_{t\in[\ell\tau, (\ell+1)\tau]}\mathbb{E} |y_{\kappa_n(t-\tau)}^{i,N,n}|^{p_\ell} <\infty. 
 $$ 
Similarly, from equation  \eqref{eq:re:}, 
\begin{align*}
 \mathbb{E}& \int_{\ell \tau}^{(\ell+1)\tau} |b_{\kappa_n(s)}(y_{\kappa_n(s)}^{i, N,n}, \mu_{\kappa_n(s)}^{y,N,n})- b^n_{\kappa_n(s)}(y_{\kappa_n(s)}^{i, N, n}, \mu_{\kappa_n(s)}^{y,N, n})|^2 ds 
 \\
 & \leq  K \mathbb{E} \int_{\ell \tau}^{(\ell+1)\tau} \big\{(1+| y_{s-\tau}^{i,N}|+|y_{s-\tau}^{i,N,n}|)^{\chi} | y_{s-\tau}^{i,N} - y_{s-\tau}^{i,N,n}|^2 + \mathcal{W}_2^2(\mu_{s-\tau}^{y,N}, \mu_{s-\tau}^{y,N,n})\big\} ds 
  \\
 & + K \mathbb{E} \int_{\ell \tau}^{(\ell+1)\tau} \big\{ (1+|y_{s-\tau}^{i,N,n}|+|y_{\kappa_n(s-\tau)}^{i,N,n}|)^{\chi}|y_{s-\tau}^{i,N,n}-y_{\kappa_n(s-\tau)}^{i,N,n}|^2+\mathcal{W}_2^2(\mu_{s-\tau}^{y,N,n}, \mu_{\kappa_n(s-\tau)}^{y,N,n}) \big\} ds
 \\
& +  K n^{-1}\mathbb{E} \int_{\ell \tau}^{(\ell+1)\tau} \{|y_{\kappa_n(s)}^{i, N,n}|^{2\chi}+|y_{\kappa_n(s-\tau)}^{i,N,n}|^{4\chi}\}
\\
&\qquad \times\big \{ 1+ |y_{\kappa_n(s)}^{i, N,n}|^{\chi+2} + |y_{\kappa_n(s-\tau)}^{i,N,n}|^{\chi+2} + \mathcal{W}_2^2 (\mu_{\kappa_n(s)}^{y,N,n}, \delta_0)+ \mathcal{W}_2^2(\mu_{\kappa_n(s-\tau)}^{y,N,n}, \delta_0) \big\} ds,
\end{align*}
which due to H\"older's inequality, equations \eqref{eq:mb:ind}, \eqref{eq:rate:indmm} and \eqref{eq:rate:ind} yields
\begin{align*}
& \mathbb{E} \int_{\ell \tau}^{(\ell+1)\tau} |b_{\kappa_n(s)}(y_{\kappa_n(s)}^{i, N,n}, \mu_{\kappa_n(s)}^{y,N,n})- b^n_{\kappa_n(s)}(y_{\kappa_n(s)}^{i, N, n}, \mu_{\kappa_n(s)}^{y,N, n})|^2 ds 
 \\
\leq & K  \int_{\ell \tau}^{(\ell+1)\tau} \big\{\mathbb{E}(1+| y_{s-\tau}^{i,N}|+|y_{s-\tau}^{i,N,n}|)^{\chi (2+\epsilon)/\epsilon} \}^{\epsilon/(2+\epsilon)} \{\mathbb{E}| y_{s-\tau}^{i,N} - y_{s-\tau}^{i,N,n}|^{2+\epsilon}\}^{2/(2+\epsilon)} + \mathbb{E}| y_{s-\tau}^{i,N} - y_{s-\tau}^{i,N,n}|^{2}\big\} ds 
  \\
 &  + K  \int_{\ell \tau}^{(\ell+1)\tau} \Big\{\big\{ \mathbb{E}(1+|y_{s-\tau}^{i,N,n}|+|y_{\kappa_n(s-\tau)}^{i,N,n}|)^{\chi (2+\epsilon)/\epsilon}\big\}^{\epsilon/(2+\epsilon)} \big\{\mathbb{E}|y_{s-\tau}^{i,N,n}-y_{\kappa_n(s-\tau)}^{i,N,n}|^{2+\epsilon}\big\}^{2/(2+\epsilon)}
 \\
 & + \mathbb{E}|y_{s-\tau}^{i,N,n}-y_{\kappa_n(s-\tau)}^{i,N,n}|^{2+\epsilon} \Big\} ds +K n^{-1} \leq K n^{-2/(2+\epsilon)}
\end{align*}
for any $n, N\in\mathbb{N}$, where the constant $K>0$ does not depend on $n$ and $N$. Similarly, by adopting a similar arguments as above, we have 
\begin{align*}
 \mathbb{E} \int_{\ell \tau}^{(\ell+1)\tau} |\sigma_{\kappa_n(s)}(y_{\kappa_n(s)}^{i, N,n}, \mu_{\kappa_n(s)}^{y,N,n})- \sigma^n_{\kappa_n(s)}(y_{\kappa_n(s)}^{i, N, n}, \mu_{\kappa_n(s)}^{y,N, n})|^2 ds &\leq K n^{-2/(2+\epsilon)}
\\
 \mathbb{E}\int_{\ell \tau}^{(\ell+1)\tau} \int_Z |\gamma_{\kappa_n(s)}(y_{\kappa_n(s)}^{i, N,n}, \mu_{\kappa_n(s)}^{y,N,n}, z)- \gamma^n_{\kappa_n(s)}(y_{\kappa_n(s)}^{i, N, n}, \mu_{\kappa_n(s)}^{y,N, n}, z)|^2 \nu(dz) ds & \leq K n^{-2/(2+\epsilon)}
\end{align*}
for any $n, N\in\mathbb{N}$, where the constant $K>0$ does not depend on $n$ and $N$. Thus, Assumption $B$--\ref{asb:diff:rate} holds in the interval $[\ell\tau, (\ell+1)\tau]$, which implies that 
\begin{align*}
\sup_{i\in \{1,\ldots,N\}}\sup_{t\in [\ell\tau,(\ell+1) \tau]}\mathbb{E}|y_t^{i, N}-y_t^{i, N, n}|^2 & \leq K n^{-2/(2+\epsilon)} 
\end{align*}
for any $n, N\in\mathbb{N}$, where the constant $K>0$ does not depend on $n$ and $N$. This completes the proof. 
\end{proof}


\appendix

\section{Auxiliary Result}
\begin{lem} \label{lem:cms}
The space $\mathbb{D}([0,T]; \mathcal{P}_2(\mathbb{R}^d))$ of c\`adl\`ag functions taking values in $\mathcal{P}_2(\mathbb{R}^d)$ is a complete metric space under the metric given by
\begin{align} \label{eq:metric}
\rho(\mu, \bar{\mu}):=\sup_{t\in[0,T]} \mathcal{W}_2 (\mu_t, \bar{\mu}_t)
\end{align}
for any $\mu,\bar{\mu}\in \mathbb{D}([0,T]; \mathcal{P}_2(\mathbb{R}^d))$. 
\end{lem}
\begin{proof}
It is easy to verify that  $\rho$ is a metric. Indeed, $\rho(\mu,\bar{\mu})<\infty$ for any $\mu,\bar{\mu}\in \mathbb{D}([0,T]; \mathcal{P}_2(\mathbb{R}^d))$ since c\`adl\`ag functions on $[0,T]$ are bounded. Also, $\rho(\mu,\bar{\mu})\geq 0$ and $\rho(\mu,\bar{\mu})= 0$ $\Leftrightarrow$ $ \mathcal{W}_2 (\mu_t, \bar{\mu}_t)=0$ $\Leftrightarrow$ $\mu_t=\bar{\mu}_t$ for all $t\in[0,T]$, \textit{i.e.} $\mu=\bar{\mu}$. Further, $\rho(\mu, \bar{\mu})=\sup_{t\in[0,T]} \mathcal{W}_2 (\mu_t, \bar{\mu}_t)=\sup_{t\in[0,T]} \mathcal{W}_2 (\bar{\mu}_t, \mu_t)$ as $\mathcal{W}_2$ is a metric on $\mathcal{P}_2(\mathbb{R}^d)$ and thus $\rho(\mu,\bar{\mu})= \rho(\bar{\mu}, \mu)$.  Finally,  $\rho(\mu,\bar{\mu})=\sup_{t\in[0,T]} \mathcal{W}_2 (\mu_t, \bar{\mu}_t) \leq \sup_{t\in[0,T]} \mathcal{W}_2 (\mu_t, \hat{\mu}_t)+ \sup_{t\in[0,T]} \mathcal{W}_2 (\hat{\mu}_t, \bar{\mu}_t)= \rho(\mu,\hat{\mu})+\rho(\hat{\mu}, \bar{\mu})$ for any $\mu, \hat{\mu}$ and $\bar{\mu} \in \mathbb{D}([0,T]; \mathcal{P}_2(\mathbb{R}^d))$. 

Let $\{\mu^n\}_{n\in\mathbb{N}}$ be a Cauchy sequence in  $\mathbb{D}([0,T]; \mathcal{P}_2(\mathbb{R}^d))$. Thus for every $\epsilon>0$, there exists a $p\in\mathbb{N}$ such that $\rho(\mu^n,\mu^m)=\sup_{t\in[0,T]} \mathcal{W}_2 (\mu_t^n, \mu_t^m)<\epsilon/3$ for all $n,m\geq p$ which further implies 
\begin{equation*}
\mathcal{W}_2 (\mu_t^n, \mu_t^m)<\epsilon/3,  \quad  \forall \,  t\in[0,T]. 
\end{equation*}
Thus,   $\{\mu^n_t\}_{n\in\mathbb{N}}$ is a uniformly Cauchy sequence in the complete metric $\mathcal{P}_2(\mathbb{R}^d)$ and hence converges uniformly to (say) $\mu_t\in\mathcal{P}_2(\mathbb{R}^d)$. Indeed, for every $t\in[0,T]$, $\{\mu^n_t\}_{n\in\mathbb{N}}$ is a Cauchy sequence in a complete metric space and hence converges to (say) $\mu_t\in\mathcal{P}_2(\mathbb{R}^d)$. We now show that the convergence is uniform. Fixing $p\in\mathbb{N}$ and $n\geq p$, we have $\mathcal{W}_2(\mu_t^n, \mu_t^m)<\epsilon/3$ for all $m\geq p$ and $t\in[0,T]$. Thus, for above  $p\in\mathbb{N}$ and $n\geq p$, one has 
\begin{equation}\label{eq:uniform}
\mathcal{W}_2 (\mu_t, \mu_t^n) \leq \mathcal{W}_2 (\mu_t, \mu_t^m)+ \mathcal{W}_2 (\mu_t^m, \mu_t^n)  <  \mathcal{W}_2 (\mu_t, \mu_t^m)+ \epsilon/3 < 2\epsilon/3,    \quad \forall \, t\in[0,T]
\end{equation}
 and  $m\geq \max\{q,p\}$, where the last inequality is obtained using the fact that from the point-wise convergence, for above $\epsilon>0$ and $t\in[0,T]$, there exists a $q\in\mathbb{N}$ such that $ \mathcal{W}_2 (\mu_t, \mu_t^m)<\epsilon/3$ for all $m\geq q$. The dependence of $q$ on $t$ does not matter as the later calculations do not involve $m$. It is clear from equation \eqref{eq:uniform} that the convergence is uniform. 
 
 We now proceed to show that $\mu \in \mathbb{D}([0,T]; \mathcal{P}_2(\mathbb{R}^d))$.  Clearly, $\mu_t \in \mathcal{P}_2(\mathbb{R}^d)$ and we consider  points in $[0, T]$ at which $\mu$ is discontinuous. Let $t\in [0,T]$ be a point of discontinuity of $\mu$.   Denote by $\mu_{t-}^n$, the left limit of $\mu^n$ at $t$. Then, for given $m,n\geq n_0$, there exists $\delta^m, \delta^n>0$ such that 
 \begin{align*}
 \mathcal{W}_2 (\mu_s^m, \mu_{t-}^m) & < \epsilon/3, \quad s\in (t-\delta^m, t)
 \\
  \mathcal{W}_2 (\mu_s^n, \mu_{t-}^n) & < \epsilon/3, \quad s\in (t-\delta^n, t). 
\end{align*}  
Choose $\delta=\delta^m \wedge \delta^n$, then for all $s\in (t-\delta, t)$, 
\begin{align*}
\mathcal{W}_2 (\mu_{t-}^n,\mu_{t-}^m) \leq \mathcal{W}_2 (\mu_{t-}^n,\mu_{s}^n) + \mathcal{W}_2 (\mu_{s}^n,\mu_{s}^m) + \mathcal{W}_2 (\mu_{s}^m,\mu_{t-}^m) <\epsilon
\end{align*}
for all $n,m\geq n_0$. Thus, $\{\mu_{t-}^n\}_{n\in\mathbb{N}}$ is a Cauchy sequence in $\mathcal{P}_2(\mathbb{R}^d)$ and hence converges to $\mu_{t-}^{*} \in\mathcal{P}_2(\mathbb{R}^d)$. We now show that $\mu_{t-}=\mu_{t-}^{*}$. For this,  let  $\epsilon>0$ be any and choose $n\in \mathbb{N}$ sufficiently large so that 
$\mathcal{W}_2 (\mu_{s}, \mu_{s}^n)<\epsilon/3$ and $ \mathcal{W}_2 (\mu_{t-}^n, \mu_{t-}^{*})<\epsilon/3$. Corresponding to this $n$, there exists a $\delta>0$ such that for any $s\in(t-\delta, t)$, we have  $\mathcal{W}_2 (\mu_{s}^n, \mu_{t-}^n)<\epsilon/3$. Thus, corresponding to above $\epsilon>0$, we have found a $\delta>0$ such that for any $s\in(t-\delta, t)$, we have
\begin{align*}
\mathcal{W}_2 (\mu_{s}, \mu_{t-}^{*}) \leq \mathcal{W}_2 (\mu_{s}, \mu_{s}^n)+ \mathcal{W}_2 (\mu_{s}^n, \mu_{t-}^n)+ \mathcal{W}_2 (\mu_{t-}^n, \mu_{t-}^{*}) < \epsilon
\end{align*}
and hence $\mu_{t-}=\mu_{t-}^{*}$.  Now, let $\mu_{t+}$ denotes the right limit of $\mu$ at $t$ and $\mu_{t+}^n$ denotes the right limit of $\mu^n$ at $t$.  As above showing that $\{\mu_{t+}^n\}_{n\in\mathbb{N}}$ is a Cauchy sequence and hence converging point-wise to $\mu_{t+}^{*}$, for every $\epsilon>0$ we obtain a $\delta>0$ such that for any $s\in (t, t+\delta)$, we have  
\begin{align*}
\mathcal{W}_2 (\mu_{s}, \mu_{t+}^{*}) & \leq \mathcal{W}_2 (\mu_{s}, \mu_{s}^n)+ \mathcal{W}_2 (\mu_{s}^n, \mu_{t+}^n)+ \mathcal{W}_2 (\mu_{t+}^n, \mu_{t+}^{*}) < \epsilon
\end{align*}
for a sufficiently large $n$ which implies $\mu_{t+}=\mu_{t+}^{*}$. But observe that $\mu_{t+}^n=\mu_t^n$ for each $n$ and the uniform limit $\mu_t$ is also the point-wise limit of the sequence $\{\mu_t^n\}$. Thus, we have $\mu_{t+}=\mu_{t+}^{*}=\mu_t$ and hence $\mu\in \mathbb{D}([0,T]; \mathcal{P}_2(\mathbb{R}^d))$ showing that the space $\mathbb{D}([0,T]; \mathcal{P}_2(\mathbb{R}^d))$ is complete under the metric \eqref{eq:metric}.
\end{proof}

\subsection*{Authors' addresses}\hfill\\
\noindent Neelima, Department of Mathematics, Ramjas College, University of Delhi, Delhi, 110 007, India. \\{\tt 
neelima\_maths@ramjas.du.ac.in} \\

\noindent Sani Biswas, Department of Mathematics, Indian Institute of Technology Roorkee,  Roorkee, 247 667, India.  \\{\tt 
sbiswas2@ma.iitr.ac.in}\\

\noindent Chaman Kumar, Department of Mathematics, Indian Institute of Technology Roorkee,  Roorkee, 247 667, India.  \\{\tt 
chaman.kumar@ma.iitr.ac.in}\\

\noindent Gon\c calo dos Reis, School of Mathematics, University of Edinburgh, Peter Guthrie Tait Road, Edinburgh, EH9 3FD, United Kingdom, and Centro de Matem\'atica e Aplica\c c$\tilde{\text{o}}$es (CMA), FCT, UNL, Portugal. \\{\tt
  G.dosReis@ed.ac.uk}\\

\noindent Christoph Reisinger, Mathematical Institute, University of Oxford.
Andrew Wiles Building, Woodstock Road, Oxford, OX2 6GG, UK. \\{\tt christoph.reisinger@maths.ox.ac.uk}\\

\end{document}